\documentclass[a4paper,Haag duality]{mathscan}
\newtheorem{thm}{Theorem}[section] 
\newtheorem{pro}[thm]{Proposition}  
\newtheorem{cor}[thm]{Corollary}    
\newtheorem{lem}[thm]{Lemma}        
\theoremstyle{definition}           
\newtheorem{rem}[thm]{Remark}       

\newcommand{\NI}{\noindent}

\newcommand{\bea}{\begin{eqnarray}}
\newcommand{\eea}{\end{eqnarray}}

\def \b #1 {\bf #1}
\newcommand{\IR}{\mathbb{R}}

\newcommand{\IM}{\mathbb{M}}
\newcommand{\IE}{\mathbb{E}}
\newcommand{\IA}{\mathbb{A}}
\newcommand{\IC}{\mathbb{C}}
\newcommand{\ID}{\mathbb{D}}

\newcommand{\IT}{\mathbb{T}}
\newcommand{\IZ}{\mathbb{Z}}
\newcommand{\IP}{\mathbb{P}}
\newcommand{\IQ}{\mathbb{Q}}

\newcommand{\cal}{\mathcal}

\newcommand{\cla}{{\cal A}}
\newcommand{\clq}{{\cal Q}}
\newcommand{\clm}{{\cal M}}
\newcommand{\clz}{{\cal Z}}
\newcommand{\cli}{{\cal I}}
\newcommand{\cls}{{\cal S}}

\newcommand{\clf}{{\cal F}}

\newcommand{\clh}{{\cal H}}
\newcommand{\clp}{{\cal P}}

\newcommand{\clb}{{\cal B}}

\newcommand{\clj}{{\cal J}}
\newcommand{\cln}{{\cal N}}
\newcommand{\cld}{{\cal D}}
\newcommand{\cll}{{\cal L}}
\newcommand{\clc}{{\cal C}}

\newcommand{\raro}{\rightarrow}

\newcommand{\vsp}{\vskip 1em}

\newcommand{\ul}{\underline}

\def \qed {\hfill \vrule height6pt width 6pt depth 0pt}
\newcommand{\be}{\begin{equation}}
\newcommand{\ee}{\end{equation}}
\newcommand{\ben}{\begin{eqnarray*}}
\newcommand{\een}{\end{eqnarray*}}
\begin{document}

\title{Translation invariant state and its mean entropy-I}

\author{ Anilesh Mohari }

\address{ The Institute of Mathematical Sciences, CIT Campus, Taramani, Chennai-600113 }

\email{anilesh@imsc.res.in}



\thanks{ }

\begin{abstract}

Let $\IM =\otimes_{n \in \IZ}\!M^{(n)}(\IC)$ be the two sided infinite tensor product $C^*$-algebra of $d$ dimensional matrices $\!M^{(n)}(\IC)=\!M_d(\IC)$ over the field of complex numbers $\IC$ and $\omega$ be a translation invariant state of $\IM$. In this paper, we have proved that the mean entropy $s(\omega)$ and Connes-St\o rmer dynamical entropy $h_{CS}(\IM,\theta,\omega)$ of $\omega$ are equal. Furthermore, the mean entropy $s(\omega)$ is equal to the Kolmogorov-Sinai dynamical entropy $h_{KS}(\ID_{\omega},\theta,\omega)$ of $\omega$ when the state $\omega$ is restricted to a suitable translation invariant maximal abelian $C^*$ sub-algebra $\ID_{\omega}$ of $\IM$. Futhermore, a translation invariant factor state of $\IM$ is pure if and only if its mean entropy is zero. The last statement can be regarded as a non commutative extension of Rokhlin-Sinai positive entropy theorem for a non-pure factor state of $\IM$. 

\end{abstract}

\maketitle 

\section{ Introduction }

\vsp 
Any stationary Markov chain gives a translation invariant Markov state on a two sided classical spin chain. One celebrated result in ergodic theory states that two translation invariant Markov states with positive Kolmogorov-Sinai dynamical [CFS] entropies are isomorphic [Or1] if and only if their dynamical entropies are equal. The set of translation invariant Markov states is a closed subset of all translation invariant states of the two sided classical spin chain. Though the set of ergodic states is dense in the set of two sided translation invariant states, all ergodic states need not be Markov states. In other words, all symbolic dynamics are not associated with stationary Markov chains [Or2].

\vsp 
This classification program was primarily formulated [CFS] with a motivation to classify the symbolic dynamics [Si] associated with automorphic systems of classical Hamiltonian dynamics with Kolmogorov property. A complete classification of automorphic systems remains incomplete in the most general mathematical set up of classical dynamical systems. Nevertheless, these results have found profound use in classical ergodic theory in special situations of paramount importance and remain an active area of research over the last few decades due to its diverse applications in other areas of mathematics [CFS].    

\vsp 
In this paper, we have formulated a classification problem for  translation invariant states of the two sided quantum spin chain. Our prime motivation is to understand the properties of translation invariant factor and pure states that appear naturally as temperature and ground states respectively of a quantum mechanical Hamiltonian [BRII,Sim]. 
Though, in spirit our results are similar to the results in classical ergodic theory, our techniques and motivations are quite different. Before specializing in two sided quantum spin chain, we will recall briefly some standard terminologies and notations used in the theory of operator algebras [Sak, BRI,Ta2]. We will also recall the classical situation of our present problem in some details in the following text and compare with its quantum counter part. 

\vsp 
Let $\clb$ be a $C^*$-algebra over the field of complex numbers $\IC$. A linear functional $\omega:\clb \raro \IC$ is called a {\it state} on $\clb$ if it is {\it positive} i.e. $\omega(x^*x) \ge 0$ for all $x \in \clb$ and {\it unital} i.e. $\omega(I)=1$, where $I$ is the unit element of $\clb$. The convex set $\clb^*_{+,1}$ of states on $\clb$ is compact in weak$^*$ topology of the dual Banach space $\clb^*$ of $\clb$. A state $\omega$ of $\clb$ is called {\it pure} if the state can not be expressed as a convex combination of two different states i.e. if $\omega$ is an extremal element in $\clb^*_{+,1}$.  

\vsp 
Let $\clb_1,\clb_2$ be two unital $C^*$-algebras [BRI,Ta2] over the field of complex numbers $\IC$. A unital linear map $\pi: \clb_1 \raro \clb_2$ is called {\it homomorphism} if 
\be 
\pi(x)^*=\pi(x^*)\;\;\mbox{and}\;\; \pi(xy)=\pi(x)\pi(y)
\ee 
for all $x,y \in \clb_1$. An injective homomorphism $\beta:\clb_1 \raro \clb_2$ is called 
{\it endomorphism}. For an unital $C^*$ algebra $\clb$, a linear bijective map $\theta: \clb \raro \clb$ is called {\it automorphism} if the map is a homomorphism. 

\vsp 
For an unital $C^*$ algebra $\clb$, a state $\omega$ of $\clb$ is called {\it invariant} 
for an automorphism $\theta: \clb \raro \clb$ if $\omega = \omega \theta$. A triplet $(\clb,\theta,\omega)$ is called a unital $C^*$-{\it dynamical system} if $\clb$ is a unital $C^*$-algebra and $\theta:\clb \raro \clb$ is an automorphism preserving a state $\omega$ of $\clb$. For a given automorphism $\theta$ on a unital $C^*$-algebra $\clb$, the set of invariant states 
$$\cls^{\theta}=\{ \omega \in \clb^*_{+,1}: \omega = \omega  \theta \}$$ 
of $\theta$ is a non empty compact set in the weak$^*$ topology of the dual Banach space $\clb^*$. An extremal element in the convex set $\cls^{\theta}$ is called {\it ergodic } state for $\theta$. An invariant state $\omega$ of $\theta$ is ergodic if and only if 
\be 
{1 \over 2n+1} \sum_{-n \le k \le n} \omega(y\theta^k(x)z) \raro \omega(yz)\omega(x)
\ee
as $n \raro \infty$ for all $x,y,z \in \clb$. 

\vsp
Let $(\clh_{\omega},\pi_{\omega},\zeta_{\omega})$ be the Gelfand-Naimark-Segal (GNS) space associated with a state $\omega$ of $\clb$, where $\pi_{\omega}:\clb \raro \clb(\clh_{\omega})$ is a $*$-representation of $\clb$ and $\zeta_{\omega}$ is the cyclic vector for $\pi(\cla)$ in $\clh_{\omega}$ such that 
$$\omega(x)= \langle \zeta_{\omega},\pi_{\omega}(x)\zeta_{\omega} \rangle$$
Let $\pi_{\omega}(\clb)'$ be the {\it commutant} of $\pi_{\omega}(\clb)$ i.e. $\pi_{\omega}(\clb)'=\{x \in \clb(\clh_{\omega}): xy=yx,\;\forall y \in \pi_{\omega}(\clb) \}$ and $\pi_{\omega}(\clb)''$ be the {\it double commutant} of $\pi_{\omega}(\clb)$ i.e.
$\pi_{\omega}(\clb)''= \{ x \in \clb(\clh_{\omega}): xy=yx\;\;\forall y \in \pi_{\omega}(\clb)' \}$. By a celebrated theorem of von-Neumann, $\pi_{\omega}(\clb)''$ is the weak$^*$ completion of $\pi_{\omega}(\clb)$ in $\clb(\clh_{\omega})$ and it admits a pre-dual Banach space. The pre-dual Banach space is often called {\it normal functional} on $\pi_{\omega}(\clb)''$. A state $\omega$ of $\clb$ is pure if and only if $\pi_{\omega}(\clb)''=\clb(\clh_{\omega})$, the algebra of all bounded operators on $\clh_{\omega}$.

\vsp 
Given a unital $C^*$-dynamical system $(\clb,\theta,\omega)$, we have a unitary operator $S_{\omega}:\clh_{\omega} \raro \clh_{\omega}$ extending the following inner product preserving map
\be 
S_{\omega} \pi_{\omega}(x)\zeta_{\omega}=\pi_{\omega}(\theta(x))\zeta_{\omega},\;x \in \clb
\ee 
and an automorphism $\Theta_{\omega}:\pi_{\omega}(\clb)'' \raro \pi_{\omega}(\clb)''$, defined by 
\be 
\Theta_{\omega}(X)=S_{\omega}XS^*_{\omega},\;X \in \pi_{\omega}(\clb)''
\ee 
Thus we have  
$$\Theta_{\omega}(\pi_{\omega}(x)) = \pi_{\omega}(\theta(x)),\;x \in \clb$$
Furthermore, $\omega$ is ergodic for $\theta$ if and only if 
$$\{ f: S_{\omega}f=f,\; f \in \clh_{\omega} \} = \{z \zeta_{\omega}: z \in \IC \}$$ 

\vsp 
An invariant state $\omega$ for $\theta$ is called {\it strongly mixing} if
\be 
\omega(x\theta^n(y)) \raro \omega(x)\omega(y)
\ee
as $|n| \raro \infty $ for all $x,y \in \clb$. A strongly mixing state is obviously ergodic, however the converse is false. A simple application of Riemann-Lebesgue lemma says that absolute continuous spectrum of $S_{\omega}$ in the orthogonal complement of invariant vector in $\clh_{\omega}$ of $S_{\omega}$ is sufficient for strong mixing property [Pa]. It is not known yet, whether converse is true. In other words, no simple criteria on $\omega$ is known yet for strongly mixing. 

\vsp 
In case, $\clb$ is a unital commutative $C^*$-algebra, then $\clb$ is isomorphic to $C(X)$, where $C(X)$ is the algebra of complex valued continuous functions on a compact Hausdorff space $X$. An automorphism $\theta$ on $C(X)$ determines a unique one to one and onto continuous point map $\gamma_{\theta}:X \raro X$ such that  $\theta(f)=f \circ \gamma_{\theta}$, for all $f \in C(X)$. A state $\omega$ on $\clb \equiv C(X)$ is determined uniquely by a regular probability measure $\mu_{\omega}$ on $X$ by $\omega(f)=\int f d \mu_{\omega}$ and its associated GNS space $\clh_{\omega}=L^2(X,\mu_{\omega})$ with representation $\pi_{\omega}(h)f=hf$ for all $f \in C(X)$ with $\pi_{\omega}(X)''=L^{\infty}(X,\clf_X,\mu_{\omega})$, where $\clf_X$ is the Borel $\sigma$-field of $X$. Thus ergodic and strongly mixing properties introduced in the non commutative framework of $C^*$ algebras are in harmony with its classical counter parts and coincide once one restricts them to commutative $C^*$-algebras. 

\vsp 
Two unital $C^*$-dynamical systems $(\cla_1,\theta_1,\omega_1)$ and $(\cla_2,\theta_2,\omega_2)$ are said to be {\it isomorphic } if there exists a $C^*$ automorphism $\alpha:\cla_1 \raro \cla_2$ such that 
\be 
\theta_2  \alpha = \alpha   \theta_1,\;\;\omega_2  \alpha = \omega_1
\ee 
on $\cla_1$. It is clear that ergodic and strong mixing properties remain covariant with respect to the isomorphism. One of the central problem in classical ergodic theory is aimed to classify classical dynamical system of automorphisms upto the isomorphism. 

\vsp 
We say two $C^*$-dynamical systems $(\cla_1,\theta_1,\omega_1)$ and $(\cla_2,\theta_2,\omega_2)$ are {\it weak$^*$ isomorphic} if there exists unital completely positive maps $\tau:\cla_1 \raro \pi_{\omega_2}(\cla_2)''$ and 
$\eta:\cla_2 \raro \pi_{\omega_1}(\cla_1)''$ such that

\vsp 
\NI (a) $\tau$ and $\eta$ are limit points of inter-twinning automorphisms between the two dynmics i.e. $\tau( \eta )$ is a limit point of a sequence $\alpha_n:\cla_1 \raro \cla_2$ of automorphisms that satisfies $\alpha_n \theta_1 = \theta_2 \alpha_n,\;n \ge 1$ in bounded-weak topology of Arveson.    

\vsp 
\NI (b) $$\omega_2 \tau = \omega_1,\;\;\;\theta_2 \tau = \tau \theta_1,\;\;\mbox{on}\;\; \cla_1$$ 

\vsp 
\NI (c) $$\omega_1 \eta = \omega_2,\;\;\;\theta_1 \eta = \eta \theta_2,\;\;\mbox{on}\;\; \cla_2$$

Thus any two isomorphic dynamics are weak$^*$ isomorphic. Converse statement is false even for two dynamics in commutative $C^*$-algebras. Ornstein's weak$^*$ isomorphism between two Bernoiulli shifts with equal Kolmogorov-Sinai dynamical entropies 
is given by a Borel isomorphism which need not be an isomorphism between two $C^*$-algebras i.e. need not be a continuous map.   
However, definition indroduced here for weak$^*$ isomorphism is inspired by the following simple observation that Connes-St\o rmer dynamical entropy is a weak$^*$ invariance. Hypothesis (a) is crucial as otherwise any two dynamics satisfy (b) and (c) 
with $\tau=\omega_1$ and $\eta=\omega_2$.     

\vsp 
We may replace automorphisms $(\theta_k:k=1,2)$ by unital endomorphisms in these definitions to include more general $C^*$-dynamical systems of $*$-endomorphisms. 

\vsp
A state $\omega$ of $\clb$ is called {\it factor} if $\pi_{\omega}(\clb)''$ is a factor i.e. if the centre $\pi_{\omega}(\clb)'' \bigcap \pi_{\omega}(\clb)'$ is equal to $\{ z \pi_{\omega}(I):z \in \IC \}$. An automorphism of $\clb$ takes a factor state to another factor state. Apart from ergodic and strong mixing properties, factor property is also an invariant under the isomorphism of two dynamics. For a commutative $C^*$-algebra $\clb$, the centre of $\pi_{\omega}(\clb)''$ is itself and thus $\omega$ is a factor state if and only if $\clh_{\omega}$ is one dimensional. In other words, $\omega$ is a Dirac measure on a point of $X$, where $\clb \equiv C(X)$. Such a state plays no interest in classical dynamical system since invariance property for the automorphism ensures that $\gamma_{\theta}$ has a fixed point in $X$.   

\vsp 
We introduce now one more invariant in the general mathematical set up of $C^*$-dynamical systems. For a family of $C^*$ sub-algebras $(\clb_i:i \in \cli)$ of $\clb$, we use the notation $\vee_{i \in \cli} \clb_i $ for the $C^*$ algebra generated by the family $(\clb_i: i \in \cli)$. For a $C^*$ sub-algebra $\clb_0$ of $\clb$, we set $\clb'_0 = \{ x \in \clb: xy=yx,\; \forall y \in \clb_0 \}$ and $\clb_0''= \{ x \in \clb: xy=yx,\; \forall y \in \clb_0' \}$. A $C^*$-dynamical system $(\clb,\theta,\omega)$ is said to have {\it weak Kolmogorov property} if there exists a $C^*$ sub-algebra $\clb_0$ of $\clb$ such that the following hold:

\vsp 
\NI (a) $\clb_0 \vee \clb_0' = \clb,\;\clb_0''=\clb_0,\;\;\theta^{-1}(\clb_0) \subseteq \clb_0$;

\vsp 
\NI (b) $\vee_{n \in \IZ} \theta^n(\clb_0) = \clb$; 

\vsp
\NI (c) $\bigcap_{n \in \IZ} \theta^{-n}(\clb_0)=\{z I: z \in \IC \}.$ 

\vsp 
\NI (d) For each $n \in \IZ$, let $F^{\omega}_{n]}(\clb_0)$ be the projection $[\pi_{\omega}(\theta^n(\clb_0))\zeta_{\omega}]$ in $\clh_{\omega}$. Then  
$$F^{\omega}_{n]}(\clb_0)\;\;\downarrow \;\;|\zeta_{\omega} \rangle \langle \zeta_{\omega}|$$
in strong operator topology as $n \raro -\infty$. In short, such an element $\omega \in \cls^{\theta}(\clb)$ is called {\it weak Kolmogorov state}. When a fixed state $\omega \in \cls^{\theta}(\clb)$ and a $C^*$-subalgebra 
$\clb_0$ are under consideration, we often omit the superscript $\omega$ and $\clb_0$ in the notation of $F^{\omega}_{n]}(\clb_0)$ and simply denote by $F_{n]}$ for each $n \in \IZ$. 

\vsp 
The relations (a)-(c) are state independent. In particular, for any translation invariant state $\omega$ of $\clb$, properties (a) and (b) ensure that 
$$F_{n]} \uparrow I_{\clh_{\omega}}$$ as $n \uparrow \infty$. Furthermore, $F_{n]} \downarrow F_{-\infty]}$ as $n \downarrow -\infty$ in strong operator topology for some projection $F_{-\infty]} \ge |\zeta_{\omega} \rangle \langle \zeta_{\omega}|$. However, in general $F_{-\infty]}$ need not be equal to $|\zeta_{\omega} \rangle \langle \zeta_{\omega}|$ even when (c) is true. Thus the property (d) is crucial to determine weak Kolmogorov property of the state $\omega$. It is clear that weak Kolmogorov property is an invariant for the dynamics $(\clb,\theta,\omega)$. 

\vsp 
For a weak Kolmogorov state $\omega$ and any $x,y \in \clb_0$, we also have 
$$|\omega(x\theta^n(y)|$$
$$=|<x^*\zeta_{\omega}, F_{n]} \theta^n(y) \zeta_{\omega}>|$$
$$\le ||F_{n]}x^*\zeta_{\omega}||\;||\theta^n(y)\zeta_{\omega}||$$
$$\le ||F_{n]}x^*\zeta_{\omega} || \; ||y||$$ 
$$ \raro 0$$
as $n \raro -\infty$ once $\omega(x)=0$. Now we use the linear property of the map $\theta$ to prove strong mixing property (5) for all $x,y \in \clb_0$. Going along the same line of the proof, we also verify (5) for all $x,y \in \theta^{-m}(\clb_0),\; m \ge 1$. Since
$\cup_{m \ge 1} \theta^{-m}(\clb_0)$ is norm dense in $\clb$, a standard density argument proves (5) for all $x,y \in \clb$. In other words, weak Kolmogorov states are strongly mixing. By the argument used above, for each $x \in \clb$ we have
\be 
\mbox{sup}_{\{y:||y|| \le 1 \}} |\omega(x \theta^n(y))-\omega(x)\omega(y)| \raro 0
\ee
as $n \raro -\infty$.       

\vsp 
For each $n \in \IZ$, we also set projections $E_{n]} \ge F_{n]}$ defined by 
$E_{n]} = [\pi_{\omega}(\theta^n(\clb_0'))'\zeta_{\omega}]$ and $(\clb,\theta,\omega)$ is called {\it Kolmogorov} if (a), (b), (c) and the following stronger property 

\vsp 
\NI (e) $E_{n]} \downarrow |\zeta_{\omega}\rangle \langle \zeta_{\omega}|$ as $n \raro -\infty$ holds. 

for some choice for $\clb_0$. Thus Kolmogorov property as defined here is as well an invariance of the dynamics unlike our previous definition [Mo1] (see section 6 for details).

\vsp 
Since $E_{n]}$ is the support projection of the state $\omega$ in $\pi_{\omega}(\theta^n(\clb'_0))''$ and by (c) we have $\vee \theta^n(\clb_0') = \clb$, for a Kolmogorov state we have then $|\zeta_{\omega} \rangle \langle \zeta_{\omega}| \in \pi_{\omega}(\clb)''$. So in such a case $\omega$ is a pure state.

\vsp 
A $C^*$ dynamical system $(\clb,\theta,\omega)$ is said to have {\it backward weak Kolmogorov property } if $(\clb,\theta^{-1},\omega)$ admits weak Kolmogorov property. Clearly the backward weak Kolmogorov property is also an invariant. A theorem of Rokhlin-Sinai [Pa] says that the Kolmogorov property for a classical dynamical systems is Kolmogorov if and only if its Kolmogorov-Sinai dynamical entropy is strictly positive. Since Kolmogorov-Sinai dynamical entropies are equal for $(\clb,\theta,\omega)$ and $(\clb,\theta^{-1},\omega)$, Kolmogorov property in classical dynamics is a time-reversible invariant. However, in the non commutative framework, such a time-reversible property for a weak Kolmogorov state is not clear even though relations (a)-(c) hold for $(\clb,\theta^{-1},\omega)$ with $\theta$ and $\clb_0'$ replacing $\theta^{-1}$ and $\clb_0$ respectively. Similar notion we also set for {\it backward Kolmorov property } for the dynamics.  

\vsp 
Let $\cli$ be a a subspace of a $C^*$-algebra $\clb$. It is called {\it an ideal} or {\it a two sided ideal } of $\clb$ if 

\vsp 
\NI (a) $\cli$ is closed under conjugation i.e. $x^* \in \cli$ if $x \in \cli$;

\vsp
\NI (b) $xy,yx \in \cli$ for all $x \in \cli$ and $y \in \clb$. 

\vsp 
A $C^*$ algebra $\clb$ is called {\it simple} if $\clb$ has no proper ideal i.e. other then $\clb$ or $\{0\}$. 

\vsp 
For a simple $C^*$ algebra $\clb$, any non trivial $*$-homomorphism $\beta:\clb \raro \clb$ is injective since the null space $\cln=\{x \in \clb_1: \beta(x)=0 \}$ is a two sided ideal. Thus $x \raro ||\beta(x)||$ is a $C^*$ norm on $\clb$. Since $C^*$ norm is unique on a $C^*$ algebra with a given involution by Gelfand spectral theorem [BRI], we get 
$$||\beta(x)||=||x||$$ 
for all $x \in \clb$. In particular, a homomorphism $\beta:\clb \raro \clb$ is an automorphism for a simple $C^*$ algebra $\clb$ if the homomorphism $\beta$ is onto.     

\vsp 
In this paper, we will investigate this abstract notion of weak Kolmogorov property in more details in the context of one lattice dimension two sided quantum spin chains studied in a series of papers [Mo1],[Mo2] and [Mo3] a notion for Kolmogorov property. 

\vsp 
In the following, we describe $C^*$ algebraic set up valid for quantum spin chain 
[BRII,Ru] and find its relation to classical spin chain [Pa] in details. Let $\IM=\otimes_{n \in \IZ} \!M^{(n)}_d(\IC)$ be the $C^*$ -completion of the infinite tensor product of the algebra $\!M_d(\IC)$ of $d$ by $d$ matrices over the field of complex numbers. Let $Q$ be a matrix in $\!M_d(\IC)$. By $Q^{(n)}$ we denote the element $...\otimes I_d \otimes I_d \otimes I_d \otimes Q \otimes I_d \otimes I_d \otimes I_d \otimes ... $, where $Q$ appears in the $n$-th component in the tensor product and $I_d$ is the identity matrix of $\!M_d(\IC)$. Given a subset $\Lambda$ of $\IZ$, $\IM_{\Lambda}$ is defined to be the $C^*$-sub-algebra of $\IM$ generated by elements $Q^{(n)}$ with $Q \in \!M_d(\IC)$, $n \in \Lambda$. The $C^*$-algebra $\IM$ being the inductive limit of increasing matrix algebras, it is a simple $C^*$-algebra [ChE],[SS]. 

\vsp 
We also set $$\IM_{loc}= \bigcup_{\Lambda:|\Lambda| < \infty } \IM_{\Lambda},$$
where $|\Lambda|$ is the cardinality of $\Lambda$. An automorphism $\beta$ on $\IM$ is called {\it local } if $\beta(\IM_{loc}) \subseteq \IM_{loc}$. Right translation $\theta$ is a local automorphism of $\IM$ defined by $\theta(Q^{(n)})=Q^{(n+1)}$. We also will use simplified notations $\IM_{R}=\IM_{[1,\infty)}$ and $\IM_{L}=\IM_{(-\infty,0]}$. Thus the restriction of $\theta$ ( $\theta^{-1}$ ), $\theta_R$ (and $\theta_L$) is a unital $*$-endomorphisms on $\IM_R$ ($\IM_L$). The restriction of a state $\omega$ to $\IM_{\Lambda}$ is denoted by $\omega_{\Lambda}$. We also set $\omega_{R}=\omega_{[1,\infty)}$ and $\omega_{L}=\omega_{(-\infty,0]}$.

\vsp 
We say a state $\omega$ of $\IM$ is {\it translation invariant} if $\omega  \theta = \omega$ on $\IM$. In this paper, we are interested to deal with $C^*$ dynamical systems $(\IM,\theta,\omega)$ and its isomorphism problem. 
Most results have natural generalisation to higher lattice dimensions. We will get back to this point an the end of last section. 

\vsp 
We begin with a simple technical result proved at the end of section 2. 

\vsp 
\begin{pro} 
Let $\IM_0$ be a $C^*$ sub-algebra of $\IM$ such that 

\vsp 
\NI (a) $M_{0}'' = M_{0},\;$ $\IM_{0} \vee \IM'_{0} = \IM$ and $\theta^{-1}(\IM_{[0}) \subseteq \IM_{0}$;

\vsp 
\NI (b) $\vee_{n \in \IZ} \theta^n(\IM_{0}) = \IM$;

\vsp 
\NI (c) $\bigcap_{n \in \IZ} \theta^n(\IM_{0})=\IC$.

\vsp 
Then there exists an automorphism $\alpha$ on $\IM$ commuting with $\theta$ with $\alpha(\IM_L)= \IM_{0}$. 
\end{pro}

\vsp 
For a translation invariant state $\omega$ on $\IM$ and the GNS space $(\clh_{\omega},\pi_{\omega},\zeta_{\omega})$ of $(\IM,\omega)$, we set a sequence of increasing projections $F^{\omega}_{n]}(\IM_L),\;n \in \IZ$ defined by 
\be 
F^{\omega}_{n]}(\IM_L) = [\pi_{\omega}(\theta^n(\IM_L))''\zeta_{\omega}]
\ee 
and 
$$F^{\omega}_{-\infty]}(\IM_L) = \mbox{lim}_{n \raro -\infty}F^{\omega}_{n]}(\IM_L)$$ 

\vsp 
For a given $C^*$-dynamical system $(\IM,\IM_{0},\theta,\omega)$ with weak Kolmogorov property, Proposition 1.1 says that there exists an isomorphic $C^*$-dynamical system $(\IM,\IM_L,\theta^{-1},\omega')$ with weak Kolmogorov property, where $\omega' = \omega \alpha$ for an automorphism $\alpha$ on $\IM$ commuting with $\theta$ and 
$$F^{\omega'}_{-\infty]}(\IM_L) = |\zeta_{\omega'}\rangle \langle \zeta_{\omega'}|$$
Similar statement holds for Kolmogorov property. 

\vsp 
Simplest example of a weak Kolmogorov state is given by any infinite tensor product state ( in particular, the unique normalized trace $\omega_0 = \otimes_{n \in \IZ} tr^{(n)}_0$ ) 
$\omega = \otimes_{n \in \IZ} \omega^{(n)}$, where $\omega^{(n)}=\omega^{(n+1)}$ for all $n \in \IZ$. In such a case, we have   
\be 
F_{n]}\pi_{\omega}(x)F_{n]}=\omega(x)F_{n]}
\ee 
for all $x \in \IM_{[n+1,\infty)}$, where $F_{n]}=[\pi_{\omega}(\theta^n(\IM_L))\zeta_{\omega}]$. 
This clearly shows $\omega$ is weak Kolmogorov. In particular, any two such infinite tensor product states need not give isomorphic dynamics since the class of such states includes both pure and non pure factor states. Thus for isomorphism problem, we need to deal with the class of Kolmogorov states with some additional invariants. 

\vsp 
The argument that we have used to prove strongly mixing property of weak Kolmogorov state, as well gives a proof for factor property of Kolmogorov state $\omega$ by Theorem 2.5 in [Pow]. In this paper, our analysis aims to deal with isomorphism problem for the classes of translation invariant states with additional properties. At this point, we warn attentive readers that  different notions of Kolmogorov dynamics are also studied in [NS] and as well as in [Mo1,Mo2,Mo3].  
       
\vsp 
We also set a family of increasing projections $E_{n]}=[\pi_{\omega}(\theta^n(\IM_R))'\zeta_{\omega}],\; n \in \IZ$ i.e. $E_{n]}$ is the support projection of $\omega$ in $\pi_{\omega}(\theta^n(\IM_R))''$. Thus we have $E_{n]} \le E_{n+1]}$ and $E_{n]} \uparrow I_{\clh_{\omega}}$ as $n \uparrow \infty$. Let $E_{n]} \downarrow E_{-\infty]}$ as $n \downarrow -\infty$ for some projection $E_{-\infty]} \ge |\zeta_{\omega}\rangle\langle\zeta_{\omega}|$. It is clear that 
\be 
F_{n]} \le E_{n]}
\ee 
for all $n \in \IZ$. The inequality (10) is strict if $\omega$ is a non pure factor state [Mo2]. 

\vsp 
By a theorem in [Mo2], a translation invariant pure state $\omega$ of $\IM$ admits {\it Haag duality} property i.e. 
\be 
\pi_{\omega}(\IM_L)''= \pi_{\omega}(\IM_R)'
\ee
Though $\IM_R'=\IM_L$ as $C^*$ sub-algebras of $\IM$, it is a non trivial fact that the equality (11) holds for a translation invariant factor state $\omega$ if and only if the state $\omega$ is pure. In such a case, for each $n \in \IZ$, we have $F_{n]} = E_{n]}$. Thus a translation invariant pure state $\omega$ admits weak Kolmogorov property if and only if $\omega$ admits
Kolmogorov property. Thus our definition of weak Kolmogorov property for a translation invariant state is an extension of Kolmogorov property for a translation invariant pure state studied in [Mo1,Mo2,Mo3]. 

\vsp 
However, for a pure $\omega$, though $F_{n]}=E_{n]}$ for all $n \in \IZ$ [Mo2], the projection $F^{\omega}_{-\infty]}(\IM_L)$ may not be equal to $|\zeta_{\omega}\rangle\langle\zeta_{\omega}|$. We have included an example of a pure state $\omega$ that fails to have the equality $F^{\omega}_{-\infty]}(\IM_L) = |\zeta_{\omega}\rangle \langle\zeta_{\omega}|$ in the appendix of [Mo4].  

\vsp 
Theorem 2.6 in [Mo1] says that the state $\omega$ is pure Kolmogorov if $\omega_R$ ($\omega_L$) is a type-I factor state ( i.e. $\pi_{\omega}(\IM_R)''$ is isomorphic to the algebra of bounded operators on a Hilbert space ). At this stage, it is not clear whether the converse statement is also true for a pure Kolmogorov state.    

\vsp 
Given a translation invariant state $\omega$ of $\IM$, one has two important numbers: 

\vsp 
\NI (a) Mean entropy $s(\omega,\IM_{loc})=\mbox{limit}_{\Lambda_n \uparrow \IZ}{1 \over |\Lambda_n|}S_{\omega_{\Lambda_n}}$, where $S_{\omega_{\Lambda_n}}=-tr_{\Lambda_n}(\rho^{\omega}_{\Lambda_n}ln \rho^{\omega}_{\Lambda_n})$ is the von-Neumann entropy of the density matrix $\rho^{\omega}_{\Lambda_n}$ associated with the state $\omega_{\Lambda_n}(x)=tr_{\Lambda}(x\rho^{\omega}_{\Lambda_n})$ and $tr_{\Lambda_n}$ is the normalised trace on $\IM_{\Lambda_n}$ i.e. von Neumann entropy of the restricted state $\omega$ to the local $C^*$ subalgebra $\IM_{\Lambda_n}$, where $\Lambda_n=\{-n \le k \le n \}$ or more generally a sequence of finite subsets $\Lambda_n$ of $\IZ$ such that $\Lambda_n \uparrow \IZ$ in the sense of Van Hove [Section 6.2.4 in BR2]. A considerable literature devoted [OP,NS] in the last few decades to realise $s(\omega,\IM_{loc})$ as an invariant for the translation dynamics $(\IM,\theta,\omega)$ i.e. whether $s(\omega,\IM_{loc})$ can be realized intrinsically as a dynamical entropy of $(\IM,\theta,\omega)$ for a translation invariant state $\omega$ of $\IM$.    

\vsp 
\NI (b) Connes-St\o rmer dynamical entropy: $h_{CS}(\IM,\theta,\omega)$ [CS,CNT,OP,St\o2, NS] which is a close candidate for such an invariant for the translation dynamics $(\IM,\theta,\omega)$. It is known that $0 \le h_{CS}(\IM,\theta,\omega) \le s(\omega,\IM_{loc})$. In case $\omega$ is a product state then it is known that $h_{CS}(\IM,\theta,\omega)=s(\omega,\IM_{loc})$. Same holds for more general states proved in [Par], [Pe] and [GoB]. Results in particular show that there is no automorphism $\beta$ on $\IM$ such that $\beta \theta = \theta^2 \beta$ [CS]. It is also known that $h_{CS}(\IM,\theta,\omega)=0$ if $\omega$ is pure. However, no translation invariant state $\omega$ of $\IM$ is known in the literature for which $h_{CS}(\IM,\theta,\omega) < s(\omega,\IM_{loc})$. Main obstruction that comes from the fact an automorphism on $\IM$ need not preserve a given local sub-algebra $\IM_{loc}$ of $\IM$.  

\vsp 
In the section 6 of this paper, one of our main results gives a proof for the equality $h_{CS}(\IM,\theta,\omega)=s(\omega,\IM_{loc})$ for any translation invariant state $\omega$ of $\IM$ and thus we can write $s(\omega,\IM_{loc})=s(\omega)$, as it is now independent of the local algebras we choose. In particular, one easy consequence of this result says now that $s(\omega)=0$ for any translation invariant pure state of $\IM$. This raises a valid question: does $s(\omega)=0$ implies purity of $\omega$? We will show how this problem is related to classical problem of zero entropy of Kolmogorov-Sinai dynamical entropy and purity of the state of the classical spin chain. In particular, our analysis proves that the converse of the statement is true if the state is a factor state of $\IM$.

\vsp 
At this stage, we may raise few more non trivial questions: 

\vsp 
\NI (a) Does the mean entropy give a complete characterisation for the class of infinite tensor product states? In other words, does the equality in mean entropy of two infinite tensor product states give weak$^*$-isomorphic translation dynamics of $\IM$? 

\vsp 
If answer for (a) is yes, then we may raise a more general question on weak$^*$ isomorphism problem for translation invariant factor states of $\IM$: 

\vsp 
\NI (b) Does the mean entropy give a complete characterisation for the class of translation invariant factor states? In other words, given a factor state $\omega \in \cls_{\theta}(\IM)$, do we have a pair of completely positive map $\tau:\IM \raro \pi_{\omega}(\IM)''$ and $\eta:\IM \raro \pi_{\omega_{\rho}}(\IM)''$ commuting with $\theta$ such that the translation invariant state $\omega_{\rho} = \hat{\omega} \circ \eta $ is an infinite tensor product state of $\IM$?  

\vsp 
Affirmative answers to both (a) and (b) in particular says that the mean entropy is a complete invariant for the class of factor states. In this paper, we include a positive answer to question (a) in Theorem 7.2 provided we extend our notion of equivalence of dynamical system which proves that the mean entropy is a complete weak$^*$ invariant of translation dynamics for 
the class of infinite tensor product faithful states but need not be isomorphic in general as expected. We also address the problem (b) under some additional restrictions on the class of factor states on $\IM$. In particular, we prove that two translation invariant pure states give isomorphic dynamics. As a simple consequence of this result, we prove that the mean entropy of a translation invariant pure state is zero. However, the converse statement is false unless for a factor state.  

\vsp 
Thus it settles a long standing conjecture [NS] that positive mean entropy makes a state non pure. In other words, a non pure translation invariant factor state is having positive mean entropy.  

\section{Norm one projections} 

\vsp 
Let $\clm$ be a von-Neumann algebra acting on a Hilbert space $\clh$. A unit vector $\zeta$ is called cyclic for $\clm$ in $\clh$ if $[\clm \zeta]=\clh$. It is called separating for $\clm$ if $x\zeta=0$ for some $x \in \clm$ holds if and only if $x=0$.
Let $\clm'$ be the commutant of $\clm$, i.e. $\clm' = \{ x \in \clb(\clh): xy=yx \}$, where
$\clb(\clh)$ is the algebra of all bounded linear operators on $\clh$. An unit vector $\zeta$ is cyclic if and only if $\zeta$ is separating for $\clm'$.  

\vsp 
The closure of the closable operator $S_0:a\zeta \raro a^*\zeta,\;a \in \clm, S$ possesses a polar decomposition $S=\clj \Delta^{1/2}$, where $\clj$ is an anti-unitary and $\Delta$ is a non-negative self-adjoint operator on $\clh$. Tomita's [BRI] theorem says that 
\be 
\Delta^{it} \clm \Delta^{-it}=\clm,\;t \in \IR\;\mbox{and} \clj \clm \clj=\clm'
\ee 
We define the modular automorphism group
$\sigma=(\sigma_t,\;t \in \IT )$ on $\clm$
by
$$\sigma_t(a)=\Delta^{it}a\Delta^{-it}$$ which satisfies the modular relation
$$\omega(a\sigma_{-{i \over 2}}(b))=\omega(\sigma_{{i \over 2}}(b)a)$$
for any two analytic elements $a,b$ for the group of automorphisms $(\sigma_t)$. 
A more useful modular relation used frequently in this paper is given by 
\be 
\omega(\sigma_{-{i \over 2}}(a^*)^* \sigma_{-{i \over 2}}(b^*))=\omega(b^*a)
\ee 
which shows that $\clj a\zeta = \sigma_{-{i \over 2}}(a^*)\zeta$ for an analytic element $a$ for the automorphism group $(\sigma_t)$. The anti-unitary operator $\clj$ and the group of automorphism $\sigma=(\sigma_t,\;t \in \IR)$ are called {\it Tomita's conjugate operator} and {\it modular automorphisms } of the normal vector state $\omega_{\zeta}:x \raro \langle \zeta,x \zeta \rangle$ on $\clm$ respectively. 

\vsp 
A faithful normal state $\omega$ of $\clm$ is called stationary for a group $\alpha=(\alpha_t:\; t \in \IR)$ of automorphisms on $\clm$ if $\omega = \omega \alpha_t$ for all $t \in \IR$. A stationary state $\omega_{\beta}$ for $(\alpha_t)$ is called $\beta$-KMS state ($\beta > 0$) if there exists a function $z \raro f_{a,b}(z)$, analytic on the open strip $0 < Im(z) < \beta$, bounded continuous on the closed strip $0 \le Im(z) \le \beta$ with boundary condition 
\be 
f_{a,b}(t)=\omega_{\beta}(\alpha_t(a)b),\;\;f_{a,b}(t+i\beta)=\omega_{\beta}(\alpha_t(b)a)
\ee
for all $a,b \in \clm$. The faithful normal state $\omega_{\zeta}$ given by the cyclic and separating vector $\zeta$ is a $1 \over 2$-KMS state for the modular automorphisms group $\sigma=(\sigma_t)$. One celebrated theorem of M. Takesaki [Ta2] says that the converse statement is also true: If the normal state $\omega_{\zeta}$ given by cyclic and separating vector $\zeta$ is a $1 \over 2$-KMS state for a group $\alpha=(\alpha_t)$ of automorphisms on $\clm$ then $\alpha_t=\sigma_t$ for all $t \in \IR$. In particular, if $\theta$ is an automorphism on $\clm$ preserving $\omega$ then $\sigma_t \theta = \theta \sigma_t$ for all
$t \in \IR$. Furthermore, KMS relation also says that the von-Neumann sub-algebra $\{a \in \clm: \sigma_t(a) = a;\; t \in \IR \}$ is equal to $\{a \in \clm: \omega(ab)=\omega(ba),\;b \in \clm \}$      

\vsp 
Let $\zeta_{\omega_1}$ and $\zeta_{\omega_2}$ be cyclic and separating unit vectors for
standard von-Neumann algebras $\clm_1$ and $\clm_2$ acting on Hilbert spaces $\clh_{\omega_1}$ and $\clh_{\omega_2}$ respectively. Let $\tau:\clm_1 \raro \clm_2$ be a normal unital completely positive map such that $\omega_2 \tau = \omega_1$. Then there exists a unique unital completely positive normal map 
$\tau':\clm_2' \raro \clm_1'$ ([section 8 in [OP] ) satisfying the duality relation 
\be 
\langle b\zeta_{\omega_2},\tau(a)\zeta_{\omega_2} \rangle =  \langle \tau'
(b)\zeta_{\omega_1},a\zeta_{\omega_1} \rangle 
\ee
for all $a \in \clm_1$ and $b \in \clm_2'$. For a proof, we refer 
to section 8 in the monograph [OP] or section 2 in [Mo1]. We set the dual 
unital completely positive map $\tilde{\tau}:\clm_2 \raro \clm_1$ defined by 
\be 
\tilde{\tau}(b) =  \clj_{\omega_1}\tau'(\clj_{\omega_2}b\clj_{\omega_2})\clj_{\omega_1}
\ee 
for all $b \in \clm_2$. In particular, we have $\omega_2 = \omega_1 \tilde{\tau}$. 

\vsp 
If $\cln=\clm_2$ is a von-Neumann sub-algebra of $\clm=\clm_1$ with a faithful normal state $\omega(a) = \langle \zeta_{\omega},a \zeta_{\omega} \rangle$ and 
$i_{\cln}:\cln \raro \clm$ is the inclusion map of $\cln$ into $\clm$. Then the dual of 
$i_{\cln}$ with respect to $\omega$, denoted by $\IE_{\omega}:\clm \raro \cln$ is a {\it norm one projection } i.e. 
\be 
\IE_{\omega}(abc)=a\IE_{\omega}(b)c
\ee 
for all $a,c \in \cln$ and $b \in \clm$ if and only if $\sigma_t^{\omega}(\cln)=\cln$. 
For a proof we refer to [Ta1] and for a local version of this theorem [AC]. 

\vsp 
We start with an elementary lemma. 

\vsp 
\begin{lem} 
Let $\clb$ be a unital $C^*$-algebra and $\tau:\clb \raro \clb$ be a unital completely positive map preserving a faithful state $\omega$. Then the following holds:

\vsp 
\NI (a) $\clb_{\tau}=\{x \in \clb: \tau(x)=x \}$ and $\clf_{\tau}= \{ x \in \clb: \tau(x^*x)=\tau(x)^*\tau(x),\;\tau(xx^*)=\tau(x)\tau(x)^* \}$
are $C^*$ sub-algebras of $\IM$ and $\clb_{\tau} \subseteq \clf_{\tau}$.

\vsp 
\NI (b) Let $\IE$ be norm-one projection on a $C^*$-sub-algebra $\clb_0 \subseteq \clb$ i.e. A unital completely positive map 
$\IE:\clb \raro \clb_0$ with range equal to $\clb_0$ satisfying bi-module property 
$$\tau(yxz)=y\tau(x)z$$ 
for all $y,z \in \clb_0$ and $x \in \clb$ and $\omega$ be a faithful invariant state for $\IE$. Then for some $x \in \clb$, $\IE(x^*)\IE(x)=\IE(x^*x)$ if and only if $\IE(x)=x$; 

\vsp 
\NI (c) $\clb_{\IE}=\clf_{\IE}$.   
\end{lem} 

\vsp 
\begin{proof} 
Though (a) is well known [OP], we include a proof in the following. The map $\tau$ being completely positive (2 positive is enough), we have Kadison inequality [Ka] $\tau(x^*)\tau(x) \le \tau(x^*x)$ for all $x \in \clb$ and equality holds 
for an element $x \in \clb$ if and only if $\tau(x^*)\tau(y)=\tau(x^*y)$ for all $y \in \clb$. This in particular shows that $\clf_{\tau}=\{x \in \clb:\tau(x^*x)=\tau(x^*)\tau(x),\;\tau(x)\tau(x^*)=\tau(xx^*) \}$ is $C^*$ -sub-algebra. We claim by the invariance property that $\omega \tau =\omega$ and faithfulness of $\omega$ and $\clb_{\tau} \subseteq \clf_{\tau}$. We choose $x \in \clb_{\tau}$ and by Kadison inequality we have 
$x^*x=\tau(x^*)\tau(x) \le \tau(x^*x)$ but $\omega(\tau(x^*x)-x^*x)=0$ by the invariance property and thus by faithfulness $x^*x=\tau(x^*)\tau(x)=\tau(x^*x)$. Since $x^* \in \clb_{\tau}$ once $x \in \clb_{\tau}$, we get $\tau(x)\tau(x^*)=\tau(xx^*)=xx^*$. This shows that $\clb_{\tau} \subseteq \clf_{\tau}$ and $\clb_{\tau}$ is an algebra by the polarization identity $4x^*y= \sum_{0 \le k \le 3}i^k(x+i^ky)^*(x+i^ky)$ for any two elements $x,y \in \clb$ as $\clb_{\tau}$ is a $*$-closed linear vector subspace of $\clb$. 

\vsp 
For (b) let $a \in \clb$ with $\IE(a^*a)=\IE(a^*)\IE(a)$. We have by the first part of the proof for (a), $\IE(a^*)\IE(b)=\IE(a^*b)$ for all $b \in \clb$ and thus
$$\omega(\IE(a^*)b))$$
$$=\omega(\IE(\IE(a^*)b))$$
(by the invariance property)
$$=\omega(\IE(a^*)\IE(b))$$
(by bi-module property)
$$=\omega(\IE(a^*b))$$
( by the first part of the argument used to prove (a) )
$$=\omega(a^*b)$$ 
for all $b \in \clb.$
So we have $\omega((\IE(a^*)-a^*)b)=0$ for all $b \in \clb$ and thus 
by faithful property of $\omega$ on $\clb$, we get $\IE(a^*)=a^*$ i.e. $\IE(a)=a$.

\vsp 
By (b), for $x \in \clb$ with $\IE(x^*)\IE(x)=\IE(x^*x)$, we have $\IE(x)=x$ and thus 
$\clf_{\IE} \subseteq \clb_{\IE}$. Now by (a), we complete the proof for (c). 
\end{proof}

\vsp 
Let $\omega$ be a faithful state on a $C^*$ sub-algebra $\IM$ and $(\clh_{\omega},\pi_{\omega}, \zeta_{\omega})$ be the GNS representation of $(\IM,\omega)$ so that 
$\omega(x)=\langle \zeta_{\omega}, \pi_{\omega}(x) \zeta_{\omega} \rangle$ for all $x \in \IM$. Thus the unit vector $\zeta_{\omega} \in \clh_{\omega}$ is cyclic and separating for von-Neumann algebra $\pi_{\omega}(\cla)''$ that is acting on the Hilbert space $\clh_{\omega}$. Let $\Delta_{\omega}$ and $\clj_{\omega}$ be the modular and conjugate operators on $\clh_{\omega}$ respectively of $\zeta_{\omega}$ as described in (12). The modular automorphisms group $\sigma^{\omega}=(\sigma^{\omega}_t:t \in \IR)$ 
of $\omega$ is defined by 
$$\sigma_t^{\omega}(a)=\Delta^{it}_{\omega}a\Delta^{-it}$$
for all $a \in \pi_{\omega}(\IM)''$. 

\vsp 
We are interested now to deal with faithful states $\omega$ on $\IM=\otimes_{n \in \IZ} \!M_d^{(n)}(\IC)$. For a given faithful state $\omega$, $\clm_{\omega,\Lambda}=\pi_{\omega}(\IM_{\Lambda})''$ is a von-Neumann sub-algebra of $\clm_{\omega}=\pi_{\omega}(\IM)''$. In general, $(\sigma^{\omega}_t)$ may not keep $\clm_{\omega,\Lambda}$ invariant. However, for an infinite tensor product state $\omega_{\rho} = \otimes_{n \in \IZ} \rho^{(n)}$ with $\rho^{(n)}=\rho$ for all $n \in \IZ$, $\sigma_t^{\omega_{\rho}}$ preserves $\clm_{\omega_{\rho},\Lambda}$. 

\vsp  
\begin{lem} 
Let $\IM=\otimes_{k \in \IZ} \!M^{(k)}_d$ and $\IM_{\Lambda}$ be the local $C^*$ algebra associated with a subset $\Lambda$ of $\IZ$. Then there exists a norm one projection $\IE_{\Lambda}:\IM \raro \IM_{\Lambda}$ 
preserving unique tracial state $\omega_0$ of $\IM$ satisfying the following properties:

\vsp 
\NI (a) $\IE_{\Lambda}$ commutes with the group of automorphisms $\{\beta_g:g \in \otimes_{k \in \IZ} U_d(\IC) \}$, where $g=\otimes_n g_n$ with all $g_n=I_d$ except finitely many $n \in \IZ$. 

\vsp 
\NI (b) For all $x \in \IM$, $||\IE_{\omega_0,\Lambda}(x)-x|| \raro 0$ as $\Lambda \uparrow \IZ$.  
\end{lem}

\begin{proof} 
There exists a unique completely positive map $\IE_{\omega_0,\Lambda}:\pi_{\omega_0}(\IM)'' \raro \pi_{\omega_0}(\IM_{\Lambda})''$ satisfying 
\be 
\omega_0(\pi_{\omega_0}(z)\IE_{\omega_0,\Lambda}(\pi_{\omega_0}(x))=\omega_0(zx)
\ee
for all $x \in \IM$ and $z \in \IM_{\Lambda}$. It follows trivially by a theorem of M. Takesaki [Ta1] since modular group being trivial preserves $\pi_{\omega_0}(\IM_{\Lambda})$. For an indirect proof, we can use duality argument used in [AC] to describe $\IE_{\Lambda}$ as the $KMS$-dual map of the inclusion map $i_{\Lambda}:z \raro z$ of $\pi_{\omega_0}(\IM_{\Lambda})''$ into $\pi_{\omega_0}(\IM)''$. The modular group being trivial we get the simplified relation (17). 
The map $\IE_{\omega_0,\Lambda}$ restricted to $\pi_{\omega_0}(\IM_{\Lambda})''$ and $\pi_{\omega_0}(\IM_{\Lambda^c})''$ are 
the identity map on $\pi_{\omega_0}(\IM_{\Lambda})''$ and $\omega_0$ on $\pi_{\omega_0}(\IM_{\Lambda^c})''$ respectively. 
In particular $\IE_{\omega_0,\Lambda}(\pi_0(\IM_{loc})) \subseteq \pi_{\omega_0}(\IM_{loc})$ and there exists 
a map $\IE_{\Lambda}: \IM \raro \IM_{loc}$ so that 
$$\IE_{\omega_0,\Lambda}(\pi_{\omega_0}(x)) = \pi_{\omega_0}(\IE_{\Lambda}(x))$$
The representation $\pi_{\omega_0}$ be faithful, we may identify $x \in \IM$ with $\pi_{\omega_0}(x)$ and $\IE_{\omega_0,\Lambda}$
with $\IE_{\Lambda}$ in the following computation. 

\vsp 
That $\IE_{\Lambda}(z)=z$ for 
$z \in \IM_{\Lambda}$ is obvious by the faithful property of normalised trace. We may 
verify the bi-module property (17) directly: for all $y,z \in \IM_{\Lambda}$ and $x \in \IM$, 
$$\omega_0(z\IE_{\Lambda}(x)y)$$
$$=\omega_0(yz\IE_{\Lambda}(x))$$
$$=\omega_0(yzx)$$
$$=\omega_0(zxy)$$
$$=\omega_0(\IE_{\omega_0,\Lambda}(zx)y)$$
This shows that $\IE_{\Lambda}(zx)=z\IE_{\Lambda}(x)$ for all $z \in \IM_{\Lambda}$ and $x \in \IM$. By taking adjoint, we also get $\IE_{\Lambda}(xy)=\IE_{\Lambda}(x)y$ for all $y \in \IM_{\Lambda}$ and $x \in \IM$. Thus we arrive at the bi-module relation (17).  

\vsp 
Since $\beta_g(\IM_{\Lambda})=\IM_{\Lambda}$ for all $\Lambda$ for $g \in \otimes_{k \in \IZ} U_d(\IC)$, for all $z \in \IM_{\Lambda}$ and $x \in \IM$ we get 
$$tr(z \IE_{\Lambda} \beta_g(x))$$
$$=tr(z \beta_g(x))=tr(\beta_{g^{-1}}(z)x)$$
$$=tr(\beta_{g^{-1}}(z)\IE_{\Lambda}(x))$$
$$=tr(z \beta_g \IE_{\Lambda}(x))$$
So we get $\beta_g \IE_{\Lambda} = \IE_{\Lambda} \beta_g$ for all $g \in \otimes_{k \in \IZ}U_d(\IC)$. 

\vsp 
Since $\IE_{\Lambda}(\IM_{loc}) \subseteq \IM_{loc}$, for an element $x \in \IM_{\Lambda'}$, $\IE_{\Lambda}(x)=x$ 
once $\Lambda' \subseteq \Lambda$. Thus $||\IE_{\Lambda}(x)-x|| \raro 0$ as $\Lambda \uparrow \IZ$ for all 
$x \in \IM_{loc}$. Now we use $||\IE_{\Lambda}|| \le 1$ for all $\Lambda \subseteq \IZ$ and norm dense property of 
$\IM_{loc}$ in $\IM$ to complete the proof. 

\end{proof}

\vsp 
\begin{lem}
Let $A$ be $C^*$ subalgebra of $\IM$ with $A \vee A'=\IM$. Then there exists a norm one projection $\IE_A:\IM \raro A$ preserving the tracial state $\omega_0$ of $\IM$. Furthermore, 

\vsp 
\NI (a) $\pi_{\omega_0}(A)''$ is a factor;

\vsp 
\NI (b) $A''=A$. 
\end{lem} 

\vsp 
\begin{proof} 
Let $\IE_{\cla}$ be the unique norm one projection from $\clm=\pi_{\omega_0}(\IM)''$ onto $\cla=\pi_{\omega_0}(A)''$ preserving tracial state $\omega_0$ of $\IM$. We need to show $\IE_{\cla}(\pi_{\omega_0}(x))=\pi_0(\IE_A(x))$ for some element $\IE_A(x) \in A$ and the map $\IE_A:\IM \raro A$ is a norm one projection. Only non-trivial part of this statement is existance of an element $\IE_A(x) \in \IM$, which can be called `Feller property' of a Markov map. 

\vsp 
For any $x \in \IM$, we have $\IE_{\cla}(\pi_{\omega_0}(x))=\pi_{\omega_0}(x)$. Since $A \vee A'=\IM$ elements $x=\sum_i x_i y_i$ 
with $x_i \in A$ and $y_i \in A'$ are norm dense in $\IM$. 

\vsp 
In particular, $\pi_{\omega}(\IM)''= \pi_{\omega_0}(A)'' \vee \pi_{\omega_0}(A')''$ and so 
$\pi_{\omega_0}(\IM)'= \pi_{\omega_0}(A)' \bigcap \pi_{\omega_0}(A')'$. But $\pi_{\omega_0}(A)'' \subseteq \pi_{\omega_0}(A')'$ and so $\pi_{\omega_0}(A)'' \bigcap \pi_{\omega_0}(A)' \subseteq \pi_{\omega_0}(\IM)'' \bigcap \pi_{\omega_0}(\IM)'$. Since $\omega_0$ is factor state of $\IM$, we conclude that $\pi_{\omega_0}(A)''$ is also a factor. 

\vsp 
We check now $\IE_{\cla}(\pi_{\omega_0}(x))= \sum_i \pi_{\omega_0}(x_i) \IE_{\cla}(\pi_{\omega_0}(y_i))$. But for any element 
$y \in A'$, we have $\IE_{\cla}(\pi_{\omega}(y)) \in \pi_{\omega_0}(A)'' \bigcap \pi_{\omega_0}(A)'$. So by the factor property of $\pi_{\omega_0}(A)''$, we conclude that $\IE_{\cla}(y_i)=\omega_0(y_i)I$
Thus $\IE_{\cla}(\pi_{\omega}(x)) \in \pi_{\omega}(A)$ for all $x \in \IM$, where $x=\sum_i x_iy_i$ with $x_i \in A$ and $y_i \in A'$. Since $||\IE_{\cla}(\pi_{\omega_0}(x))|| \le ||\pi_{\omega}(x)||$ for all $x \in \IM$, we conclude $\IE_{\cla}(\pi_{\omega_0}(x)) \in \pi_{\omega_0}(A)$ for all $x \in \IM$. So there exists an element $\IE_{A}(x) \in A$ so 
that $\IE_{\cla}(\pi_{\omega_0}(x))=\pi_{\omega_0}(\IE_{A}(x))$ by the norm dense property. Now we use faithfulness of the reprentation $\pi_{\omega_0}$ of $\IM$, to find bi-module property of $\IE_{A}$ from that of $\IE_{\cla}$. 

\vsp 
Since $A \vee A'=\IM$, we have $A' \bigcap A'' = (A \vee A')'=\IC$. We also have $A \subseteq A''$ for any sub-algebra, so 
with $A$ replaced by $A'$, we get $A' \subseteq A'''$. On the other hand taking commutant on both side, we get $A''' \subseteq A'$. This shows $A'=A'''$. In particular, $A'' \vee A''' = A'' \vee A' = \IM$ since $A \vee A' =\IM$ and $A \subseteq A''$. 

\vsp 
By the first part of the present lemma, there exists a norm one projection $\IE_{A''}:\IM \raro A''$, onto $A''$ preserving $\omega_0$. Since $A \vee A' = \IM$, elements $a = \sum_i x_iy_i$ with $x_i \in A$ and $y_i \in A'$ are norm dense in $\IM$. If
$a \in A''$, we choose a sequence $a_n = \sum_i x^n_iy^n_i$ with $x^n_i \in A$ and $y^n_i \in A'$ so that $||a-a_n|| \raro 0$ and verify that $\IE_{A''}(a_n) = \sum_i x^n_i \IE_{A''}(y^n_i) = \sum_i x^n_i \omega_0(y^n_i) \in A$ since $A'' \bigcap A' = (A' \vee A)'=\IC$. However $||\IE_{A''}(a_n)-\IE_{A''}(a)|| \le ||a_n-a|| \raro 0$ and $\IE_{A''}(a)=a$ and thus $a \in A$.
\end{proof}

\vsp 
\begin{lem}
For any uniformly hyperfinite $C^*$-subalgebra $A$ of a $C^*$-algebra $\IM$, $A''=A$. 
\end{lem}

\vsp 
\begin{proof}
\vsp 
Let $\{A_{p_n},n \ge 1\}$ be a family of finite sub-factors of increasing dimension $p_n$ such that $\vee A_{p_n} = A$. Then 
$p_{n-1}|p_n$ and $A$ is $C^*$-isomorphic to $\otimes_{n \ge 1} B_{q_n}$, where $B_{q_n}$ is $C^*$-isomorphic to $A_{p_n} \bigcap A'_{p_{n-1}}$ of dimension $q_n={p_n \over p_{n_1}}$, where we have set $A_{p_0}=\IC$. So it is a straight forward to calculate $A'' \equiv \otimes_{n \ge 1} B_{q_n}'' = \otimes_{n \ge 1}B_{q_n} \equiv A$ 
as $B_{q_n}''=B_{q_n}$ for $n \ge 1$ being finite dimensional.   
\end{proof} 

\vsp 
We verify few simple relations. For any $C^*$-subalgebra $A$ of a $C^*$-algebra $\IM$, $A \subseteq A''$ and so by taking communtant on both side we get $A''' \subseteq A'$. However $A' \subseteq A'''$ by its defining commuting relation. So $A'=A'''$. Thus $A'' \bigcap A' = (A' \vee A)'= \IC$ if $A \vee A' = \IM$, if so then $A''$ is also a factor and so is $A$.

\vsp 
We end this section with a proof of Proposition 1.1. 

\vsp 
\begin{proof} 
\vsp 
Let $\IM_0$ be a $C^*$-algebra of $\IM$ such that 
$\IM_0''=\IM_0$ and $\theta^{-1}(\IM_0) \subseteq \IM_0$. Then the mutually commuting family of $C^*$-algebras 
$$\IM^{(n)} = \theta^n(\IM_0) \bigcap \theta^{n-1}(\IM_0)'$$ 
are isomorphic copies of $\IM^0$ with $\IM^{(n)}= \theta^n(\IM^{(0)})$. We claim that 
$$\vee_{n \in \IZ} \IM^{(n)} = \IM$$ 
For a proof, let $a$ be an element in the commutant of $\vee_{n \in \IZ} \IM^{(n)}$. Let $\IE_0$ be the conditional
expectation from $\IM$ onto $\IM_0$ with respect to the normalized trace $\omega_0$ on $\IM$. Then $\IE_0(a) \in \IM_0$ 
and for any $y \in \IM^{(0)}$ we have 
$$\IE_0(a)y$$
$$=\IE_0(ay)$$
$$=\IE_0(ya)$$
$$=y\IE_0(a)$$ 
So $\IE_0(a) \in \IM_0 \bigcap \IM^{(0)})'$. 

\vsp 
We claim that 
$$\theta^{-1}(\IM_0) = \IM_0 \bigcap (\IM^{(0)})'$$ 
 
\vsp 
By our definition of $\IM^{(0)}$, we have 
$\theta^{-1}(\IM_0) \subseteq \IM'_0 \vee (\theta^{-1}(\IM_0))'' \subseteq (\IM^{(0)})'$. 
Since $\theta^{-1}(\IM_0) \subseteq \IM_0$, we also have 
$$\theta^{-1}(\IM_0) \subseteq \IM_0 \bigcap (\IM^{(0)})' \subseteq \IM_0$$
where the last inclusion is obvious. 

\vsp 
We claim first that $A \vee A' = \IM$ if $A=\theta^{-1}(\IM_0) \vee \IM'_0$. For a proof we fix any $x \in \IM_0$ and use $\theta^{-1}(\IM_0) \vee \theta^{-1}(\IM'_0)=\IM$ to find elements $y^k_n \in \theta^{-1}(\IM_0)$ and $z^k_n \in \theta^{-1}(\IM'_0)$ so that $x$ is the norm limit of the sequence $\sum_k y^k_n z^k_n$. Then $x$ is also the norm limit of the sequence $\sum_k y^k_n \IE_{\IM_0}(z^k_n)$, where $\IE_{\IM_0}(z^k_n) \in \IM_0 \bigcap \theta^{-1}(\IM'_0)$ since $\theta^{-1}(\IM_0) \subseteq \IM_0$. This show that 
$$\IM_0 = (\IM_0 \bigcap \theta^{-1}(\IM_0) ) \vee (\IM_0 \bigcap \theta^{-1}(\IM_0)').$$ 
Now we use $\IM_0 \vee \IM'_0 = \IM$ to conclude $A \vee A' = \IM$.

\vsp 
So by Lemma 2.3 (b), we have $A''=A$, i.e. $(\IM^{(0)})'=\theta^{-1}(\IM_0) \vee \IM'_0$. 

\vsp 
Elements in $\IM'_0$ commutes with $\theta^{-1}(\IM_0)$ as it is a subset of $\IM_0$ and so any element $a$ in $\IM^{(0)'}$ is of the norm limit of elements of the form $a_n=\sum_i x^n_iy^n_i$, where $x^n_i \in \theta^{-1}(\IM_0)$ and $y^n_i \in \IM'_0$. 
If $a$ is also in $\IM_0$, we have $\IE_{0}(a)=a$ and 
$$\IE_{0}(a_n)=\IE_{0}(\sum_i x^n_iy^n_i)= \sum_i x_i^n \IE_{0}(y^n_i),$$ 
where we recall $\IE_{0}$ is conditional expectation from $\IM$ onto $\IM_0$ preserving normalised trace $\omega_0$. 
Since $\IE_{0}(y)$ is an element in $\IM_0 \bigcap \IM'_0$ for any element $y \in \IM'_0$, we get each 
$\IE_{0}(a_n) \in \theta^{-1}(\IM_0)$ since $\IM_0 \bigcap \IM_0'=\IC$. As $||\IE_0(a-a_n)|| \le ||a-a_n|| \raro 0$ by our choice $a_n$ for $a$, we conclude that $a \in \theta^{-1}(\IM_0)$. 

\vsp 
Now we may repeat the argument with elements $y \in \IM^{(-1)}$ to show that 
$\IE_0(a) \in \theta^{-2}(\IM_0)$. Thus by mathematical induction on $n$ we conclude that 
$\IE_0(a) \in \bigcap_{n \le 0} \theta^n(\pi_{\omega_0}(\IM_0)$. By our assumption (c), we get 
$\IE_0(a)$ is a scalar. 

\vsp 
Let $\IE_n$ be the conditional expectation from $\IM$ onto $\theta^n(\IM_0)$ with respect to the unique normalized trace $\omega_0$. Then by the same argument use above, we get $\IE_n(a)$ is a scalar. However, by (b), we have for any $y \in \IM$
$$\IE_n(\pi_{\omega_0}(y)) - \pi_{\omega_0}(y) \raro 0$$ 
as $n \raro \infty$ in weak operator topology. We conclude that $\pi_0(a)$ is a scalar multiple of identity operator 
and so $a=\omega_0(a)I$. This shows that $(\vee_{n \in \IZ} \IM^{(n)})' = \IC$ once relations (a)-(c) in the statement of Proposition 1.1 hold. We use a temporary notation $\IP = \vee_{n \in \IZ} \IM^{(n)}$ and so $\IP \subseteq \IM$ and $\IP'=\IC$.

\vsp 
Since $\IP \bigcap \IP' = \IC$ and any element $x \in \IM^{(0)} \bigcap (\IM^{(0)})'$ commutes with all elements in $\IM^{(n)}$ for each $n \in \IZ$, we have $\IM^{(0)} \bigcap (\IM^{(0)})'= \IC$. We claim further that the linear map which extends the following map 
\be 
x=\Pi_{n \in \IZ} x_n \raro \tilde{x} =  \otimes_{n \in \IZ} \tilde{x}_n
\ee 
is a $C^*$ isomorphism between $\IP$ and $\tilde{\IM} = \otimes_{n \in \IZ } \tilde{\IM}^{(n)}$,  
where $x_n \in \IM^{(n)}$, taking values $I$ except for finitely many $n \in \IZ$ 
and $\tilde{\IM}^{(n)}$ are copies of $\IM^{(0)}$ with elements $\tilde{x}_n=\theta^{-n}(x_n) \in \IM^{(0)}$ for all $n \in \IZ$. The universal property of tensor products implies: as a vector space $\IP$ is isomorphic to $\otimes_{n \in \IZ} \tilde{\IM}^{(n)}$. That the map is also $C^*$ isomorphic, follows once we verify 
\be 
||x|| = ||\tilde{x}||
\ee 
for $x=\Pi_{n \in \IZ} x_n$. 
Note that $\IM$ being a nuclear $C^*$-algebra [ChE], so there is a unique $C^*$ norm determined by its cross norm [Pau]. For commuting self-adjoint elements $(x_n)$, spectrum $\sigma(x)$ of $x=\Pi_{n \in \IZ}x_n$ is given by $\sigma(x)= \{ \Pi_{n \in \IZ} \lambda_n: \lambda_n \in \sigma(x_n) \}$, where $\sigma(x_n)$ is the spectrum of $x_n$. Thus 
$$||x||=\mbox{sup}_{\lambda_n \in \sigma(x_n)} \Pi_{n \in \IZ} |\lambda_n| $$
$$= \Pi_{n \in \IZ_n} \mbox{sup}_{\lambda_n \in \sigma(x_n) }|\lambda_n| $$
$$= \Pi_{n \in \IZ} ||x_n||$$

\vsp 
So for the equality of norms in (21), we can use Gelfand theorem on spectral radius for 
a self-adjoint element $x=\Pi_{ n \in \IZ}x_n $ and then use $C^*$ property of the norms to verify the equality (21) for all $x=\Pi_{n \in \IZ} x_n$. Thus the $C^*$ algebra $\IP$ is isomorphic to $\tilde{\IM}$.

\vsp 
This shows that there exists an automorphism $\alpha:\IP \raro \tilde{\IM}$ that commutes with $\theta$ taking $\IM^{(n)}$ to $\tilde{\IM}^{(n)}$ for each $n \in \IZ$ and $\tilde{\IM}$ i.e. in particular, $(\IP,\theta,\omega_0)$ is isomorphic to $(\tilde{\IM},\theta,\omega_0)$, where we used the same notations $\omega_0$ for the unique normalized traces on $\IM$ and $\tilde{\IM}$ respectively. Since Connes-St\o rmer dynamical entropy is an invariant for $C^*$ dynamical system and is equal to mean entropy for the unique normalised trace [CS], we get $\tilde{\IM}^{(n)} \equiv \!M^{(n)}_{d'}(\IC)$ for some $d' \le d$. Since $\IP'=\IC$, we get equality $d'=d$. 
Thus we can identify $\tilde{\IM}$ with $\IM$. 

\vsp 
This shows the isomorphism (20) is induced by an automorphism $\alpha:\IM \raro \IM$ such that
$\alpha(\IM_0)=\IM_L$, where $\alpha$ commutes with $\theta$ by our construction. 
\end{proof} 

\vsp 
\begin{rem} 
Proposition 1.1 says, for a nuclear $C^*$ algebra $\clb$ with a normalised trace $\omega_0$, a $C^*$-dynamical system $(\clb,\theta,\omega)$ satisfying (a)-(c) with $\clb_0 \subset \clb$ is isomorphic to $(\tilde{\clb},\tilde{\theta},\tilde{\omega})$, where $\tilde{\theta}$ is 
the translation dynamics on a two-sided infinite tensor product $C^*$-algebra $\tilde{\clb} = \otimes_{n \in \IZ} \tilde{\clb}^{(n)}$ 
with an invariant state $\tilde{\omega}$ and $\tilde{\clb}^{(n)},n \in \IZ$ are copies of a $C^*$ sub-algebra of $\clb$. In particular, $\tilde{\omega}_0$ is also a trace on $\tilde{\clb}$.    
\end{rem}

\section{Maximal abelian $C^*$ sub-algebras and automorphisms of $\IM$ }

\vsp 
We recall briefly notations used in the following text. Let $\Omega=\{1,2,..,d\}$ 
and $\Omega^{\IZ}=\times_{n \in \IZ} \Omega^{(n)}$, where $\Omega^{(n)}$ are copies of $\Omega$ and equip with product topology. Thus $C(\Omega^{\IZ})$ can be identified with the $C^*$-sub-algebra $\ID^e= \otimes_{n \in \IZ} D_e^{(n)}(\IC)$, where $D_e^{(n)}(\IC)=D_d(\IC)$ for all $n \in \IZ$ and $D_d(\IC)$ is the set of diagonal matrices with respect to an orthonormal basis $e=(e_i:1 \le i \le d)$ for $\IC^d$. In other words, 
we identify $d \times d$ diagonal matrices $D_d(\IC)$ with the algebra $C(\Omega)$ of complex valued continuous functions on $\Omega$. 

\vsp 
A commutative $*$-subalgebra $D$ of $\!M_d(\IC)$ is called maximal abelian if $D'=D$. In such a case $D=D_d(\IC)$ for some orthonormal basis $e$ for $\IC^d$. More generally, an abelian $C^*$ sub-algebra $D$ of a $C^*$ algebra $\clb$ is called maximal abelian if $D'=D$. As a first step, we investigate some maximal abelian $C^*$-subalgebras of $\IM=\otimes_{n \in \IZ}\!M^{(n)}_d(\IC)$ those are $\theta$ invariant i.e. $\theta(D)=D$. In the following we use norm one projections $\IE_{\Lambda}: \IM \raro \IM_{\Lambda}$ with respect to unique tracial state. We recall that $\IE_{\Lambda}$ commutes with $\{\beta_g:g \in \otimes U_d(\IC) \}$ by Lemma 2.2. 
 
\vsp 
\begin{lem} 
For a subset $\Lambda$ of $\IZ$ we have 

\NI (a) $\IM_{\Lambda}'=\IM_{\Lambda'}$, where $\Lambda'$ is the complementary set of $\Lambda$ in $\IZ$;

\NI (b) $\IM_{\Lambda} \bigcap \ID^e = \ID^e_{\Lambda}$;

\NI (c) $\ID^e$ is a maximal abelian $*$-sub-algebra of $\IM$ and $\ID^e \equiv C(\Omega^{\IZ})$, the algebra of continuous functions on $\Omega^{\IZ}$;     

\NI (d) We set $C^*$-sub-algebra $\IM^e_{\Lambda} = \IM_{\Lambda} \vee \ID^e$ of $\IM$. Then the following hold:  

\NI (i) $(\IM^e_{\Lambda})'=\ID^e_{\Lambda'}$;

\NI (ii) $(\ID^e_{\Lambda'})' = \IM^e_{\Lambda}$. 

\end{lem} 

\vsp 
\begin{proof} 
That $\IM_{\Lambda'} \subseteq \IM_{\Lambda}'$ is obvious. For the reverse inclusion, we fix any element $x \in \IM_{\Lambda}'$ and consider the elements $\IE_{\Lambda_1}(x)$ for 
$\Lambda \subseteq \Lambda_1$ so that $\Lambda_1 \bigcap \Lambda'$ is a finite subset of $\IZ$. 
For any unitary element $u \in \IM_{\Lambda}$ we have
$$u\IE_{\Lambda_1}(x)u^*$$
$$= \IE_{\Lambda_1}(uxu^*)$$
$$=\IE_{\Lambda_1}(x)$$
Thus $\IE_{\Lambda_1}(x) \in \IM_{\Lambda_1} \bigcap \IM_{\Lambda}'=\IM_{\Lambda_1 \bigcap \Lambda'}$. Since $||x-\IE_{\Lambda_1}(x)|| \raro 0$ as $\Lambda_1 \uparrow \IZ$ in the sense of van-Hove [BR2], we conclude that $x \in \IM_{\Lambda'}$. 

\vsp 
For (b), we can repeat ideas of the proof of (a) with obvious modification. To prove the non trivial inclusion, we fix any element $x \in \IM_{\Lambda} \bigcap \ID^e$ and consider 
$\IE_{\Lambda_1}(x)$ for $\Lambda \subseteq \Lambda_1$ such that $\Lambda_1 \bigcap \Lambda'$ is a finite subset of $\IZ$. For any unitary element $u \in \IM_{\Lambda_1 \bigcap \Lambda'} \vee \ID^e_{\Lambda}$ 
we have 
$$u\IE_{\Lambda_1}(x)u^*$$
$$=\IE_{\Lambda_1}(uxu^*)$$
$$=\IE_{\Lambda_1}(x)$$
Thus $\IE_{\Lambda_1}(x) \in \IM_{\Lambda} \bigcap \ID^e_{\Lambda}=\ID^e_{\Lambda}$. 
Taking von-Hove limit $||x -\IE_{\Lambda_1}(x)|| \raro 0$ as $\Lambda_1 \uparrow \IZ$, we get $x \in \ID^e_{\Lambda}$.  This completes the proof of (b).

\vsp 
For non trivial inclusion $D_e' \subseteq D_e$, we fix an element $x$ in the commutant of $\ID^e$. Then we have 
$$u\IE_{\Lambda}(x)u^*$$
$$=\IE_{\Lambda}(uxu^*)$$
$$=\IE_{\Lambda}(x)$$ 
for all unitary element $u \in \ID^e_{\Lambda}= \otimes_{k \in \Lambda}D^{(k)}_e$ and thus $\IE_{\Lambda}(x) \in \ID^e_{\Lambda}$ since $\ID^e_{\Lambda}$ is a maximal abelian sub-algebra of $\IM_{\Lambda}$ for $|\Lambda| < \infty$. Since $||x-\IE_{\Lambda}(x)|| \raro 0$ as $\Lambda \uparrow \IZ$ as van Hove limit, we conclude that $x \in \ID^e$ as $\ID^e$ is the norm closure of $\{\ID^e_{\Lambda}: \Lambda \subset \IZ \}$.  

\vsp 
For (d), we have the following equalities: 
$$(\IM^e_{\Lambda})'=\IM_{\Lambda}' \bigcap (\ID^e)'$$
$$=\IM_{\Lambda'} \bigcap \ID^e $$
(by Lemma 3.1 (a) )
$$=\ID^e_{\Lambda'}$$ 
(by Lemma 3.1 (b) ), where $\Lambda'$ is the complementary set of $\Lambda$. 

\vsp 
Furthermore, we claim also that 
$$(\ID^e_{\Lambda'})'=\IM^e_{\Lambda}$$
That $\IM^e_{\Lambda} \subseteq (\ID^e_{\Lambda'})'$ is obvious. For the reverse inclusion, let $x \in (\ID^e_{\Lambda'})'$ and consider the sequence of elements $x_n=\IE_{\Lambda_n}(x)$, where $\Lambda_n$ is sequence of subsets containing $\Lambda$ and $\Lambda_n \uparrow \IZ$ as $n \raro \infty$. Then for all $y \in \ID^e_{\Lambda_n \bigcap \Lambda'}$, we have 
$$y\IE_{\Lambda_n}(x)$$
$$=\IE_{\Lambda_n}(yx)$$
$$=\IE_{\Lambda_n}(xy)$$
$$=\IE_{\Lambda_n}(x)y$$
and thus $\IE_{\Lambda_n}(x) \in (\ID^e_{\Lambda_n \bigcap \Lambda'})' \bigcap \IM_{\Lambda_n} 
\subseteq \IM^e_{\Lambda}$. 
Taking limit $n \raro \infty$ in $||x-x_n|| \raro 0$, we get $x \in \IM^e_{\Lambda}$. 
\end{proof}

\vsp 
Let $\clp$ be the collection of orthogonal projections in $\IM$ i.e. $\clp = \{ p \in \IM: p^*=p,\;p^2=p \}$.  
So $\clp$ is a norm closed set and $\omega_{\rho}(e^I_I)={1 \over d^n}$ , where 
$e^I_I=..\otimes I_d \otimes |e_{i_1}\rangle\langle e_{i_1}| \otimes |e_{i_2}\rangle\langle e_{i_2}|..
\otimes |e_{i_n}\rangle\langle e_{i_n}| \otimes I_d ..$, where $e=(e_i:1 \le i \le d)$ is an orthonormal 
basis for $\IC^d$. This shows that $\clp$ has no non-zero minimal projection and range of 
the map $\omega_{\rho}:\clp \raro [0,1] $ at least has all $d-$adic numbers i.e. 
$\cli_d = \{{j \over d^n}: 0 \le j \le d^n,\;n \ge 1 \}$.       

\vsp 
An abelian sub-collection $\clp_0$ of projections is called {\it maximal } if there exists an abelian 
collection of projections $\clp_1$ containing $\clp_0$ then $\clp_1=\clp_0$. As a set $\clp_0$ is a closed subset of $\clp$. It is simple to verify that $\clp_0$ is maximal if and only if any projection commuting with elements in $\clp_0$ is an element itself in $\clp_0$. It is less obvious that the norm closure of the linear span of $\clp_0$ i.e. $\cls(\clp_0)$ is a maximal abelian $C^*$-subalgebra of $\IM$. In the following proposition we prove a simple result first as a preparation for a little more deeper result that follows next. 

\vsp 
\begin{lem} 
The set $\clp_e =\{p \in \ID^e: p^*=p,\;p^2=p \}$ is a maximal abelian set of projections in $\IM$. Furthermore 
$$\clp_e=\{ p \in \ID^e \bigcap \IM_{\Lambda}: p^*=p,\;p^2=p,\;\mbox{finite subset } \Lambda \subset \IZ \}$$    
\end{lem} 

\vsp 
\begin{proof} Let $q$ be a projection in $\IM$ that commutes with all the element of $\clp_e$. Then $q$ also 
commutes with $\ID^e$ and $\ID^e$ being maximal abelian by Lemma 3.1 we get $q \in \ID^e$. Thus $q \in \clp_e$
 
\vsp 
By Lemma 3.1 (c) $\ID^e \equiv C(\Omega^{\IZ})$ and so a projection $p \in \ID^e$ can be identified with an indicator function of a close set. 
The product topology on $\Omega^{\IZ}$ being compact, any closed set is also compact. Thus 
any close set in $\Omega^{\IZ}$ is a finite union of cylinder sets. So 
$p \in \IM_{\Lambda}$ for a finite subset $\Lambda$ of $\IZ$ depending only on $p$.  
\end{proof}

\vsp 
We use the following notions $S^1=\{z \in \IC: |z|=1 \}$ and $(S^1)^d=S^1 \times S^1 ...\times S^1 (d \mbox{-fold})$ in the text. For an element $\ul{z} \in (S^1)^d$, we define automorphism $\beta_{\ul{z}}$ on $\!M_d(\IC)$ by 
\be 
\beta_{\ul{z}}(x)= D_{\ul{z}}x D_{\ul{z}}^*,
\ee
where $D_{\ul{z}}$ is the diagonal unitary matrix $((\delta^i_jz_i))$. We use notation 
$\beta^{(k)}_{\ul{z}}$ for automorphism acting trivially on $\IM$ except on $\!M_d^{(k)}(\IC)$ as $\beta_{\ul{z}}$. For $\ul{z}_k \in (S^1)^d,\;k \in \IZ$, 
$\beta=\otimes_{k \in \IZ}\beta^{(k)}_{\underline{z}_k}$ is an automorphism 
on $\IM$ and $\beta$ commutes with $\theta$ if $\ul{z}_k=\ul{z}$ for all $k 
\in \IZ$. We use notation $\beta_{\ul{z}}$ for $\otimes \beta^{(k)}_{\ul{z}}$. If
$\ul{z}=(z,z,..,z)$ then we simply use $\beta_z$. Note that $\beta_z=I$ on $\IM$ 
for any $z \in S^1$.  

\vsp 
\begin{lem} 
Let $\beta$ be an automorphism on $\IM$ such that $\beta=I$ on 
$\ID^e$ and $\beta \theta = \theta \beta$ on $\IM$. Then $\beta=\otimes_{k \in \IZ}\beta^{(k)}_{\underline{z}}$ 
on $\IM$ for some $\ul{z} \in (S^1)^d$. 
\end{lem} 

\vsp 
\begin{proof} We fix a finite subset $\Lambda$ of $\IZ$. For any $x \in \IM^e_{\Lambda}=\IM_{\Lambda} \vee \ID^e$, we have $xy=yx$ for all $y \in \ID^e_{\Lambda'}$. Thus 
$$\beta(x)y$$
$$=\beta(x)\beta(y)$$
$$=\beta(xy)$$
$$=\beta(yx)$$
$$=\beta(y)\beta(x)$$
$$=y\beta(x)$$
for all $y \in \ID^e_{\Lambda'}$. 
This shows that $\beta(\IM^e_{\Lambda}) \subseteq \IM^e_{\Lambda}$ since $\IM^e_{\Lambda} = (\ID^e_{\Lambda'})'$ by Lemma 3.1 (d). Furthermore, $\beta$ being an automorphism and 
$\beta^{-1}(x)=x$ for all $x \in \ID^e$, we also have $\beta^{-1}(\IM^e_{\Lambda}) \subseteq \IM^e_{\Lambda}$. Thus we have 
$$\beta(\IM^e_{\Lambda})=\IM^e_{\Lambda}$$ 
In particular, we have 
$$\beta(\IM_{\Lambda}) \vee \ID^e_{\Lambda'}= \IM^e_{\Lambda}$$ 
since $\beta$ fixes all elements in $\ID^e$. The last identity clearly shows that  $\beta(\IM_{\Lambda})=\IM_{\Lambda}$. Since $\beta$ fixes all elements in $\ID^e$ and so in particular all elements in $\ID^e_{\Lambda}$ and $\beta$ commutes with $\theta$, we conclude that 
$$\beta= \beta^{(n)}_{\underline{z}}$$ 
for some $\underline{z} \in (S^1)^d$ on each $\!M^{(n)}_d(\IC),\;n \in \IZ$. Thus $\beta=\otimes_{n \in \IZ}\beta^{(n)}_{\underline{z}}$ on $\IM$ by multiplicative 
property of $\beta$ and universal property of tensor products.     
\end{proof} 

\vsp 
We arrive at the following hyper-rigidity property of maximal abelian 
$C^*$ sub-algebra $\ID^e$.  
\vsp 

\begin{cor}
If actions of two automorphisms $\alpha$ and $\beta$ are equal on $\ID^e$ and commuting with $\theta$ then 
$$\beta  \alpha^{-1} = \otimes_{n \in \IZ} \beta^{(n)}_{\underline{z}}$$
for some $\underline{z} \in (S^1)^d$.
\end{cor} 

\begin{proof} 
The automorphism $\beta^{-1} \alpha$ acts trivially on $\ID^e$ and commutes with $\theta$. Thus
by Lemma 3.3 we get the required result.  
\end{proof} 

\vsp 
We are left to answer a crucial existence question on automorphisms, namely, given an auto-morphism $\beta_0:\ID^e \raro \ID^e$, is there an auto-morphism $\beta:\IM \raro \IM$ extending $\beta_0$? We begin with a lemma to that end. 

\vsp 
\begin{lem} 
Let $\ID \subseteq \IQ$ be unital $C^*$ subalgebras of $\IM$ such that $\ID'=\ID$, $\theta(\ID)=\ID$ and $\theta(\IQ)=\IQ$. Then there exists an automorphism $\alpha$ commuting with $\theta$ such that 
$$(\ID,\theta,\omega_0) \equiv^{\alpha} (\otimes_{n \in \IZ} D_d^{(n)}(\IC),\theta,\omega_0) =(\ID^e,\theta,\omega_0)$$ 
and 
$$(\IQ,\theta,\omega_0) \equiv^{\alpha} (\otimes_{n \in \IZ} \!Q^{(n)},\theta,\omega_0)$$
where $D_d^{(n)}(\IC) \subseteq \!Q^{(n)}$ are isomorphic $C^*$-sub-algebra of $\!M_d^{(n)}(\IC)$ with $\theta(\!Q^{(n)})=\!Q^{(n+1)}$ and $\omega_0$ is the unique normalize trace on $\IM$.
\end{lem} 

\vsp 
\begin{proof} 

\vsp 
Any two maximal abelian $C^*$-subalgberas of an AF algebra are isomorphic by [SS] and so there exists an automorphsim $\beta$ on $\IM$ such that $\beta(\ID^e)=\ID$. As a first step we will show that it is possiblem to construct an automorphism 
$\alpha:\ID^e \raro \ID$ commuting with $\theta$ by modifying $\beta$ once $\theta$ keeps $\ID$ invariant. 

\vsp 
We choose an automorphism $\beta:\IM \raro \IM$ so that $\beta(\ID^e)=\ID$. We now choose a sequence of ablian sub-algberas $\{ A_{[n}:n \in \IZ \}$ of $\ID$ generated by the family of projections $\{\theta^k(\beta(\ID^e_{(0)}):n \le k < \infty \}$. By our construction, we have $A_{[n+1} \subseteq A_{[n}$ for all $n \in \IZ$ and 
each $A_{[n}$ as $C^*$-algebra is isomorphic to $\ID^e_{[n,\infty}$. The inductive limit $C^*$-algebra of 
$A_{[n} \raro^{\theta} A_{[n+1} \subseteq A_{[n}$, say $\IA$ is isomorphic to $\ID^e$ and there is an automorphism $\alpha:\ID^e \raro \IA$ that takes $\ID^e_{\Lambda}$ to $\vee_{k \in \Lambda} \theta^k(\beta(\ID^e_{(0)}))$ commuting with $\theta$. 

\vsp 
We claim that $\IA=\ID$. Suppose not. Then Komogorov-Sinai dynamical entropy of $(\ID,\theta,\omega_0)$ is strictly 
more then Kolmogorov-Sinai dynamical entrop of $(\IA,\theta,\omega_0)$, $ln(d)$. However Connes-St\o rmer dynamical 
entropy of $(\IM,\theta,\omega_0)$ is equal to mean entropy $ln(d)$, where $\omega_0$ is the unique normalised
trace of $\IM$. This brings a contradiction to monotone property of Connes-St\o rmer dynamical entropy since 
$(\ID,\theta,\omega)$ is a $C^*$-dynamical sub-system of $(\IM,\theta,\omega_0)$.  

\vsp 
We aim now to extend automorphim $\alpha:\ID^e \raro \ID$ to $\IM$ commuting with $\theta$. We construct a sequence of increasing finite factors $\{M_n:n \ge 0 \}$ such that $M_n$ contains $B_n=\vee \{\theta^k(\beta(\ID^e_{(0)})):0 \le k \le n \}$ as follows:

\NI For $n=0$, we set $M_0=\beta(\IM_{(0)})$. Suppose we have choosen increasing finite factors 
$M_1,M_2,..., M_k$ so that $B_r \subseteq M_r$ for all $0 \le r \le k$ then we set $M_{k+1}$ to be the mininal factor 
that contains $M_k$ and $\theta(M_k)$. 

\vsp 
We also set $M_{[0}$ as $\vee_{k \ge 0} M_k$. It is clear that $\theta(M_{[0}) \subseteq M_{[0}$ and each $M_k$ being factor, we can verify now that $M_{[0}$ is also a factor as follows: Let $\IE_k$ be the norm one projection from $\IM$ onto $M_k$ preserving 
normalised trace of $\IM$. For any element $x \in M_{[0} \bigcap M'_{[0}$, we have 
$\IE_k(x)y=\IE_k(xy)=\IE_k(yx)=y\IE_k(x)$ 
for all $y \in M_k$ since $x \in M_{[0}' \subseteq M_k'$ as well. 
Thus $\IE_k(x) = \omega_0(x)I$ for each $k \ge 0$. Since $x \in M = \vee_{k \ge 0} M_k$, by taking limit as $n \raro \infty$, we get $\mbox{lim}_{n \raro \infty}\omega_0(\IE_n(x)y)= \omega_0(xy)$ for all $y \in E_k$ and $k \ge 0$ and so $x = \omega_0(x)I$ i.e. $x$ is a scaler multiple of $I$. This proves that $M_{[0}$ is a factor. 

\vsp 
We define by $M_{[n} = \theta^n(M_{[0})$ for all $n \ge 1$ and so 
by our construction we have $\theta(M_{[0}) \subseteq M_{[0}$ and 
$M_{[n+1} \subseteq M_{[n}$ and $A_{[n} \subseteq M_{[n}$ for all $n \ge 0$.  
Let $M$ be the $C^*$ algebra of the inductive limit $M_{[n} \raro ^{\theta} M_{[n+1} \subseteq M_{[n}$ and 
so $M= \vee_{n \in \IZ}M_{[n}$, where $M_{[n}= \theta^n(M_{[0})$ for all $n \in \IZ$. 
Factor property of $M$ follows by the same line of argument used above to show factor property of $M_{[0}$. 
So we have $\theta^n(M_{[0} \vee M_{[0}') = M_{[n} \vee M_{[n}'$ for all $n \in \IZ$. 

\vsp 
By our construction each $M_{[n}$ is a uniformly hyperfinite algebra and so is $M$. So by lemma 2.4, we have 
$M_{[0}''= M_{[0}$. Since $M$ is also a factor that contains a maximal abelian $C^*$-subalgebra $\ID$ of $\IM$, 
we get $M'=\IC$ and so $M = M''=\IM$ by Lemma 2.4. Since $\IM$ is isomorphic to $\IM_{[0}' \otimes \IM_{[0}$, we 
also get $\IM = \IM_{[0}' \vee \IM_{[0}$.    

\vsp 
We use now Proposition 1.1 to find an extension of $\alpha:\ID^e \raro \ID$ to $\alpha:\IM \raro \IM$ commution with $\theta$. We set $\IQ_n =\IQ \bigcap \IM_{[n}$ for all $n \in \IZ$ and verify that $\alpha$ takes $x= \Pi x_n$ in $Q$ in particular 
to $\tilde{x}= \otimes \tilde{x}_n$, where $x_n \in Q_{[n} \bigcap M_{n-1}'$ and $\tilde{x}_n \in \IQ^{(n)} \subseteq 
M_d^{(n)}$.  
 
\end{proof} 

\vsp 
\begin{cor}  
Let $\beta_0$ be an automorphism on $\ID^e$ commuting with $\theta$ then there exists an local automorphism $\beta$ on $\IM$ extending $\beta_0$ commuting with $\theta$. If $\alpha$ is an another automorphism extending $\beta_0$ and commuting with $\theta$ 
then $\beta  \alpha^{-1} = \beta_{\underline{z}}$ for some $\underline{z} \in (S^1)^d$.    
\end{cor}  

\vsp 
\begin{proof} 
For each $n \in \IZ$, we consider the orthogonal projections 
$$f^k_k(n)=\beta_0(|e_k\rangle\langle e_k|^{(n)}):1 \le k \le d$$ 
and the commutative $C^*$ algebra $\clc_0$ generated by elements $\{ \beta_0(|e_k\rangle\langle e_k|^{(n)}):1 \le k \le d,\;n \neq 0 \}$. By Lemma 3.2 the family of projections $(f^k_k(0):1 \le k \le d)$ are elements in $\IM_{\Lambda}$ for some finite subset $\Lambda$ of $\IZ$ and so $\clc'_0 \subseteq \IM_{\Lambda} \vee \ID^e_{\Lambda'}$ 

\vsp 
Furthermore, we find some local automorphism $\alpha_{\Lambda}$ which acts trivially on $\IM_{\Lambda'}$ such that    
$$\alpha_{\Lambda}(\clc_0) = \theta^k(\ID^e_{\IZ_*}),\;\;\IZ_*= \{n \in \IZ, n \neq 0 \}$$ 
for some $k \in \IZ$. By taking the commutant of the equality, we get 
$$\theta_{-k}(\alpha_{\Lambda}(\clc_0')) = \ID^e \vee \!M_d^{(0)}(\IC),$$
where we have used Lemma 3.1 (d). So $\clc_0'$ is isomorphic to $\ID^e \vee \!M_d^{(0)}(\IC)$. 

\vsp 
Thus we can find elements $(f^k_j(0):1 \le k,j \le d)$ in $\clc_0'$ satisfying the matrix relations namely 
$$(f^i_j(0))^*=f^j_i(0): f^i_j(0)f^k_l(0)=\delta^k_jf^i_l(0)$$
extending the isomorphism $\beta_0:\ID^e \raro \ID^e$ to 
$$\beta_0:\ID^e \vee \!M^{(0)}_d(\IC) \raro \clc_0'$$ 
with 
$$\beta_0(e^i_j(0))=f^i_j(0),\;1 \le i,j \le d$$

\vsp 
We set $\beta$ on $\!M_d^{(n)}(\IC)$ by 
$$\beta(x)= \theta^n  \beta_0 \theta^{-n}(x)$$ 
We extend $\beta$ on $\IM_{loc}$ by linearity and $*$-multiplicative property. 
Let $\beta:\IM \raro \IM$ be the unique bounded extension of $\beta:\IM_{loc} \raro \IM$. Thus $\beta$ is a unital $*$-homomorphism.   

\vsp 
Since $\beta_0 \theta = \theta \beta_0$ on $\ID^e$, we verify for $x=\theta^n(x_0)$ with $x_0 \in \ID^e_{\{0\}}$ that 
$$\beta(x) = \theta^n \beta_0  \theta^{-n}  \theta^n(x_0)$$
$$=\theta^n \beta_0(x_0)$$
$$=\beta_0 \theta^n(x_0)$$
$$=\beta_0(x)$$     
Thus the map $\beta$, which is linear and $*$-multiplicative on $\IM$, in particular satisfies $\beta=\beta_0$ on $\ID^e$.  

\vsp 
We need to show $\beta$ is an automorphism on $\IM$. The multiplicative property of $\beta$ on $\IM$ says that $\cln=\{z :\beta(z)=0 \}$ is a two sided ideal of $\IM$. $\IM$ being a 
simple [Pow,ChE,BR1], $C^*$ algebra $\cln$ is either trivial null set or $\IM$. Since $\beta(I)=I$, we conclude that $\cln$ is the trivial null space i.e. $\beta$ is injective. In particular the map $x \raro ||\beta(x)||$ is a $C^*$-norm on $\IM$ where by $*$ homomorphism property of $\beta$ we verify that 
$$||\beta(x)^*\beta(x)||$$
$$=||\beta(x^*)\beta(x)||$$
$$=||\beta(x^*x)||$$  
for all $x \in \IM$.
However, $C^*$ norm being unique on a $*$-algebra if it exists, we get 
$||\beta(x)||=||x||$ for all $x \in \IM$. 

\vsp 
Furthermore, the map $\beta$ being norm preserving, $\beta(\IM)$ is norm closed and thus a $C^*$-sub algebra of $\IM$. Since $\beta \theta = \theta \beta$ on $\IM$, we also have $\theta(\beta(\IM))=\beta(\theta(\IM))=\beta(\IM)$. So the equality $\beta(\IM)=\IM$ 
follows by Lemma 2.4 once we show that $\beta(\IM)'=\IC$.  

\vsp 
Since $\beta(\ID^e)=\beta_0(\ID^e)=\ID^e$ and $\ID^e$ is maximal abelian, an element $x \in \beta(\IM)'$ is also in $\ID^e$ and thus $x=\beta_0(y)$ for some $y \in \ID_e$. Now we verify the following simple equalities for all $z \in \IM$: 
$$\beta(yz-zy)$$
$$=\beta(y)\beta(z)-\beta(z)\beta(y)$$
$$=\beta_0(y)\beta(z)-\beta(z)\beta_0(y)$$
$$=x\beta(z)-\beta(z)x$$
$$=0,$$
where $x \in \beta(\IM)'$. Thus we have $yz-zy=0$ for all $z \in \IM_{\Lambda}$ by the injective property of $\beta$. So we conclude that $y$ is a scalar multiple of identity. This shows $x=\beta_0(y)$ is also a scalar multiple of the identity element of $\IM$ i.e.  
$\beta(\IM)'=\IC$ and so by Lemma 3.4 (b) $\beta(\IM)=\IM$.
\end{proof} 

\section{Connes-St\o rmer Dynamical entropy for $C^*$-dynamical system:} 

\vsp 
In this section we will qickly recall Connes-St\o rmer dynamical entropy for a $C^*$-dynamical system with some additional remarks needed for our purpose. Let $(\cla,\theta,\omega)$ be a $C^*$-dynamical system and $\gamma_i:\clb_i \raro \cla$ be a 
set of channels with mutual entropy 
$$H_{\omega}(\gamma_1,\gamma_2,..,\gamma_n)= \mbox{sup}_{\{\omega_{i_1,i_2,..i_n}:\sum_{i_1,i_2,..,i_n} \omega_{i_1,i_2,..,i_n} = \omega\} } H_{\omega}(\gamma_1,\gamma_2,...\gamma_n,\omega_{i_1,i_2,..,i_n}),$$ 
where 
$$H_{\omega}(\gamma_1,\gamma_2,...\gamma_n,\omega_{i_1,i_2,..,i_n})=
-\omega_{i_1,i_2..,i_n}(1))ln(\omega_{i_1,i_2,...,i_n}(1)) +
\sum_{1 \le k \le n} \sum_{i_k }S(\omega^{(k)}_{i_k} \circ \gamma_k , \omega \circ \gamma_k)
$$  
and $\omega^{(k)}_{i_k}$ is obtained by
taking the sum of the elements $\omega_{i_1,i_2, ...,i_n}$ with the $k$-th index equal to $i_k$. 

\vsp 
Let $h_{\omega}(\gamma,\theta)$ be the entropy of a channel $\gamma:\clb \raro \cla$ given as in definition 3.2.1 of [NS] 
by the limiting value of ${1 \over n+1} H_{\omega}(\gamma,\theta \gamma,..,\theta^n\gamma)$ as $n \raro \infty$.

\begin{thm} 
Let $(\cla_1,\theta_1,\omega_1)$ and $(\cla_2,\theta_2,\omega_2)$ be two $C^*$ dynamical systems. Let $\tau$ be a unital completely positive map from $\cla_1$ to $\pi_{\omega_2}(\cla_2)''$ with $\omega_2 = \omega_1 \tau$ and there exists a sequence of automorphisms $\alpha_n:\cla_1 \raro \cla_2$ satisfying 
\be 
\alpha_n \theta_1= \theta_2 \alpha_n
\ee 
with limiting point as $\tau$ in Arveson's bounded weak topology i.e. $\alpha_n(x) \raro \tau(x)$ as $n \raro \infty$ in bounded weak topology induced by the pair $\cla_1$ and $\pi_{\omega_2}(\cla_2)''$. Similarly let $\eta$ be a unital completely positive map from $\cla_2 \raro \pi_{\omega_1}(\cla_1)''$ with $\omega_1=\omega_2 \eta$ and limiting point of the sequence $\alpha_n^{-1}$ in Arveson's bounded weak topology. Then the following holds:

\vsp 
\NI (a) Unital completely positive maps $\tau:\cla_1 \raro \pi_{\omega_2}(\cla_2)''$ and $\eta:\cla_2 \raro \pi_{\omega_1}(\cla_1)''$ satisfies intertwining relation given in (7).

\vsp 
\NI (b) $h_{\omega_1}(\theta_1)=h_{\omega_2}(\theta_2)$

\end{thm}   

\vsp 
\begin{proof} 
For each $k \ge 1$, by Proposition 3.1.2 (ii) we have
\be 
H_{\omega_2}(\alpha_k \gamma_1, \alpha_k \gamma_2,..., \alpha_k \gamma_n) = H_{\omega_2 \alpha_k}(\gamma_1,\gamma_2,..,\gamma_n)
\ee
The limit supremum of the righthand side in the equality (23) goes to $H_{\omega_1}(\gamma_1,\gamma_2,..,\gamma_2)$ by upper-semicontinuous property 3.1.12 in [NS]. 

\vsp 
We take $\gamma_n = \theta^n \gamma$ for a fixed chennel $\gamma:\clb \raro \cla_1$ and $\epsilon > 0$ to choose $N$ such that
$$h_{\omega_1}(\gamma,\theta_1) - \epsilon \le {1 \over n+1} H_{\omega_1}(\gamma, \theta_1 \gamma, ...\theta_1^n \gamma)$$ 
for all $n \ge N$. 
$$h_{\omega_1}(\gamma,\theta_1) -\epsilon \le {1 \over n+1} (H_{\omega_1}(\gamma, \theta_1 \gamma, ...\theta_1^n \gamma)+\epsilon)$$ 
$$ \le {1 \over n+1} H_{\omega_2 \alpha_k}(\gamma,\theta_1 \gamma,..,\theta_1^n \gamma)$$ 
$$ \le {1 \over n+1} H_{\omega_2}(\alpha_k\gamma, \theta_2 \alpha_k \gamma,...,\theta^n_2 \alpha_k \gamma)$$
for a subsequence of $\alpha_k$ depending on $n$. Now we use Proposition 1.1.11 in [NS] for 
$$h_{\omega_1}(\gamma,\theta_1) -\epsilon \le { 1 \over n+1} H_{\omega_2}(\tau \gamma, \theta_2 \tau \gamma,...\theta_2^n \tau \gamma)$$
where we use finite dimensional property of $\clb$ to ensure bounded weak convergences is equivalent to convergence of 
$\langle f,\alpha_n \gamma (x) g \rangle \raro \langle f, \tau \gamma(x)g \rangle$ for all $x \in \clb$ and $f,g \in \clh_{\omega_2}$, to get $||\tau \gamma- \alpha_n \gamma||_{\omega_2}=\mbox{sup}_{||x|| \le 1} \omega_2((\alpha_n \gamma-\tau \gamma)(x)^*(\alpha_n \gamma- \tau \gamma)(x)) \raro 0$ as $n \raro \infty$ i.e. the convergence required for Proposition 1.1.11 in [NS].  

\vsp 
Thus we have 
$$h_{\omega_1}(\gamma,\theta_1) -\epsilon \le h_{\omega_2}( \tau \gamma , \theta_2) \le h_{\omega_2}(\theta_2)$$
Since $\epsilon$ can be made arbitarilly small, we get $h_{\omega_1}(\gamma,\theta_1) \le h_{\omega_2}(\theta_2)$. This inequality holds for an arbitary channel $\gamma:\clb \raro \cla$ and so $h_{\omega_1}(\theta_1) \le h_{\omega_2}(\theta_2)$. Inter-changing the role of $\theta_1$ and $\theta_2$ in the above argument, we also get $h_{\omega_2}(\theta_2) \le h_{\omega_1}(\theta_1)$. 
\end{proof} 

\vsp 
The set of unital contractive completely positive maps from a $C^*$ algebra $\cla$ to the algebra $\clb(\clh)$ of bounded operators on a Hilbert space $\clh$ is compact in Arveson's bounded weak topology. So for two given $C^*$-dynamics, weak$^*$ isomorphism problem reduces to finding out a sequence of automorphisms $\alpha_n$ satisfying the inter-twinning relation (22).

\vsp
\section{ Abelian partition and Kolmogorov Sinai theorem:}

\vsp 
In this section we review Kolmogorov-Sinai dynamical entropy and its relation with our invariance problem in classical set up.
Let $(\Omega,\clf,\mu)$ be a probability space and $\theta$ be a one-to-one and onto measurable map on 
$\Omega$ so that $\mu \circ \theta =\mu$. In this section we will characterize the maximal class of measurements 
that commute with the class of measurements associated with measurable partitions. We will show although 
the class is larger then that studied by Kolmogorov-Sinai, the dynamical entropy is same.

\vsp
\begin{pro} 
Let $\zeta=(\zeta_i)$ be a family of positive maps so that $\sum_i\zeta_i(1)=1$ and 
$\zeta_i\eta_j=\eta_j\zeta_i$ for any $\eta_j(\psi)=\eta_j\psi$, where $(\eta_j)$ is a partition of the probability 
space $\Omega$ into measurable sets. Then there exists a family of bounded measurable function $\zeta_i$ so that 
$\zeta_i(\psi)=\zeta_i^*\psi\zeta_i$ and $\sum_i\zeta_i^*\zeta_i=1$.
\end{pro} 

\vsp
\begin{proof} 
Since $\mu(\zeta_i(\psi)) \le \mu(\psi)$, there exists a family of bounded measurable functions
$\zeta_i$ so that $\mu(\zeta_i(\psi))=\mu(\zeta^*_i\psi\zeta_i)$ for all bounded measurable function $\psi$. Since 
$\zeta_i$ commutes with $\eta_j(\psi)=\eta_j\psi\eta_j$, we verify that 
$\mu(\eta_j\zeta_i(\psi)\eta_j)=\mu(\zeta_i^*\eta_j\psi\eta_j\zeta_i)$ for any measurable partition $(\eta_j)$.Hence
$\zeta_i(\psi)=\zeta^*_i\psi\zeta_i$. We also note that unless we demand that $\zeta_i$ are non-negative functions
this family of functions are not uniquely determined by the family of positive maps. 
\end{proof} 

\vsp
We denote by $\cll = \{ \zeta: \zeta_i(\psi)=\zeta_i^*\psi\zeta_i,\;\sum_i\zeta_i^*\zeta_i=1 \}.$ For any $\zeta \in
\cll$ we define 
$$H_{\mu}(\zeta)= \mu(\zeta_i(1))S({1 \over \mu(\zeta_i(I)} \mu \circ \zeta_i, \mu)$$
$$=-\sum_i\mu(\zeta^*_i\zeta_i)ln\mu(\zeta^*_i\zeta_i) + \sum_i\mu(\zeta_i^*\zeta_iln\zeta^*_i\zeta_i)$$
where $S$ is the Kullback-Liebler relative entropy and check the following relations as in Kolmogorov-Sinai 
theory hold:

\vsp
\begin{pro} 
For any finite or countable partitions $\zeta,\eta,\beta \in \cll$ following hold:

\NI (a) $H_{\mu}(\zeta \circ \eta) \ge H_{\mu}(\zeta)$

\NI (b) $H_{\mu}(\zeta | \eta \circ \beta) \le H_{\mu}(\zeta|\eta)$

\NI (c) $H_{\mu}(\zeta \circ \eta | \beta) \le H_{\mu}(\zeta | \beta) + H(\eta | \beta)$

\NI (d) $H_{\mu}(\theta(\zeta ) | \theta (\beta)) = H_{\mu}(\zeta | \beta)$
\end{pro}

\vsp
\begin{proof} 
Statement (a) and (b) follows from a more general situation [MO**]. Since $\cll$ is a commutative class, (b) is equivalent
to (c). (d) also follows from the general case. 
\end{proof}

\vsp
We define dynamical entropy $h_{\mu}(\zeta,\theta)$ as in Proposition 5.2
and note in the present case $\zeta$ is an invariant partition for $\mu$. 
   
\vsp
\begin{pro} 
For two finite partitions $\zeta,\eta \in \cll$
$$h_{\mu}(\zeta,\theta) \le h_{\mu}(\eta,\theta) +H_{\mu}(\zeta|\eta)$$
\end{pro}

\vsp
\begin{proof} 
In spirit we follow [Pa]. Take $\zeta_n=\theta^n(\zeta)\circ ...\circ \zeta$ for 
any partition $\zeta$ and verify the following step by Proposition 6.2
$$H_{\mu}(\zeta_n) \le H_{\mu}(\eta_n \circ \zeta_n)$$ 
$$=H_{\mu}(\eta_n) + H_{\mu}(\zeta_n|\eta_n)$$ 
$$\le H_{\mu}(\eta_n)+ \sum_{1 \le k \le n} H_{\mu}(\theta^k(\zeta)|\theta^k(\eta))$$
$$\le H_{\mu}(\eta_n) + nH_{\mu}(\zeta|\eta).$$
Thus the result follows. 
\end{proof}

\bigskip
In case $\eta$ is a simple partition of measurable sets then 
$$H_{\mu}(\zeta|\eta)=
-\sum_{i \in \zeta}\mu(\IE_{\eta}[\zeta_i(1))] ln \IE_{\eta}
[\zeta_i(1)]) + \sum_{i \in \zeta} \mu(\zeta_i(1)ln \zeta_i(1))$$
where $E_{\eta}(\psi)=\sum_j{<\eta_j, \psi>_{\mu} \over
\mu(\eta_j)} \chi_{\eta_j}$ is the conditional expectation on the $\sigma$-field
generated by the partition $\eta$. So given any partition $\zeta \in \cll$
and $\epsilon > 0$ there exists a simple partition $\eta$ so that $H_{\mu}(\zeta|\eta)
\le \epsilon.$ Thus a simple consequence of Proposition 4.3, we have the following result.

\vsp
\begin{thm} 
Let $(\Omega,\clf,\mu)$ be a probability space and
$\theta$ is a one-one and onto measurable $\mu$-invariant map. Then
$$\mbox{sup}_{\zeta \in \cll}h_{\mu}(\zeta,\theta)=\mbox{sup}_{\zeta \in \cll_0}
h_{\mu}(\zeta,\theta).$$
\end{thm} 

\vsp
In the following we show that the dynamical entropy is same even when we restrict to the class $\cll \cap C(\Omega)$ 
or $\cll \cap C^{\infty}(M),$ (in case $\Omega$ is a locally compact topological space ) for a regular measure $\mu$. 
To that end we start with the following proposition where $\cll_0$ is the class of partition of unity into disjoint 
measurable sets $(\zeta_i)$. 

\vsp
\begin{pro} 
Let $\Omega$ be a compact Hausdorff space with a regular measure $\mu$. For any fixed 
partition $\zeta \in \cll_0$ and $\epsilon > 0$ there exists a partition $\eta \in \cll \cap C(\Omega)$ so that
$$H_{\mu}(\zeta|\eta) \le \epsilon$$
\end{pro} 

\vsp
\begin{proof} 
Given an element $\zeta \in \cll_0$ i.e. a partition of unity into disjoint measurable 
sets $(\zeta_i)$, we choose by
regularity of the measure $\mu$ a sequence of  
non-negative continuous functions $(\eta^n_i)$ so that $\mu(\eta^n_i\zeta_j) \raro
\mu(\zeta_i)\delta^i_j$. Now we set $\zeta^n_i={\eta^n_i \over \sum_{i \zeta}\eta_i } $. Thus $\zeta^n$ is 
a partition of unity with elements in $C(\Omega)$.  
Thus $0 \le H_{\mu}(\zeta \circ \zeta^n)-H_{\mu}(\zeta^n) 
\le H^c_{\mu}(\zeta \circ 
\zeta^n) - H^c_{\mu}(\zeta^n) \raro 0 $.
\end{proof}  

\vsp
\begin{thm} 
Let $\Omega$ be a compact Hausdorff space with a 
regular probability measure $\mu$ and $\theta$ be a homeomorphism on 
$\Omega$. Then 
$$h_{\mu}(\theta)=\mbox{sup}_{\zeta \in \cll \cap C(\Omega) } 
h(\zeta,\theta).$$
and $h_{\omega}(\theta)$ is an invariance i.e. 
$$h_{\mu}(\theta) = h_{\mu \alpha^{-1}}(\alpha \theta \alpha^{-1}),$$
where $\alpha:\Omega \raro \Omega'$ is an automorphism from $\Omega$ 
onto $\Omega'$.  

\end{thm}

\vsp
\begin{proof} 
It follows from Proposition 5.5 and the basic inequality 
in Proposition 5.3.
\end{proof} 

\vsp 
Now we have the following obvious invariance theorem crucial for present problem.  

\vsp 
\begin{thm} 
Let $(\Omega_1,\theta_1,\omega_1)$ and $(\Omega_2,\theta_2,\omega_2)$ be two dynamics as in Theorem 4.6 and $\mu_i$ be their associated regular measures such that $\omega_i(f)=\int_{\Omega_i}fd\mu_{\omega_i}$. Let 
$\alpha:C(\Omega_1) \raro L^{\infty}(\Omega_2,d\mu_{\omega_2})$ and $\beta:C(\Omega_2) \raro L^{\infty}(\Omega_1,d\mu_{\omega_1})$ be two unital positive maps satisfying the following relation: 

\vsp 
\NI (a) $\alpha$ (res $\beta$) be bounded weak limit point of a sequence of automorphisms $\alpha_n:\Omega_1 \raro \Omega_2$ (res $\alpha_n^{-1}:\Omega_2 \raro \Omega_1$ ) that inter-twins $\theta_1$ and $\theta_2$ i.e. $\alpha_n \theta_1 = \theta_2 \alpha_n$ ( res $\alpha_n^{-1} \theta_2 = \theta_1 \alpha^{-1}_n$ ) for each $n \ge 1$.  

\vsp  
\NI (b) $$\omega_1= \hat{\omega}_2 \alpha,\;\;\alpha \theta_1 = \hat{\theta}_1 \alpha\; \mbox{on}\; C(\Omega_1)$$

\vsp
\NI (c) $$\omega_2 = \hat{\omega}_1 \beta, \;\; \beta \theta_2 = \hat{\theta}_2 \beta\; \mbox{on}\; C(\Omega_2)$$  

\vsp 
Then 
$$h_{\omega_1}(\theta_1)=h_{\omega_2}(\theta_2)$$
\end{thm} 

\vsp 
\begin{proof} In principle we will follow the proof given for Theorem 4.1 with obvious adaptation. We have 
$$H_{\omega_2}(\alpha_k(\ul{\zeta}) \circ \theta_2(\alpha_k(\ul{\zeta})) \circ ..\theta_2^n(\alpha_k(\ul{\zeta}))) = H_{\omega_2 \alpha_k }(\ul{\zeta} \circ \theta_1(\ul{\zeta}) \circ \theta^n_1(\ul{\zeta}))$$
if $\alpha:\Omega_1 \raro \Omega_2$ were an auto-morphism inter-twinning between $\theta_1$ and $\theta_2$ and $\alpha(\ul{\zeta})=(\alpha \zeta_i \alpha^{-1})$ is a partition in 
$\cla_2$ for a partion $\ul{\zeta}$ for $\cla_1$. 

\vsp 
For a given $\epsilon > 0$, we choose $N$ so that for $n \ge N$ and use continuity in the variable $\omega_2 \omega_n$ in the right hand side to get 
$$h_{\omega_1}(\ul{\zeta},\theta_1)-\epsilon $$
$$\le {1 \over n+1} (H_{\omega_1}(\ul{\zeta} \circ \theta_1(\ul{\zeta})\circ ..\theta_1^n(\ul{\zeta}))+\epsilon)$$
$$ \le {1 \over n+1}H_{\omega_2 \alpha_k}(\ul{\zeta} \circ \theta_1(\ul{\zeta}) ..\circ \theta_1^n(\ul{\zeta}))$$
i.e. 
$$\le {1 \over n+1} H_{\omega_2}( \alpha_k(\ul{\zeta})) \circ \alpha_k(\theta_1(\ul{\zeta}_1)) \circ ..\alpha_k(\theta_1^n(\ul{\zeta}_1)))$$ 
i.e. 
 
$$h_{\omega_2}(\ul{\zeta},\theta_1) -\epsilon \le {1 \over n+1} H_{\omega_2}(\alpha_k(\ul{\zeta}) \circ \theta_2(\alpha_k(\ul{\zeta})) \circ \theta^n_2(\alpha_k(\ul{\zeta})))$$
for a sub-sequence $(\alpha_k)$ depending on $N$ and $n \ge N$ and so limiting value of the right hand side is 
$${1 \over n }H_{\omega_2}(\alpha(\ul{\zeta} \circ \theta_2(\alpha(\ul{\zeta}) ..\circ \theta_2^n\alpha(\ul{\zeta})$$ 
as $\alpha_k \raro \alpha$ in weak-bounded topology.

So 
$$h_{\omega_1}(\ul{\zeta},\theta_1) \le h_{\omega_1}(\alpha(\ul{\zeta}),\theta_2) \le h_{\omega_2}(\theta_2)$$ 
Now we use symmetry of the argument for $h_{\omega_1}(\theta_1) \le h_{\omega_2}(\theta_2)$ and thus equality holds.

\end{proof}

\section{The mean entropy and Connes-St\o rmer dynamical entropy}

\vsp 
\begin{pro} 
Let $\omega$ be a translation invariant state of $\IM$. Then $\omega=\omega \IE^e_{\omega_0}$ if and only if $s(\omega)=h_{KS}(\ID^e,\theta,\omega)$ for some orthonormal basis $e=(e_i)$ of $\IC^d$, where $h_{KS}(\ID^e,\theta,\omega)$ is the Kolmogorov-Sinai dynamical entropy of $(\ID^e,\theta,\omega)$ and $\IE^e_{\omega_0}$ is the norm one projection from $\IM$ on $\ID^e$ preserving the unique normalise trace $\omega_0$ of $\IM$. 
\end{pro} 

\begin{proof} 
We fix a translation invariant state $\omega$ of $\IM$ and consider the state $\omega^e=\omega \IE^e_{\omega_0}$ on $\IM$, where $\IE^e_{\omega_0}$ is the norm one projection from $\IM$ onto $\ID^e$ preserving the unique normalised trace $\omega_0$ of $\IM$. For any finite subset $\Lambda$ of $\IZ$, the restriction of $\IE^e_{\omega_0}$ to $\IM_{\Lambda}$ is also a norm one projection with respect to normalised trace of $\IM_{\Lambda}$ and thus we have 
$$S(\omega \IE^e_{\omega_0}|\IM_{\Lambda})=S(\omega|\IM_{\Lambda})$$ 

\vsp 
We consider the restrictions $\omega_n$ and $\omega^e_n$ of two states $\omega$ and $\omega^e$ respectively to $\IM_{\Lambda_n}$, where $\Lambda_n=\{k: 1 \le k \le n \}$. 
We need to show only `if' part of the statement as `only if' part is obvious since $S(\omega|\IM_{\Lambda_n})=S(\omega \IE^e_{\omega_0} | \IM_{\Lambda_n})$ for each $n \ge 1$ by the remark above. 

\vsp 
We also consider $C^*$ sub-algebras $\IP^e_n=\IM_{\Lambda_{n-1}} \vee \ID^e_{n}$ of $\IM_{\Lambda_n}$. Note that 
$\omega^e_n = \omega^e_n \IE_{\IP^e_n}$, where $\IE_{\IP^e_n}$ is the norm one projection from 
$\IM_{\Lambda_n}$ onto $\IP^e_n$ with respect to the normalised trace. Thus by a basic theorem [OP,NS], we have
$$S(\omega_n,\omega^e_n) = S(\omega_n|\IP^e_n,\omega^e_n|\IP^e_n) +  S(\omega_n, \omega_n \IE_{\IP^e_n})$$  
Since 
$$S(\omega_n|\IP^e_n,\omega^e_n|\IP^e_n) \ge S(\omega_{n-1},\omega^e_{n-1})$$ 
and 
$$S(\omega_n,\omega_n \IE_{\IP^e_n}) \ge S(\omega_n|\!M^{(n)}(\IC),\omega_n \IE_{\IP^e_n}|\!M^{(n)}(\IC))$$ 
by monotonicity property of the relative entropy, we get by the translation invariant property of the states $\omega$ and $\omega^e$ that 
$$S(\omega_n,\omega^e_n) \ge S(\omega_{n-1},\omega^e_{n-1}) + S(\omega_1,\omega^e_1)$$  
By induction on $n \ge 1$, we get 
$$S(\omega_n,\omega^e_n) \ge n S(\omega_1,\omega^e_1)$$
Along the same line of argument, we also get
$$S(\omega_{nk},\omega^e_{nk}) \ge n S(\omega_k,\omega^e_k)$$
However we have  
$$S(\omega_n,\omega^e_n) = -tr_0(\hat{\omega}_n (ln(\hat{\omega}_n) - ln(\hat{\omega}^e_n))$$ 
$$=S(\omega_n)-tr_0(\IE_{\omega_0}^e(\hat{\omega}_n)ln(\hat{\omega}^e_n))$$
(by bi-module property) 
$$=S(\omega_n)-S(\omega^e_n)$$ 
for each $n \ge 1$ and 
$$\mbox{lim}_{n \raro \infty}{1 \over n}S(\omega^e_n)=h_{\omega}(\ID^e,\theta)$$ 

\vsp 
Thus $0 \le {1 \over k} S(\omega_k,\omega^e_k) \le { 1 \over nk} (S(\omega^e_{nk})-S(\omega_{nk}))$ for all $k \ge 1$ and the equality $s(\omega)=h_{\omega}(\ID^e,\theta)$ implies that $S(\omega_k,\omega^e_k)=0$ and so $\omega_k= \omega^e_k$ for all $k \ge 1$. Thus we have $\omega_R = \omega^e_R$ on $\IM_R$. Similarly we also have 
$\omega =\omega^e$ on $\theta^n(\IM_R)$ for all $n \in \IZ$. By taking inductive 
limit of these two family of states to $-\infty$, we complete the proof. 
\end{proof} 

\vsp 
\begin{pro} 
Let $\omega$ be a translation invariant state $\omega$ of $\IM$ and $\IE^e$ be a norm one projection from $\IM$ on $\ID^e$ such that $\omega = \omega \IE^e$. Then 
$s(\omega)$, $h_{CS}(\IM,\theta,\omega)$ and $h_{KS}(\ID^e,\theta,\omega)$ are equal.
\end{pro} 

\begin{proof} 
For a translation invariant state $\omega$ of $\IM$, we always have $h_{CS}(\IM,\theta,\omega) \le s(\omega)$ [NS]. Furthermore,  
$$S_{\Lambda_n} \le \sum_{i_1,..,i_n} -\omega(\zeta_{i_1,..i_n})ln(\omega(\zeta_{i_1,..,i_n}),$$ 
where $\zeta_{i_1,..,i_n} = \zeta_{i_1} \theta(\zeta_{i_2}) ..\theta^{n-1}(\zeta_{i_n})$ and $\zeta=(\zeta_i)$ is the partition of unity with 
the orthogonal projections $\zeta_i=|e_i \rangle \langle e_i|^{(1)},\;1 \le i \le d$ in $\IM^{(1)}$. Thus $s(\omega) \le h_{KS}(\ID^e,\theta,\omega)$. However, since $\omega = \omega \circ \IE^e$, Theorem 3.2.2 in [NS] also says that 
$$h_{KS}(\ID^e,\theta,\omega) \le h_{CS}(\IM,\theta,\omega)$$ 
Thus we get the required equality. 
\end{proof}  

\vsp 
We recall here from [CT], a normal weight $\phi$ on a properly infinite von-Neumann algebra 
$\clm$ is called {\it integrable } if the set of integrable elememts 
$\clp_{\phi} = \{ X \in \clm: \int_{\IR} \sigma^{\phi}_t(X) dt \in \clm \}$ is weak$^*$-dense in $\clm$, where the integral 
is also defined in weak$^*$ topology. For an integrable element $X \in \clp_{\phi}$ of a normal faithful weight $\phi$, analytic element $\int_{\IR} f(t)\sigma_t^{\phi}(X)dt \in \clm$ is also in $\clp_{\phi}$, where $f$ is any element in 
$L^1(\IR)$. So $\clp_{\phi} \bigcap \cln_{\phi}$ is weak$^*$ dense in $\clm$, if $\cln_{\phi}$ is a weak$^*$ dense sub-algbera of analytic elements of $(\sigma^{\phi}_t)$ in $\clm$. Futhermore, there exists a weak$^*$-dense sublagebra of $\clm$ contained in $\clp_{\phi} \bigcap \cln_{\phi}$ such that the normal $\phi$-preserving norm one projection map 
$\IE_{\phi}$ onto 
$\clm_{\sigma^{\phi}}$ is given by 
$$\IE_{\phi}(X)= \int_{\IR} \sigma^{\phi}_t(X)dt$$
for all $X \in \cln_{\phi}$. 

\vsp 
\begin{lem} 
Let $\phi$ be a normal faithful weight on a von-Neumann algebra $\clm$ with infinite multiplicity then $\clm_{\sigma^{\phi}}' \bigcap \clm \subseteq \clm_{\sigma^{\phi}}$;  
\end{lem}

\vsp 
\begin{proof} 
By Theorem 5.1 in [CT], the statement in Lemma 6.3 is true for any integrable normal weight on a von-Neumann algebra $\clm$. For the general case, we will use approximation method used in Theorem 4.7 in [CT] for normal faithful weight of infinite multipiclity. As in Theorem 4.7 in [CT], we choose a self-adjoint element $h \in \clm_{\phi}$ with absolutely continuous spectrum in commutant of a type-$I_{\infty}$ sub-factor $F_{\infty}$ of $\clm_{\phi}$ so that $F_{\infty} \bigcap \{h \}'$ is properly infinite. For each $n \ge 1$, we choose self-ajoint element affiliated to $\{h \}''$ so that 
$$-{1 \over n} < h_n-I < {1 \over n}$$
and consider the sequence of faithful nornal integrable weights $\phi_n$ defined on $\clm_{\phi}$ by $\phi_{h_n}(Y)=\phi(h_nY)$ for all $Y \in \clm_{\phi}$. We extend these weights to $\clm$ defined by $\phi_n:X \raro \phi_{h_n} \circ \IE_{\phi}(X)$ for all $X \in \clm$. 

\vsp 
So by Connes cocycle theorem [Con,CT], we have 
\be 
\sigma_t^{\phi_n}(X) = h_n^{it} \sigma_t^{\phi}(X)h_n^{-it}
\ee
and invariant algebra $\clm_{\phi_n}$ is independent of $n \ge 1$ and thus each $\clm_{\phi_n} = \clm_{\phi} \bigcap \{ h \}'$ since $h^{it}_n \raro I$ in strong operator topology as $n \raro \infty$. 

\vsp 
So it is easy to check the following set-theoretic inclusions since $\clm_{\phi_n} \subseteq \clm_{\phi}$ and each $\phi_n$ is integrable:  
$$\clm_{\phi}' \bigcap \clm \subseteq \clm'_{\phi_n} \bigcap \clm \subseteq \clm_{\phi_n} \subseteq \clm_{\phi}$$
\end{proof} 

\vsp 
\begin{cor} 
Let $\phi$ be a faithful normal weight on a properly infinite von-Neumann algebra $\clm$ then $$\clm'_{\sigma^{\phi}} \bigcap \clm \subseteq \clm_{\sigma^{\phi}}$$
\end{cor}

\vsp 
\begin{proof}  
Let $F_{\infty}$ be a type-$I_{\infty}$ factor. Then $\phi \otimes tr$ is a normal faithful weight on $\clm \otimes F_{\infty}$, so by Proposition 6.3, we have 
$(\clm_{\sigma^{\phi}}' \bigcap \clm) \otimes I \subseteq (\clm_{\sigma^{\phi}} \otimes F_{\infty})' \bigcap \clm \otimes F_{\infty} \subseteq \clm_{\sigma^{\phi}} \otimes F_{\infty}$. This clearly shows the required inclusion. 
\end{proof}  

\vsp 
\begin{pro} 
Let $\omega$ be a faithful state of $\IM$ and $(\clh_{\omega},\pi_{\omega},\zeta_{\omega})$ be its GNS space with von-Neumann algebra $\clm=\pi_{\omega}(\IM)''$. Let $\phi$ be a normal faithful weight on $\clm$ and $(\sigma_t^{\phi})$ be its
modular group on $\clm$. Then the following hold:

\vsp 
\NI (a) There exists a unique group of automorphisms $(\hat{\sigma}^{\phi}_t)$ on $\IM$ such that 
$$\sigma^{\phi}_t(\pi_{\omega}(x))=\pi_{\omega}(\hat{\phi}^{\phi}_t(x))$$  
for all $t \in \IR$ and $x \in \IM$; 

\vsp 
\NI (b) The centralizer $\clc_{\phi} = \{a \in \clm : \phi(ab)=\phi(ba),\;\forall b \in \clm \}$ of $\phi$ in $\clm$ is equal to $\clm_{\sigma^{\phi}}$, where $\clm_{\sigma^{\phi}} = \{a \in \clm: \sigma_t^{\phi}(a)=a, \; \forall t \in \IR \}$ is the von-Neumann subalgebra of invariant elements of the modular automorphisms $(\sigma_t^{\phi})$. 

\vsp 
\NI (c) The centralizer $\IC_{\phi} = \{x \in \IM : \phi(\pi_{\omega}(xy))=\phi(\pi_{\omega}(yx)),\;\forall y \in \IM \}$ of $\phi$ in $\IM$ is equal to $\IM_{\hat{\sigma}^{\phi}}$, $\IM_{\hat{\sigma}^{\phi}} = \{ x \in \IM: \hat{\sigma}^{\phi}_t(x)=x \}$ is the $C^*$-subalgebra of invariant elements of the modular automorphisms $(\hat{\sigma}_t^{\phi})$. 

\vsp 
\NI (d) $\clm_{\sigma^{\phi}} = \pi_{\omega}(\IM_{\hat{\sigma}^{\phi}})''$; 

\vsp 
\NI (e) Let $\IZ_{\omega}$ be the centre of $\IM_{\hat{\sigma}^{\omega}}$ and $\clz_{\omega}$ be the centre of $\clm_{\sigma^{\omega}}$ then  

\vsp 
\NI (i) $\clm_{\sigma^{\omega}} = \pi_{\omega}(\IM_{\hat{\sigma}^{\omega}})''$; 

\vsp 
\NI (ii) $\clm_{\sigma^{\omega}}' \bigcap \pi_{\omega}(\IM)'' = \clz_{\omega};$ 

\vsp 
\NI (iii) $\IM_{\hat{\sigma}^{\omega}}'=\IZ_{\omega}$;

\end{pro} 

\vsp 
\begin{proof} 
For each fixed $t \in \IR$, $\pi_t:\IM \raro \clb(\clh_{\omega})$ is a unital representation defined by 
$$\pi_t(x) = \sigma^{\phi}_t(\pi_{\omega}(x))$$ 
with $\pi_t(\IM)''=\pi_{\omega}(\IM)''$. Thus by a Theorem of R.T. Powers (Theorem 3.7 in [Pow]), we get 
$$\pi_t(x)=\pi_{\omega}(\hat{\sigma}^{\phi}_t(x))$$ 
for some automorphism $\hat{\sigma}^{\phi}_t:\IM \raro \IM$. For each $t \in \IR$, $\hat{\sigma}^{\phi}_t$ is uniquely determined by $\sigma^{\phi}_t$ and the map $t \raro \hat{\sigma}^{\phi}_t$ is a group of automorphisms, are consequence of the faithful property of the representation $\pi$ ($\IM$ being a simple C$^*$ algebra, 
any non-degenerate representation of $\IM$ is faithful ). 

\vsp 
Statements (a),(b) and (c) are simple consequence of the general Tomita-Takesaki theory [BR1]. The statement (d) needs a careful argument and it is far from being obvious.  

\vsp 
For the time being we assume that $\phi$ is an integrable faithful normal weight with infinite mutiplicity. As in Lemma 2.3 in [CT], we set $\clm_{\phi,\lambda} = \{ X  \in \clm: \sigma^{\phi}_t(X) = \lambda^{it} X \}$ for 
$\lambda > 0$ and use Fourier inversion formula given in Lemma 2.3 (c) to conclude $\pi_{\omega}(\IM_{\hat{\sigma}^{\phi}})''$ is equal to $\clm_{\sigma^{\phi}}$ since $\pi_{\omega}(\IM)$ is weak$^*$ dense in $\clm$. This shows that the statement (d) is true for an integrable normal faithful weight with infinite mutiplicity.   

\vsp 
For the general case, we choose a suitable state $\omega'$ on $\IM$ in order to guarantee existence of an integrable normal weight $\psi$ on $\cln=\pi_{\omega'}(\IM)''$. We consider the integrable normal weight  $\phi \otimes \psi$ on $\clm \otimes \cln$ with infinite multiplicity. By the first part of present argument 
$(\clm \otimes \cln)_{\phi \otimes \psi}=\pi_{\omega \otimes \omega'}((\IM \otimes \IM)_{\hat{\sigma}^{\phi} \otimes \hat{\sigma}^{\psi}})''$. We use conditional expectation $\IE^{\phi \otimes \psi}_{\phi}: \clm \otimes \cln \raro \clm$ on left von-Neumann algebra preserving their respective normal weights and its normal property to get the required equality (d).   

\vsp 
For (i) of (e), we use (d). For (ii) of (e), it is enough to show for factor states $\omega$ of $\IM$ as we can use central decomposition of more general state $\omega$ into factor states to get the required equality. If $\omega$ is a type-I then modular group is inner and thus centre of $\clm_{\sigma^{\omega}}$ is equal to its commutant. For type-III, the result follows from Corollary 6.4. So we are left to prove the statement for type-II factor state $\omega$. The statement is true if $\omega$ is the unique normalised trace $\omega_0$ on $\IM$. However $\pi_{\omega}(\IM)''$ is unitarily equivalent to either $\pi_{\omega_0}(\IM)''$ or $\pi_{\omega_0}(\IM)'' \otimes F_{\infty}$ for a type-$I_{\infty}$ factor $F_{\infty}$. So (ii) 
holds as well for $\omega$ since modular group of $\omega$ is as well inner for semi-finite von-Neumann factors.   

\vsp 
The last statement (iii) of (e) follows trivially from (d) and (ii) of (e).

\end{proof}

\vsp 
A translation invariant state $\omega$ is called {\it classical } if $\omega = \omega \circ \IE^e_{\omega_0}$, where $\IE^e_{\omega_0}$ is the normalized trace $\omega_0$ preserving norm one projection from $\IM$ onto $\ID^e$ for some orthonormal basis $e=(e_i)$ of $\IC^d$. The set $\clc^e_{\theta}=\{\omega \in \IM^*_{+,1}: \omega = \omega \IE^e_{\omega_0} \}$ of classical translation invariant states form a compact convex subset of the compact convex set of translation invariant states. We aim now to prove any translation invariant state of $\IM$ is a classical state modulo an automorphism on $\IM$ commuting with $\theta$. 

\vsp 
\begin{pro} 
Let $\omega$ in Propostion 6.3 be also translation invariant. Then the following holds:

\vsp 
\NI (a) The von-Neumann sub-algebra $\clm_{\sigma^{\omega}}$ of $\clm$ is also $\theta$-invariant i.e.  
$\theta(\clm_{\sigma^{\omega}})=\clm_{\sigma^{\omega}}$ and $\pi_{\omega}(\IM_{\hat{\sigma}^{\omega}})''
=\clm_{\sigma^{\omega}}$; 

\vsp 
\NI (b) The $C^*$ sub-algebra $\IM_{\hat{\sigma}^{\phi}}$ of $\IM$ is also $\theta$-invariant i.e.  
$\theta(\IM_{\hat{\sigma}^{\omega}})=\IM_{\hat{\sigma}^{\omega}}$ and 
$$\theta (\hat{\sigma}^{\omega}_t(x)) = \hat{\sigma}^{\omega}_t (\theta(x))$$
for all $x \in \IM$; 

\vsp 
\NI (c) There exists a maximal abelian $C^*$-sub-algebra $\ID_{\omega}$ of $\IM$ such that 

\NI (i) $\IZ_{\omega} \subseteq \ID_{\omega} \subseteq \IM_{\hat{\sigma}^{\omega}}$;

\NI (ii) $\theta(\ID_{\omega})=\ID_{\omega}$;

\NI (iii) There exists an automorphism $\alpha$ on $\IM$ such that $\theta \alpha = \alpha \theta$ and 
$\alpha(\ID_{\omega})=\ID^e$;

\vsp 
\NI (d) Let $\omega$ be a translation invariant state and $0 < \lambda < 1$ then the centraliser $\IC_{\omega} = \{x \in \IM: \omega(xy)=\omega(yx),\forall y \in \IM \}$ of $\omega$ is equal to the centraliser $\IC_{\omega_{\lambda}} = \{a \in \IM: \omega_{\lambda}(xy)=\omega_{\lambda}(yx),\forall y \in \IM \}$ of $\omega_{\lambda} =\lambda \omega + (1-\lambda) \omega_0$, where $\omega_0$ is the unique normalized trace on $\IM$. Furthermore, there exists a maximal abelian $C^*$-sub-algebra $\ID_{\omega}$ of $\IM$ such that 

\vsp 
\NI (i) $\IZ_{\omega} \subseteq \ID_{\omega} \subseteq \IC_{\omega}$, where $\IZ_{\omega}$ is the centre of $\IC_{\omega}$;

\vsp 
\NI (ii) $\theta(\ID_{\omega})=\ID_{\omega}$;

\vsp 
\NI (iii) There exists an automorphism $\alpha$ on $\IM$ such that $\theta \alpha = \alpha \theta$ and 
$\alpha(\ID_{\omega})=\ID^e$;
\end{pro} 

\begin{proof} 
That $\theta \hat{\sigma}^{\omega}_t=\hat{\sigma}^{\omega}_t \theta$ also follows by the faithful property of $\pi_{\omega}$ and commuting property $\Theta_{\omega} \sigma^{\omega}_t = \sigma^{\omega}_t \Theta_{\omega}$ on $\clm_{\omega}$. So (a) and (b) are simple consequence of faithful property of the reprsentation $\pi_{\omega}$ of $\IM$.   

\vsp 
That $\clc_{\omega} = \clm_{\sigma^{\omega}}$ in (e) follows by a standard result in Tomita-Takasaki theory [BR1]. The invariance property of $\clm_{\sigma^{\omega}}$ under $\theta$ follows since $\omega$ is $\theta$ invariant and thus commutes with the modular automorphism $\sigma_t^{\omega}$ on $\clm$. That $\IC_{\omega}=\IM_{\hat{\sigma}^{\omega}}$ in (c) by (a) once we apply (b) with $a=\pi_{\omega}(x)$ and $b=\pi_{\omega}(y)$ to conclude $x \in \IC_{\omega}$ if and only if $a \in \clc_{\omega}$.

\vsp 
Since $\omega \theta = \omega$, $\Theta_{\omega} \sigma^{\omega}_t = \sigma^{\omega}_t \Theta_{\omega}$ by the uniquess of KMS relation for $\omega$ and so $\theta \hat{\sigma}^{\omega}_t = \hat{\sigma}^{\omega}_t \theta$ for all $t \in \IR$. In particular, $\theta(\IM_{\hat{\sigma}^{\omega}})=\IM_{\hat{\sigma}^{\omega}}$.

\vsp 
Let $\cld_{\omega}$ be the non-empty collection of abelian sub-algebras $D$ of $\IM$ such that $\theta(D)=D$ and $\IZ_{\hat{\sigma}^{\omega}} \subseteq D \subseteq \IM_{\hat{\sigma}^{\omega}}$. We give a partial ordering $D_1 \succeq D_2$ if $D_2 \subseteq D_1$. For a maximal ordered subcollection $D_{\alpha}$, closer of their union $\bigcup U_{\alpha}$ is an upper bound for the collection and thus by Zorn's lemma there exists a maximal element say $\ID^{\omega}$ in $\cld_{\omega}$.

\vsp 
Let $\omega_{\lambda}=\lambda \omega +(1-\lambda)\omega_0$ for $\lambda \in [0,1]$. Then $\IM_{\hat{\sigma}^{\omega_{\lambda}}}=\IM_{\hat{\sigma}^{\omega}}$ for $\lambda \in (0,1]$ and thus $\ID^{\omega}$ is a maximal element in $\cld_{\omega_{\lambda}}$ as well. We claim that $\ID_{\omega}$ is a maximal element in $\cld_{\omega_0}$ as well. For a proof, let $D$ be an element 
in $\cld_{\omega_0}$ such that $\ID^{\omega} \subseteq D$. All that we need to show $D \subseteq \IM_{\hat{\sigma}^{\omega}}$ and that follows as any element $x \in D$ commutes with $\ID^{\omega}$ and so atleast commutes with $\IZ_{\hat{\sigma}_{\omega}}$ but $\IZ'_{\hat{\sigma}^{\omega}}=\IM_{\hat{\sigma}^{\omega}}$ and thus $x \in \IM_{\hat{\sigma}^{\omega}}$. 

\vsp 
So for the general case, we need the following lemma.

\vsp 
\begin{lem} The following statements are true:

\NI (a) Let $\clc_{\hat{\theta}}$ be the collection of translation invariant abelian von-Neumann subalgera of $\pi_{\omega_0}(\IM)''$. Then an element $\cln \in \clc_{\hat{\theta}}$ is maximal
if and only if $\cln$ is a maximal abelian von-Neumann subalgebra of $\pi_{\omega_0}(\IM)''$ i.e. $\cln = \cln' \bigcap \pi_{\omega_0}(\IM)''$. 

\NI (b) Let $\cln$ be a maximal abelian von-Neumann sub-algebra of $\pi_{\omega_0}(\IM)''$ then there exsits a maximal abelian $C^*$-subalgebra $\ID_{\cln}$ of $\IM$ so that $\cln = \pi_{\omega_0}(\ID_{\cln})''$.

\vsp 
\NI (c) Let $\cld$ be the collection of all translation invariant abelian $C^*$-sub-algebras of $\IM$. Then an element 
$\ID \in \cld$ is maximal in $\cld$ if and only if $\ID$ is maximal abelian in $\IM$ i.e. $\ID' = \ID$.  
\end{lem}

\vsp 
\begin{proof} 

\vsp 
We write $tr_0:\IM_d \raro \IC$ for normalised tracial state of $\IM_d$ with respect to an othernormal basis $(e_i)$ for 
$\IC^d$ by 
$$tr_0(a) = \sum_i{ 1 \over n} <e_iae_i>$$ 
We consider the GNS space $(\clh_{tr_0},\pi_{tr_0},\zeta_{tr_0})$ of $(\IM_d,tr_0)$, i.e. $\clh_{tr_0}$ is isomorphic to 
$\IC^d \otimes \IC^d$ and $\pi_{tr_0}(a)=a \otimes I_d$ with Tomita's conjugation 
$\clj_{tr_0}(a)= I_d \otimes a^*$ for all $a \in \IM_d$. Thus $\pi_{tr_0}(\IM_d)$ and $\pi_{tr_0}(\IM_d)'$ are isomorphic 
to $\IM_d \otimes I_d$ and $I \otimes \IM_d$ respectively and the vector state $\zeta_{tr_0}$ is given by the unite vector $\zeta_{tr_0} = \sum_{i,j} {1 \over n} e_i \otimes e_j$. 

\vsp 
More generally, for a finite subset $\Lambda$ of $\IZ$, GNS space $(\clh_{\omega_{\Lambda}},\pi_{\omega_{0,\Lambda}},\zeta_{\omega_{0,\Lambda}}$ of $\IM_{\Lambda}$ associated with the state $\omega_{0,\Lambda}=\omega_0|\IM_{\Lambda}$ is given by 
$$(\otimes_{k \in \Lambda} \clh^{(k)}_{tr_0},\otimes_{k \in \Lambda} \pi^{(k)}_{tr_0},\otimes_{k \in \Lambda} \zeta^{(k)}_{tr_0})$$ 
with $\pi_{\omega_{0,\Lambda}}(x) = x \otimes I$ for all $x \in \IM_{\Lambda}$. There exists a pure translation invariant state $\zeta_{\omega_0}$ of $\IM \otimes \IM$ such that its marginals are the unique normalised tracial state $\omega_0$ of $\IM$.

\vsp 
We consider the GNS space $(\clh_{\omega_0},\pi_{\omega_0},\zeta_{\omega_0})$ of $(\IM,\omega_0)$, the cyclic representation determined uniquely upto isomorphism. The state $\omega_0$ being a factor, we have $\pi_{\omega_0}(\IM)'' \vee \pi_{\omega_0}(\IM)' = \clb(\clh_{\omega_0})$. Thus there exists a natural Hilbert space isomorphism between the GNS Hilbert space 
associated with the state $\omega_0$ of $\IM$ to GNS Hilbert space of the pure state $\zeta_{\omega_0}$ of $\IM \otimes \IM$  and the isomorphism takes $\pi_{\omega_0}(x)$ to $x \otimes I$ and $\clj_{\omega_0}(\pi_{\omega_0}(x))$ to $I \otimes x^*$ respectively.  

\vsp 
We set a notation $\hat{\theta}$ for the cannonical translation action on $\clb(\clh_{\omega_0})$ defined by $\hat{\theta}(X)=U_{\theta}XU_{\theta}^*$, where $U_{\theta}$ is the unitary operator defined by extending the action $U_{\theta} \pi_{\omega_0}(x) \zeta_{\omega} = \pi_{\omega_0}(\theta(x))\zeta_{\zeta}$ for all $x \in \IM$. 

\vsp 
\NI (a) Let $\cln$ be a $\theta$ invariant abelian von-Neumann subalgebra of $\pi_{\omega_0}(\IM)''$ and $\cln_{ma}$ be a
maximal abelian von-Neumann subalgebra of $\pi_{\omega}(\IM)''$ that contains 
$\cln$. So it satisfies
$$\cln \subseteq \bigcap_{n \in \IZ} \theta^n(\cln_{ma})''  \subseteq \cln_{ma} = \cln_{ma}' \bigcap \pi_{\omega_0}(\IM)''$$
$$ \subseteq \cln_{ma}' \bigcap \vee_{n \in \IZ} \theta^n(\clm_{ma})'' \subseteq \cln' \bigcap \pi_{\omega_0}(\IM)'' \subseteq \cln'$$

\vsp 
Then the following holds: 

\vsp 
\NI (i) $\cln$ is maximal in $\clc_{\hat{{\theta}}}$ if and only if $\cln=\bigcap_{n \in \IZ} \theta^n(\cln_{ma})$ for any maximal abelian von-Neumann subalgebra $\cln_{ma}$ of $\pi_{\omega_0}(\IM)''$ that contains $\cln$; 

\vsp 
\NI (ii) If $\cln$ is maximal in $\clc_{\hat{\theta}}$ then $\cln \vee \clj_{\omega_0}(\cln)$ is maximal in 
$\{A \subseteq \clb(\clh_{\omega}) : A \subseteq A',\;\hat{\theta}(A)=A \}$; 

\vsp 
A simple proof follows once we identify $\cln \vee \clj_{\omega_0}(\cln)$ as tensor product of two abelian von-Neumann algebras
in $\pi_{\zeta_{\omega_0}}(\IM \otimes \IM)'' = \clb(\clh_{\omega_0})$.    

\vsp 
\NI (iii) If $\cln_{ma}$ is a maximal abelian von-Neumann subalgebra of $\pi_{\omega_0}(\IM)''$ then 
$\cln_{ma} \vee \clj_{\omega_0}(\cln_{ma})$ is maximal abelian in $\clb(\clh_{\omega_0})$; 

\vsp 
The abelian von-Neumann algebra $\cln_{ma} \vee \clj_{\omega_0}(\cln_{ma})$ is isomorphic to tensor product of two abelian von-Neumann algebras in $\pi_{\omega_0}(\IM \otimes \IM)''=\clb(\clh_{\omega_0})$ and thus its maximal abelian property follows from maximal abelian property of its marginals.  

\vsp 
\NI (iv) The projection $P_{\omega_0} = [\cln \vee \clj_{\omega_0}(\cln) \zeta_{\omega_0}]$ is $\hat{\theta}$ invariant and so by maximal property of $\cln \vee \clj_{\omega_0}(\cln)$ in $\clb(\clh_{\omega_0})$, we get $P_{\zeta} \in \cln \vee 
\clj_{\omega_0}(\cln)$. 

\vsp 
Let $\cln_{ma}$ be a maximal abelian-subalgebra of $\pi_{\omega_0}(\IM)''$ containing $\cln$. Then 
$\cln_{ma} \vee \clj_{\omega_0}(\cln_{ma})$ also contains $\cln \vee \clj_{\omega_0}(\cln)$ and so $P_{\zeta_{\omega_0}}$ is 
also an element in $\cln_{ma} \vee \clj_{\omega_0}(\cln_{ma})$ and $P_{\zeta_{\omega_0}}\zeta_{\omega_0} = \zeta_{\omega_0}$. 
So in particular, $P_{\zeta_{\omega_0}}=[\cln_{ma} \vee \clj_{\omega_0}(\cln_{ma})\zeta_{\omega_0}]$ i.e. 

\vsp 
\NI (v) $[\cln \zeta_{\omega_0}] = [\cln_{ma}\zeta_{\omega_0}]$; 

\vsp 
\NI (vi) $\cln = \cln_{ma}$; 

\vsp 
It follows by separating property of $\zeta_{\omega_0}$ for $\pi_{\omega_0}(\IM)''$ as follows: By (v) for a given element $x \in \clm_{ma}$ and $\epsilon > 0$, there exists a $y \in \cln$ such that $||(x-y)\zeta_{\omega_0}|| \le \epsilon$. Thus $||(x-y)z'\zeta_{\omega_0}|| \le \epsilon ||z'||$ for all $z' \in \pi_{\omega_0}(\IM)'$. This shows $x$ is an element in strong limit of elements in $\cln$. Since $\cln$ is closed in strong operator topology, we get $x \in \cln$. 

\vsp 
For (b) we use uniqueness of hyperfinite type-$II_1$ factor and identification 
$\pi_{\omega_0}(\IM)''$ with $\pi_{\omega_0}(\IM)'' \otimes \IM_d$ for existence of
an automorphism $\alpha$ on $\IM$ so that $\cln=\pi_{\omega_0}(\alpha(\ID^e)''$ by a theorem of R.T. Powers [Pow]. 
 
\vsp 
The last statement is now obvious from (a) and (b). 

\end{proof}

\vsp 
The statement (d) in Proposition 6.6 is a simple application of Proposition 5.3.28 in [BR-II] valid for $\sigma^{\omega}$-KMS state of $\IM$ since $\IC_{\omega}=\IM_{\hat{\sigma}^{\omega_{\lambda}}}$ for all $0 < \lambda < 1$, independent of $\lambda \in (0,1)$.
\end{proof} 

\vsp 
\begin{pro} 
Let $\omega$ be a translation invariant state of $\IM$. Then there exists an automorphism $\alpha$ on $\IM$ such that 

\NI (a) $\alpha \theta = \theta \alpha;$

\NI (b) $\omega \alpha = \omega \alpha \IE^e_{\omega_0}.$
\end{pro} 

\begin{proof} 
By Proposition 6.6 (d), we have an automorphism $\alpha$ on $\IM$ such that 
$\alpha \theta = \theta \alpha$ and $\omega \alpha = \omega \alpha \IE^e$, where $\IE^e$ 
is a norm one projection from $\IM$ onto $\ID^e$. Thus by Proposition 6.1 and 
Proposition 6.2, we also have 
$\omega \alpha = \omega \alpha \IE^e_{\omega_0}$.  
\end{proof}   

\vsp 
\begin{thm} 
Let $\omega$ be a translation invariant pure state of $\IM$. Then $(\IM,\theta,\omega)$ is isomorphic to $(\IM,\theta,\omega_{\lambda})$, where $\omega_{\lambda}$ is an infinite tensor product pure state of $\IM$. 
\end{thm} 

\vsp 
\begin{proof} By Proposition 6.6 (d), we have $\omega \alpha = \omega \alpha \IE^e_{\omega_0}$ on $\IM$ for some automorphism $\alpha$ on $\IM$. Thus the state $\omega \alpha$ is also pure. We claim that its restriction to $\ID^e$ is also pure. If $\omega \alpha = \mu \omega_1 + (1-\mu)\omega_2$ on $\ID^e$ for some $\mu \in (0,1)$, then we have 
$$\omega = \omega \alpha \IE^e_{\omega_0} $$
$$= \mu \omega_1 \IE^e_{\omega_0} + (1-\mu)\omega_2 \IE^e_{\omega_0}$$ 
on $\IM$. Thus by the purity of $\omega$,  
$\omega_1 \IE^e_{\omega_0} = \omega_2 \IE^e_{\omega_0} = \omega \alpha$ on $\IM$. 
Thus $\omega_1 = \omega_2 = \omega \alpha$ on $\ID^e$. 

\vsp 
The state $\omega \alpha$ being pure on $\ID^e$, the measure is supported onto a singleton point in $\Omega^{\IZ}$, where $\ID^e \equiv C(\Omega^{\IZ})$. Furthermore, the state being translation invariant, the support point is fixed by the right shift action $\theta$ on $\Omega^{\IZ}$. Thus $\omega \alpha = \omega_{e_i}$ for some $1 \le i \le d$, where $\omega_{e_i}=\otimes_{n \in \IZ} \omega^{(n)}_{e_i}$ is the infinite tensor product state with $\omega^{(n)}_{e_i}=\omega^{(n+1)}_{e_i}=\rho_{e_i}$ and  
$\rho_{e_i}$ is the pure vector state on $\!M_d(\IC)$ given by  
the vector $e_i \in \IC^d$ of norm one.    
\end{proof} 

\vsp 
\begin{cor} 
Any translation invariant pure state of $\IM$ admits Kolmogorov property. 
\end{cor} 

\vsp 
\begin{proof}
Since both pure and Kolmogorov properties are invariant for the translation dynamics, the statement is a simple consequence of Theorem 6.9 as the pure state $\omega_{\lambda}$ admits Kolmogorov property. 
\end{proof} 

\vsp 
\begin{rem} 
The unique ground state $\omega$ for fero-magnetic $XY$-model being translation invariant pure state, it admits Kolmogorov property but it admits Ren\'ee-Schroeder property $[\pi_{\omega}(\IM_R)''\zeta_{\omega}]=[\pi_{\omega}(\IM)''\zeta_{\omega}]$. In other words a translation invariant pure state need not be Kolmogorov in the sense defined earlier in [Mo1,Mo2], which is not an invariance. 
\end{rem} 

\vsp 
Now we introduce a formal definition for mean entropy of a translation invariant state $\omega$ of $\IM$. Let $\clm_0$ be collection of $C^*$-subalgebra $\IM_0$ of $\IM$ satisfying relations given in Proposition 1.1. By Proposition 1.1, $\IM_{0,n} = \IM'_0 \bigcap \theta^n(\IM_0)$ is isomorphic to $\otimes_{1 \le k \le n} \IM_d^{(k)}(\IC)$, we set 
$$s(\omega,\IM_0)=\mbox{lim}_{n \raro \infty}{1 \over n}S(\omega|\IM_{0,n})$$
and
$$s(\omega)=\mbox{inf}_{\IM_0 \in \clm_0}s(\omega,\IM_0)$$
So $s(\omega)$ is an invariance for translation invariant state $\omega$ and the 
$\omega \raro s(\omega)$ is upper-semicontinuous.  

\vsp 
\begin{thm} 
Mean entropy $s(\omega)$ is an invariant for the dynamics $(\IM,\theta,\omega)$ and is equal to $h_{CS}(\IM,\theta,\omega)$. The map $\omega \raro h_{CS}(\IM,\theta,\omega)$ is upper-semicontinuous. 
\end{thm}

\vsp 
\begin{proof} We assume for the time being the state $\omega$ is faithful. It is known [NS] that $h_{CS}(\IM,\theta,\omega) \le s(\omega,\IM_0)$ for any $\IM_0 \in \clm_0$ and so $h_{CS}(\IM,\theta,\omega) \le s(\omega)$. We will show $h_{CS}(\IM,\theta,\omega) \ge s(\omega)$.

\vsp 
We verify now the following steps:
$$h_{CS}(\IM,\theta,\omega)$$
$$=h_{CS}(\IM,\theta,\omega \alpha)$$
$$=h_{KS}(\ID^e,\theta,\omega \alpha)$$
$$=h_{KS}(\alpha(\ID^e),\theta,\omega)$$
$$=\mbox{lim}_{n \raro \infty} {1 \over n} S(\omega|\zeta_n)$$
where $\zeta_n = \zeta_0 \vee \theta(\zeta_0) \vee ..\vee \theta^{n-1}(\zeta_0)$ and 
$\zeta_0$ is a commutative algebra generated by a finite family of projections in 
$\pi_{\omega}(\alpha(\ID^e))''$ that gives a strong generator [Pa] for the dynamics $(\pi_{\omega}(\alpha(\ID^e))'',\theta,\omega_0)$ i.e. $\pi_{\omega}(\vee_{n \in \IZ} \theta^n(\zeta_0))'' = \pi_{\omega}(\alpha(\ID^e))''$. 
One such trivial choice is given by $\zeta_0=\alpha(\ID^e_{1})$ and $S(\omega|\zeta_n) = S(\omega|\alpha(\IM_{0,n})$ and thus $h_{CS}(\IM,\theta,\omega)=s(\omega,\IM_0) \ge s(\omega)$, where $\IM_0=\alpha(\IM_{(-\infty,0]})$.  

\vsp 
Now we deal with a translation invariant state $\omega$ of $\IM$ which need not be faithful. We 
will use a limiting argument that we have used in Proposition 6.6. For $\lambda \in (0,1)$, we set faithful states $\omega_{\lambda}$ defined by 
$$\omega_{\lambda}= \lambda \omega + (1-\lambda)\omega_0,$$ 
where we recall $\omega_0$ is the unique normalised trace on $\IM$. 

\vsp 
The state $\omega_{\lambda}$ being the $\sigma^{\omega_{\lambda}}$-KMS state on $\IM$, the centralizer $\clc_{\omega_{\lambda}} = \{x \in \IM: \omega_{\lambda}(xy)=\omega_{\lambda}(yx), \forall y \in \IM \}$ is equal to $\IM_{\sigma^{\omega_{\lambda}}}$. So by the tracial property of $\omega_0$, $\IM_{\sigma^{\omega_{\lambda}}}$ is independent of $\lambda \in (0,1)$. We define by  
$\IM_{\sigma^{\omega}}=\IM_{\sigma^{\omega}_{\lambda}}$. 
Note that this modified definition is consistent with definition for a faithful $\omega$.  

\vsp 
Since we still have no proof for affine property of the map $\omega \raro h_{CS}(\IM,\theta,\omega)$, we will use a direct argument. By Proposition 6.6, there exists an automorphism $\alpha$ on $\IM$ commuting with $\theta$ on $\IM$ satisfying 
$$\omega_{\lambda} \alpha = \omega_{\lambda} \alpha \IE^e_{\omega_0}$$
Since $\omega_0 \IE^e_{\omega_0} = \omega_0$ and $\omega_0 \alpha = \omega_0$, we conclude that $\omega \alpha = \omega \alpha \IE^e_{\omega_0}$. Now we apply Proposition 6.6 to complete the proof by adapting the argument used for faithful state $\omega$ by replacing the role $\IE_{\omega}$ by $\alpha \IE_{\omega_0}$, which is norm one projection from $\pi_{\omega}(\IM)''$ onto $\pi_{\omega}(\alpha(\ID^e))''$. 

\end{proof} 

\vsp 
In the proof for Theorem 6.12, partition $\zeta_0$ for $\alpha(\ID^e)$ need not be an element in $\IM_{loc}$ unless $\alpha$ takes $\IM_{loc}$ to itself. Furthermore, Lemma 3.2 suggest that even there may not exists a partition $\ul{\zeta}=(\zeta_i)$ made of projections in 
$\IM_{loc}$ such that $\ul{\zeta}$ is a strong generator for $(\alpha(\ID^e),\theta,\omega)$ unless $\alpha$ preserves local algebra. 

\vsp 
Question that arises now, can we approximate a given partition of $\alpha(\ID^e)$ made of projections by partitions made of possible positive elements in the local algebra $\IM_{loc}$ in the weak$^*$-topology of $\pi_{\omega}(\IM)''$? 

\vsp 
Since $\pi_{\omega}(\IM_{loc})''=\pi_{\omega}(\IM)''$, for a given $\epsilon > 0$, there exists a $\delta > 0$ and finite subset $\Lambda_{\epsilon}$ of $\IZ$ and channel $\zeta_{\delta}:\IM_{\Lambda} \raro \IM$ such that $||\zeta_{\delta} - \zeta||_{\omega} = \mbox{sup}_{x \in \IM_{\Lambda_{\delta}}: ||x|| \le 1 } |\omega(\zeta_{\delta}(x)-\zeta(x)| \le < \delta$, where $\zeta(x)=\sum_{1 \le i \le d} \alpha(e_i) x \alpha(e_i)$ and 
$$|H_{\omega}(\zeta,\theta(\zeta),..,\theta^n(\zeta))-H_{\omega}(\zeta_{\delta},\theta(\zeta_{\delta}),..\theta^n(\zeta_{\delta})| \le (n+1)\epsilon$$
by Proposition 3.1.11 in [NS]. 

\vsp 
Since 
$$H_{\omega}(\zeta_{\delta},\theta(\zeta_{\delta}),..\theta^n(\zeta_{\delta}) \le H_{\omega}(\IM_{\Lambda},\theta(\IM_{\Lambda})),..,\theta^n(\IM_{\Lambda})) $$
$$\le n+1 H_{\omega}(\IM_{\bigcup_{1 \le k \le n} \Lambda + k})$$
we conclude that

\vsp 
\NI (a) $h_{\omega}(\zeta,\theta)=s(\omega,\IM_0);$

\vsp 
\NI (b) $h_{\omega}(\zeta,\theta) \le h_{\omega}(\zeta_{\delta},\theta) + \epsilon;$

\vsp 
\NI (c) $h_{\omega}(\zeta_{\delta},\theta) \le s_{\omega}(\IM_R,\omega).$

\vsp 
This shows that the mean entropy $s(\omega)=s(\omega,\IM_0)$ for any $\IM_0 \in \clm_0$ and so we arrive at the following theorem. 

\vsp 
\begin{thm} 
Let $\omega$ be a translation invariant state of $\IM$ then mean entropy $s(\omega,\IM_{loc}) = \mbox{lim}_{|\Lambda_n| \raro \infty} {1 \over |\Lambda_n|}S(\omega|\IM_{\Lambda_n})$ is an invariance equal to Connes-St \o rmer dynamical entropy $h_{CS}(\IM,\theta)$.  
\end{thm}

\vsp 
\begin{thm} 
Let $\omega$ and $\omega'$ be two faithful translation invariant state of $\IM$. Then $(\IM,\theta,\omega)$ and $(\IM,\theta,\omega')$ are isomorphic if and only if $(\IM_{\hat{\sigma}^{\omega}},\theta,\omega)$ and $(\IM_{\hat{\sigma}^{\omega'}},\theta,\omega')$ are isomorphic. 
\end{thm} 

\vsp 
\begin{proof} 
For the `only if' part, we use the uniqueness of the modular automorphism satisfying KMS condition associated with a faithful state to prove the following equality 
$$\alpha \sigma^{\omega}_t = \sigma^{\omega'}_t \alpha$$ 
for all $t \in \IR$, where 
$$(\IM,\theta,\omega) \equiv^{\alpha} (\IM,\theta,\omega')$$ 
Thus we have $\alpha(\IM_{\hat{\sigma}^{\omega}})=\IM_{\hat{\sigma}^{\omega'}}$ and $\theta \alpha = \alpha \theta$ on $\IM_{\hat{\sigma}^{\omega}}$ with $\omega = \omega' \alpha$ on $\IM_{\hat{\sigma}^{\omega}}$. 

\vsp 
Let $(\IM_{\hat{\sigma}^{\omega}},\theta,\omega) \equiv^{\alpha_0} (\IM_{\hat{\sigma}^{\omega'}},\theta,\omega')$. In particular, there exists translation invariant maximal abelian $C^*$ sub-algebras $\ID_{\omega}$ and $\ID_{\omega'}$ of $\IM_{\hat{\sigma}^{\omega}}$ and $\IM_{\hat{\sigma}^{\omega'}}$ respectively such that $\alpha_0(\ID_{\omega})=\ID_{\omega'}$. By Proposition 6.6, we assume without loss of generality that $\omega = \omega \IE^e_{\omega_0}$ with $\ID_{\omega}=\ID^e$ and $\omega' = \omega' \IE^e_{\omega_0}$ with $\ID_{\omega'}=\ID^e$. 
Let $\alpha$ be an automorphism on $\IM$ that commutes with $\theta$ and extends $\alpha_0:\ID_{\omega} \raro \ID_{\omega'}$. Such an $\alpha$ exists is guaranteed by Lemma 3.5.  

\vsp 
We claim that $\omega' = \omega \alpha$. For the equality we check the following steps:
$$\omega'$$
$$=\omega' \IE^e_{\omega_0}$$
$$=\omega \alpha \IE^e_{\omega_0}$$
$$=\omega \IE^e_{\omega_0} \alpha$$
$$=\omega \alpha,$$
where we have used the property that $\alpha \IE^e_{\omega_0} \alpha^{-1}=\IE^e_{\omega_0}$ since both sides are norm one projections from $\IM$ onto $\ID^e$ preserving the normalised trace $\omega_0$.   
\end{proof} 

\vsp 
As a consequence of Propositions proved so far, we have some additional results listed in the following theorems. 

\vsp 
\begin{thm} 
Let $\omega$ be a translation invariant state of $\IM$ and $\omega = \omega \IE_{\omega_0}$, where $\IE_{\omega_0}$ is a norm one projection from $\IM$ onto the maximal abelian $C^*$ sub-algebra $\ID_{\omega}$ described as in Proposition 6.6. Then $(\IM,\theta,\omega)$ is ergodic, weakly mixing, strongly mixing, pure if and only if $(\ID_{\omega},\theta,\omega)$ is ergodic, weakly mixing, strongly mixing, pure respectively. Furthermore the following hold:

\NI (a) $\omega$ is a factor state if and only if for each $x \in \ID_{\omega}$ we have
\be 
\mbox{sup}_{\{y \in \ID_{\omega}:||y|| \le 1 \}} |\omega(x \theta^n(y))-\omega(x)\omega(y)| \raro 0
\ee
as $n \raro -\infty$. 

\vsp 
\NI (b) Two-point spatial correlation function of $(\IM,\theta,\omega)$ decays exponentially i.e. for some $\delta > 0$ 
\be 
e^{\delta n}|\omega(x \theta^n(y))-\omega(x)\omega(y)| \raro 0
\ee 
for all $x,y \in \IM_{loc}$ as $n \raro -\infty$ 
if and only if two-point spatial correlation function of $(\ID_{\omega},\theta,\omega)$ decays exponentially i.e.  
\be 
e^{\delta n}|\omega(x \theta^n(y))-\omega(x)\omega(y)| \raro 0
\ee
for all $x,y \in (\ID_{\omega})_{loc}$ as $n \raro -\infty$.  
\end{thm} 

\vsp 
\begin{proof} Since properties under our consideration are invariant for translation dynamics, we assume without loss of generality that $\ID_{\omega}=\ID^e$ and $\omega = \omega \IE^e_{\omega_0}$ by Proposition 6.6.  

\vsp 
We fix any $x,y,z \in \IM_{loc}$. Then $x,z \in \IM_{(-\infty,m]}$ and there exists an $n_0 \in \IZ$ such that 
$\theta^n(y) \in \IM_{(m,\infty)}$ for all $n \ge n_0$. For such $n \ge n_0$, we have 
$$\omega(x \theta^n(y)z)$$ 
$$=\omega( \IE^e_{\omega_0}(x \theta^n(y)z))$$ 
$$=\omega(\IE^e_{\omega_0}(xz) \IE^e_{\omega_0}(\theta^n(y)))$$
$$=\omega(\IE^e_{\omega_0}(xz) \theta^n (\IE^e_{\omega_0}(y)))$$
This shows that ergodicity, weak weakly mixing, strongly mixing of $(\ID_{\omega},\theta,\omega)$ implies ergodicity, weak weakly mixing, strongly mixing of $(\IM,\theta,\omega)$ respectively. Along the same line of argument, the asymptotic beheviour given in (25) also says that the state $\omega$ of $\IM$ is a factor by Theorem 2.5 in [Pow]. That asymptotic behaviour (27) and (26) are equivalent follows along the same line of argument. 

\vsp 
Let $\omega$ be pure and its restriction to $\ID^e$, $\omega_c = \mu \omega_1 + (1-\mu)\omega_2$ for some $\mu \in (0,1)$ and states $\omega_1,\omega_2$ of $\ID^e$. Then we have 
$\omega = \omega \IE^e_{\omega_0} = \mu \omega_1 \IE^e_{\omega_0} + (1-\mu)\omega_2\IE^e_{\omega_0}$. 
By pure property of $\omega$, we have $\omega = \omega_1 \IE^e_{\omega_0}$ and $\omega= \omega_2 \IE^e_{\omega_0}$. Thus 
$\omega = \omega_1 = \omega_2$ on $\ID^e$. Conversely, let $\omega$ be pure on $\ID^e$. Since pure $\omega$ is translation invariant, it has support at a singleton point and the point is fixed by the translation action. 
Thus we get $\omega = \omega_{e_i}$ on $\ID^e$ for some $1 \le i \le d$, where $\omega_{e_i}= \otimes_{n \in \IZ} \rho^{(n)}_{e_i}$ and $\rho^{(n)}_{e_i}(x) = <e_i,xe_i>, \forall x \in \!M^{(n)}_d(\IC)$ and $n \in \IZ$. 
Since $\omega = \omega \IE^e_{\omega_0}$ and $\omega_{e_i}= \omega_{e_i} \IE^e_{\omega_0}$, 
we also have $\omega = \omega \IE^e_{\omega_0}= \omega_{e_i} \IE^e_{\omega_0}=\omega_{e_i}$. Thus $\omega$ is pure if and only if $\omega$ on $\ID_{\omega}$ is pure.   
\end{proof}

\vsp 
\begin{thm} 
Let $\omega$ be a translation invariant state $\omega$ of $\IM$. Then the following hold:

\NI (a) $\omega$ is a factor state if and only if $\omega$ is weakly Kolmogorov;

\NI (b) If $\omega$ is a factor state then it is pure if and only the mean entropy $s(\omega)=0$;

\NI (c) If $s(\omega)=0$ and $\omega$ is not a pure state of $\IM$ then $\omega$ is not weakly Kolmogorov and $\omega|\ID_{\omega}$ is not Kolmogorov, 
where $\ID_{\omega}$ is taken as in Proposition 6.6.

\end{thm} 

\vsp 
\begin{proof} 
Since both factor and weakly Kolmogorov property are invariance, without loss of generality we take $\clb_0=\clb_L$ for the purpose of the proof. We have already gave a proof in the introduction that weak Kolmogorov property gives factor property 
of $\omega$. For the converse, we will use Power's criteria [Pow] given in (7) for the factor property of $\omega$ as follows: Note that $\{x \zeta_{\omega}:||x|| \le 1 \}$ is dense in the unite ball of the Hilbert space $\clh_{\omega}$ and thus for any two elements $x,y \in \IM$, we have 
$$||F_{n]}x \zeta_{\zeta}||$$
$$=\mbox{sup}_{y \in \IM_L, ||y|| \le 1}\langle \theta^n(y^*), F_{n]} x \zeta_{\omega} \rangle$$   
$$=\mbox{sup}_{y \in \IM_L, ||y|| \le 1} |\langle F_{n]} \theta^n(y^*), x \zeta_{\omega}>|$$
$$=\mbox{sup}_{y \in \IM_L, ||y|| \le 1} |\langle \theta^n(y^*), x \zeta_{\omega}>|$$
$$=\mbox{sup}_{y \in \IM_L, ||y|| \le 1} \omega(\theta^n(y)x)|$$
$$ \raro 0$$
as $n \raro -\infty$ once $\omega(x)=0$ by Power's criteria. Going along the same line of the proof, we also verify same holds
for all $x,y \in \theta^{-m}(\IM_L),\; m \ge 1$ provided $\omega(x)=0$. Now we use the linear property of the map $\theta$ to prove that $F_{n]} x \zeta_{\omega} \raro \omega_0(x)\zeta_{\omega}$ strongly as $n \raro -\infty$ for all $x \in \bigcup_{n \ge 1}\theta^n(\IM_L)$ and thus $F_{n]} \raro |\zeta_{\omega} \rangle \langle \zeta_{\omega}|$ as $n \raro -\infty$.  

\vsp 
The first statement (b) is obvious now by Theorem 6.13 since $h_{CS}(\IM,\theta,\omega)=0$ for a pure state $\omega$. But now we will investigate little more deep into this issue to show the converse is also true for a translation invariant factor state $\omega$. We assume without lose of generality that the state $\omega$ on $\IM$ satisfies $\omega = \omega \IE^{e}_{\omega_0}$ and thus $s(\omega)=h_{KS}(\ID^e,\theta,\omega)$. 

\vsp 
If $[\pi_{\omega}(\ID^e)\zeta_{\omega}]$ is non-degenerate i.e. not one dimensional then 
$$f_{n]}=[\pi_{\omega}(\theta^n(\ID^e_{\IZ_-})\zeta_{\omega}] \downarrow |\zeta_{\omega} \rangle \langle \zeta_{\omega}|$$ 
as $n \raro -\infty$ if and only if $h_{KS}(\ID^e,\theta,\omega) > 0$ by Rohlin-Sinai theorem [Pa], where $\IZ_- = \{n \in \IZ: n \le 0 \}$. In particular, $f_{n]} \downarrow |\zeta_{\omega} \rangle \langle \zeta_{\omega}|$ as $n \raro -\infty$ if and only if $s(\omega)> 0$ 
provided $\omega$ on $\ID^e$ is not pure i.e. is not degenerate. This completes the proof of (b) once we use (a).

\vsp 
The statement (c) follows from (a) and (b)
 
\end{proof}

\vsp 
\begin{rem} 
There are translation invariant non pure state $\omega_c$ on $\ID^e$ for which $h_{KS}(\ID^e,\theta,\omega_c)=0$. For such a state $\omega_c$, $\omega = \omega_c \IE^e_{\omega_0}$ is a translation invariant non pure state of $\IM$ with $s(\omega)=0$. This shows that the zero mean entropy, i.e. $s(\omega)=0$ in general, does not imply $\omega$ is pure. 
It is simple to prove Kolmogorov property for $(\ID_{\omega},\theta,\omega)$ whenever $(\IM,\theta,\omega)$ is Kolmogorov. Though, the converse statement is likely to be true, at present we do not have a proof for the converse statement. In case, the converse statement is true, we can lift Ornstein's counter example [Or2] for a non isomorphic translation dynamics to Bernoulli shift to prove that positive Connes-St\o rmer dynamical entropy is also not a complete invariant for translation invariant factor states of $\IM$.  
\end{rem}

\section{ A complete weak isomorphism theorem for Bernoulli states of $\IM$ }

\vsp 
\begin{lem} 
For two translation invariant states $\omega_1$ and $\omega_2$ on $\IM$, 
$(\IM,\theta,\omega_1)$ and $(\IM,\theta,\omega_2)$ are isomorphic if and only if $(\pi_{\omega_1}(\IM)'',\Theta_1, \phi_{\zeta_{\omega_1}})$ and $(\pi_{\omega_1}(\IM)'',\Theta_1,\phi_{\zeta_{\omega_1}})$ are isomorphic. 
\end{lem} 

\vsp 
\begin{proof}
A simple proof uses the fact that the $C^*$- algebra $\IM$ being a $UHF$ [Pow], an automorphism $\alpha:\pi_{\omega_1}(\IM)'' \raro \pi_{\omega_2}(\IM)''$ determines a unique automorphism 
$\alpha_0:\IM \raro \IM$ such that $\alpha(\pi_{\omega_1}(x))=\pi_{\omega_2}(\alpha_0(x))$ for 
all $x \in \IM$. 
\end{proof} 

\vsp 
For two translation invariant state $\omega_1$ and $\omega_2$ of $\IM$, isomorphism of 
dynamical systems $(\pi_{\omega_1}(\ID^e)'',\Theta,\hat{\omega}_1)$ and $(\pi_{\omega_1}(\ID^e)'',\Theta,\hat{\omega}_2)$ 
may not make dynamical systems $(\ID^e,\theta,\omega_1)$ and $(\ID^e,\theta,\omega_2)$ isomorphic even if 
$\omega_1$ and $\omega_2$ are infinite tensor product states. Thus unlike isomorphism problem in classical spin chain, quantum chain problem give raises no additional mathematical obstructions or favour while we deal with isomorphism problem 
in the associated $W^*$- or von-Neumann dynamical systems, where simplicity of the underlying $C^*$ algebra $\IM$ makes crucial difference from the classical situation where $\ID^e$ is far from being simple. This essentially forces us to consider a weaker notion for isomorphism between two dynamical systems for quantum spin chain in order to include classical
isomorphism problems as part of quantum problem.  

\vsp 
We say two dynamical systems $(\IM,\theta,\omega_1)$ and $(\IM,\theta,\omega_2)$ are {\it weakly isomorphic } if there exists
unital completely positive maps $\tau_{12}: \IM \raro \pi_{\omega_2}(\IM)''$ and 
$\tau_{21}:\IM \raro \pi_{\omega_1}(\IM)''$ satisfying the following:

\vsp 
\NI (a) $\omega_1 = \hat{\omega_2} \tau_{12}$ and $\omega_2 = \hat{\omega}_1 \tau_{21} $ on $\IM$;

\vsp 
\NI (b) $\tau_{12} \theta_1  = \Theta_2 \tau_{12}$ and $\tau_{21} \theta_2 = \Theta_1 \tau_{21}$;

\NI where we have used notation $\hat{\omega}$ for a GNS vector state of $\pi_{\omega}(\IM)''$ associated with $\omega$ of $\IM$. 

\vsp 
Let $\omega$ be a faithful state on the $C^*$ algebra $\IM=\otimes_{n \in \IZ} \!M_d^{(n)}(\IC)$ and $(\clh_{\omega},\pi_{\omega}, \zeta_{\omega})$ be the GNS representation of $(\IM,\omega)$ so that $\omega(x)=\langle \zeta_{\omega}, \pi_{\omega}(x) \zeta_{\omega} \rangle$ for all $x \in \IM$. Thus the unit vector $\zeta_{\omega} \in \clh_{\omega}$ is cyclic and separating for von-Neumann algebra $\pi_{\omega}(\cla)''$ that is acting on the Hilbert space $\clh_{\omega}$. Let $\Delta_{\omega}$ and $\clj_{\omega}$ be the modular and conjugate operators on $\clh_{\omega}$ respectively of $\zeta_{\omega}$ as described in section 2. We recall, the modular automorphisms group $\sigma^{\omega}=(\sigma^{\omega}_t:t \in \IR)$ 
of $\omega$ is defined by 
$$\sigma_t^{\omega}(a)=\Delta^{it}_{\omega}a\Delta_{\omega}^{-it}$$
for all $a \in \pi_{\omega}(\IM)''$. 

\vsp 
For a given faithful state $\omega$ of $\IM$, $(\sigma^{\omega}_t)$ may not keep $\ID^e_{\omega}$ invariant, where $\ID^e_{\omega}=\pi_{\omega}(\ID^e)''$ is a 
von-Neumann sub-algebra of $\clm_{\omega}=\pi_{\omega}(\IM)''$. However, for an 
infinite tensor product state $\omega_{\rho} = \otimes_{n \in \IZ} \rho^{(n)}$ with $\rho^{(n)}=\rho$ for all $n \in \IZ$, $\sigma_t^{\omega_{\rho}}= \otimes_{n \in \IZ}\sigma_t^{\rho_n}$ preserves $\ID^e_{\omega}$ if $\rho$ is a diagonal matrix in 
the orthonormal basis $e=(e_i)$ of $\IC^d$. In such a case, $\IE^e_{\omega_{\rho}}: \clm_{\omega_{\rho}} \raro \ID_{\omega_{\rho}}$ satisfy the bi-module property (17) 
and furthermore, we have  
\be 
\IE^e_{\omega_{\rho}}(\pi_{\omega_{\rho}}(x))=\pi_{\omega_{\rho}}(\IE^e_{\omega_0}(x))
\ee
for all $x \in \IM$, where $\IE^e_{\omega_0}$ is the norm one projection from $\IM$ onto $\ID^e$ with respect to the unique normalised trace $\omega_0$ of $\IM$. One can write explicitly 
$$\IE^e_{\omega_0} = \otimes_{n \in \IZ} \IE^{(n)}_0$$
where $\IE^{(n)}_0 = \IE_0$ for all $n \in \IZ$ and $\IE^e_0$ is the norm one projection 
from $\!M_d(\IC)$ onto the algebra of diagonal matrices $\!D_d(\IC)$ with respect to 
the basis $(e_i)$. The map $\IE^e_0:\!M_d(\IC) \raro \!D_d(\IC)$ is given by 
$$\IE^e_0(x) = \sum_{1 \le i \le d} \langle e_i, x e_i \rangle |e_i \rangle \langle e_i|$$
for all $x \in \!M_d(\IC)$. 

\vsp 
The map $\rho \raro \hat{\rho}$ is one to one and onto between the set of density matrices in $\!M_d(\IC)$ and the set of states on $\!M_d(\IC)$, where 
$$\hat{\rho}(x)=tr(x\rho)$$ 
for all $x \in \!M_d(\IC)$. We define entropy for a density matrix $\rho$ in $\!M_d(\IC)$ by 
$$S(\rho) = -tr(\rho ln \rho)$$
$$=-\sum_i \lambda_i ln(\lambda_i),$$ 
where $\rho= \sum_{1 \le i \le d} \lambda_i |e_i \rangle\langle e_i|$ 
for some $\lambda_i \in [0,1]$ with $\sum_i \lambda_i = 1$. One alternative 
description of $S(\rho)$ is given by a variational formula [OP]:   
$$S(\rho) = \mbox{inf}_{f=(f_i)} - \sum_i \hat{\rho}(|f_i\rangle\langle f_i|)ln (\hat{\rho}(|f_i\rangle\langle f_i|)),$$ 
where infimum is taken over all possible orthonormal basis. Thus the variational expression for $S(\rho)$ achieves its values for the basis $e=(e_i)$ that makes $\rho$ diagonal.   

\vsp 
\begin{thm} 
Mean entropy is a complete weak invariant of translation dynamics for the class of infinite tensor product faithful states of $\IM$ i.e. Two infinite tensor product faithful states $\omega_{\rho}$ and $\omega_{\rho'}$ give weak isomorphic translation dynamics if and only if $S(\rho)=S(\rho')$. 
\end{thm}

\vsp 
\begin{proof} 
For the time being, we assume both the faithful states $\rho$ and $\rho'$ admit diagonal representations with respect to an orthonormal basis $e=(e_i)$ of $\IC^d$. So the restrictions of $\omega_{\rho}$ and $\omega_{\rho'}$ are Bernoulli states on $\ID^e$ with equal Kolmogorov-Sinai dynamical entropies. Thus by a theorem of D. Ornstein [Or1], we find an automorphism $\bar{\beta}_0:\pi_{\rho}(\ID^e)'' \raro \pi_{\rho'}(\ID^e)''$ such that $\bar{\beta}_0 \theta = \theta \bar{\beta}_0$ on $\pi_{\rho}(\ID^e)''$ and $\omega_{\rho} = \omega_{\rho'} \bar{\beta}_0$ on $\pi_{\rho}(\ID^e)''$. 

\vsp 
For the time being we assume that $\bar{\beta}_0$ determines an automorphism $\beta_0$ on $\ID^e$ such that \be 
\bar{\beta}_0(\pi_{\rho}(x)) = \pi_{\rho}(\beta_0(x))
\ee 
for all $x \in \ID^e$ and so $\beta_0 \theta = \theta \beta_0$ on $\ID^e$. Let $\beta$ be an automorphism on $\IM$ that extends $\beta_0:\ID^e \raro \ID^e$ and commutes with $\theta:\IM \raro \IM$. Such an automorphism exists by Corollary 3.6.  

\vsp 
Now we consider the state $\omega= \omega_{\rho} \beta$ on $\IM$. The state $\omega$ is faithful and its modular automorphism group $\sigma_t^{\omega}:\pi_{\omega}(\IM)'' \raro \pi_{\omega}(\IM)''$ associated with the cyclic and separating vector $\zeta_{\omega}$ for $\pi_{\omega}(\IM)''$ satisfies the relation 
\be 
\sigma_t^{\omega} \hat{\beta} = \hat{\beta} \sigma_t^{\omega_{\rho}}
\ee 
on $\pi_{\omega_{\rho}}(\IM)''$, where $\hat{\beta}: \pi_{\omega_{\rho}}(\IM)'' \raro \pi_{\omega}(\IM)''$ is the automorphism defined by extending the map
\be 
\hat{\beta}: \pi_{\omega_{\rho}}(x) \raro \pi_{\omega}(\beta(x))
\ee   
for all $x \in \IM$. For a proof, we can use uniqueness of the automorphisms group of $\omega$ satisfying ${1 \over 2}$-KMS condition (14) [BR2].   

\vsp 
Since automorphism $\beta$ is an extension of automorphism $\beta_0$ of $\ID^e$, the equation (30) says that $\sigma_t^{\omega}$ also preserves $\pi_{\omega}(\ID^e)''$. Thus once again by a theorem of M. Takesaki [Ta2], there exists a norm one projection 
$$\IE^e_{\omega}:\pi_{\omega}(\IM)'' \raro \pi_{\omega}(\IM)''$$ 
with range equal to $\pi_{\omega}(\ID^e)''$ satisfying 
$$\omega = \omega \IE^e_{\omega} $$ 
on $\pi_{\omega}(\IM)''$.

\vsp 
The rest of the proof are done in the follows elementary steps:

\vsp 
\NI (a) $\hat{\beta}^{-1} \IE^e_{\omega_{\rho}} \hat{\beta} = \IE^e_{\omega}$;

\vsp 
\begin{proof} 
It is obvious that both the maps $\IE^e_{\omega}$ and $\hat{\beta}^{-1} \IE^e_{\omega_{\rho}} \hat{\beta}$ are norm one projections satisfying bi-module property (17) from $\pi_{\omega}(\IM)''$ onto $\pi_{\omega}(\ID^e)''$. Furthermore, 
$$\omega \hat{\beta}^{-1} \IE^e_{\omega_{\rho}} \hat{\beta}$$ 
$$=\omega_{\rho} \IE^e_{\omega_{\rho}} \hat{\beta}$$
$$=\omega_{\rho} \beta$$
$$=\omega$$
Thus by uniqueness of the norm project with respect to $\omega$, we get the equality in (a). 
\end{proof} 

\vsp 
\NI (b) $\IE^e_{\omega} \pi_{\omega}(x) = \pi_{\omega}(\IE^e_{\omega_0}(x))$ for all $x \in \IM$. 

\begin{proof}
For any automorphism $\beta$, $\omega_0 \beta$ is also a normalised trace and thus by the uniqueness of normalised trace, we get $\omega_0 \beta = \omega_0$. Furthermore, $\beta^{-1} \IE^e_{\omega_0} \beta$ is also the norm one projection with respect to the unique 
normalised trace $\omega_0$ and thus we get $\beta^{-1} \IE^e_{\omega_0} \beta = \IE^e_{\omega_0}$ on $\IM$ by the uniqueness of the norm one projection with respect to $\omega_0$. We now complete the proof of (b) by the covariant property (a) and 
(28) valid for $\omega_{\rho}$.   
\end{proof} 

\vsp 
Let $\alpha_{p}$ be the automorphism on $\IM$ that extends linearly the map $x= \otimes_{n \in \IZ} x_n \raro \otimes_{n \in \IZ} x_{p(n)}$, where $p$ is a permutation on $\IZ$ that keeps all but finitely many points of $\IZ$ unchanged.  
We also define automorphism $\hat{\alpha}_p:\pi_{\omega}(\IM)'' \raro \pi_{\omega}(\IM)''$ by extending the map 
$$\pi_{\omega}(x) \raro \pi_{\omega}(\alpha_p(x))$$
for all $x \in \IM$ to its weak$^*$-closer. 
In particular, we have 
$$\alpha_{p} \IE^e_{\omega_0} = \IE^e_{\omega_0} \alpha_{p}$$ on $\IM$ by (28). 

\vsp 
\NI (c) $\omega \alpha_{p} = \omega$ for all permutation $p$ on $\IZ$. 

\begin{proof} We verify the following equalities: 
$$\omega(x)$$
$$=\langle \zeta_{\omega},\pi_{\omega}(x) \zeta_{\omega}\rangle$$
$$=\langle \zeta_{\omega}, \IE^e_{\omega}(\pi_{\omega}(x)) \zeta_{\omega} \rangle$$
(since $\omega \IE^e_{\omega}=\omega$ )
$$=\langle \zeta_{\omega}, \hat{\alpha}_p ( \IE^e_{\omega}(\pi_{\omega}(x)))\zeta_{\omega} \rangle$$
(since $\omega \alpha_p = \omega$ on $\ID^e$ as $\omega=\omega_{\rho'}$ on $\ID^e$ by 
our construction.) 
$$=\langle \zeta_{\omega}, \hat{\alpha}_p (\pi_{\omega}(\IE^e_{\omega_0}(x))) \zeta_{\omega} \rangle$$
( by (b) )
$$= \langle \zeta_{\omega}, \pi_{\omega}( \alpha_p(\IE^e_{\omega_0}(x))) \zeta_{\omega} \rangle$$
( by definition of $\hat{\alpha}_p$ )
$$= \langle \zeta_{\omega}, \pi_{\omega}( \IE^e_{\omega_0}(\alpha_p(x))) \zeta_{\omega} \rangle$$
( since $\IE^e_{\omega_0} \alpha_p = \alpha_p \IE^e_{\omega_0}$ )
$$= \langle \zeta_{\omega}, \IE^e_{\omega}(\pi_{\omega}(\alpha_p(x))) \zeta_{\omega} \rangle$$
( again by (b) )
$$= \langle \zeta_{\omega}, \pi_{\omega}(\alpha_p(x)) \zeta_{\omega} \rangle$$
( since $\omega = \omega \IE^e_{\omega}$ )
$$=\omega(\alpha_p(x))$$
\end{proof} 

\vsp 
\NI (d) $\omega$ is an infinite tensor product state;

\begin{proof} 
It follows from (c) once we use a quantum version of classical de Finetti theorem [St\o 2], which says any translation invariant factor state with permutation symmetry 
is an infinite product state.    
\end{proof} 

\vsp 
The state $\omega$ being an infinite product translation invariant state of $\IM$,   
we may write $\omega=\otimes_{n \in \IZ} \omega^{(n)}$, where $\omega^{(n+1)}=\omega^{(n)}$ for all $n \in \IZ$ and $\omega^{(0)}(x)=tr(x\rho_0)$ for some density matrix $\rho_0$ in $\!M_d(\IC)$. The mean entropy being equal to CS dynamical entropy for any infinite product state, we get $S(\rho_0)=S(\rho)$. 

\vsp 
Since $\omega = \omega_{\rho} \beta_0 = \omega_{\rho'}$ on $\ID^e$ by our construction, $\hat{\rho}_0(|e_k \rangle \langle e_k|)=\hat{\rho'}(|e_k \rangle \langle e_k|)$ for all $1 \le k \le d$. Thus 
$$S(\rho_0)$$
$$=S(\rho)$$
$$=S(\rho')$$
$$= - \sum_k \hat{\rho'}(|e_k \rangle \langle e_k|) ln \hat{\rho'}(|e_k \rangle \langle e_k|)$$
$$= - \sum_k \hat{\rho}_0(|e_k \rangle\langle e_k|)ln(\hat{\rho}_0(|e_k \rangle\langle e_k|))$$ 

\vsp 
By our starting assumption $\rho'$ is also diagonal with respect to the orthonormal basis $e=(e_i)$. Thus $\rho_0=\rho'$ once we show $\rho_0$ is also diagonal in the basis $e=(e_i)$.

\vsp 
\NI (e) For a density matrix $\rho_0$ in $\!M_d(\IC)$ we have 
\be 
S(\rho_0) = - \sum_i \hat{\rho}_0(|e_i \rangle \langle e_i|)ln(\hat{\rho}_0(|e_i \rangle \langle e_i|))
\ee 
if and only if $\rho_0 = \sum_i \lambda_i |e_i \rangle \langle e_i|$ for some $\lambda_i \in [0,1]$ i.e. $\rho_0$ admits a diagonal representation in the basis $e=(e_i)$. 

\vsp 
\begin{proof} 
Let $E_0$ be the norm one projection from $\!M_d(\IC)$ on $d$ dimensional diagonal matrices $D_d(\IC)$ with respect to the orthonormal basis $e=(e_i)$. Then equality in 
(32) says $S(\rho_0)=S(E_0(\rho_0))$ as 
$$\hat{\rho_0} \circ E_0(x)$$
$$=tr(\rho_0 E_0(x))$$
$$=tr(E_0(\rho_0)E_0(x))$$
$$=tr(E_0(\rho_0)x)$$
for all $x \in \!M_d(\IC)$. We compute now the relative entropy [OP] 
$$S(\hat{\rho}_0, \hat{\rho}_0 \circ E_0))$$
$$= - tr( \rho_0 ln (\rho_0 - E_0(\rho_0))$$
$$= S(\rho_0) + tr E_0(\rho_0 ln E_0(\rho_0))$$
$$(\mbox{ by the bi-module property of}\;E_0 )$$
$$= S(\rho_0) + tr E_0(\rho_0) ln (E_0(\rho_0))$$
$$= S(\rho_0) - S( E_0(\rho_0))$$
$$=0\;\;\mbox{by (32)}$$ 
Thus $\hat{\rho}_0 = \hat{\rho}_0 E_0$ i.e. $\rho_0=E_0(\rho_0)$. This completes the proof for `only if' part.

\vsp 
Thus we complete the proof by concluding that $\hat{\rho}_0$ admits a diagonal representation in $e=(e_i)$ and so $\rho_0=\rho'$ since they have equal values 
on $|e_i \rangle\langle e_i|$ for all $1 \le i \le d$ as $\omega = \omega_{\rho'}$ on $\ID^e$. Thus $\omega = \omega_{\rho_0}=\omega_{\rho'}$.

\vsp 
For more general faithful states $\rho$ and $\rho'$ of $\!M_d(\IC)$ satisfying $S(\rho)=S(\rho')$, we can find an element $g \in S U_d(\IC)$ such that $\rho$ and $\rho' \beta_g$ are simultaneously diagonalizable with respect to an orthonormal basis 
$(e_i)$. Since $S(\rho')=S(g\rho'g^*)$ and $\beta_g \theta = \theta \beta_g$ for 
any $g \in SU_d(\IC)$, we can complete the proof by the first part of the argument. 
\end{proof}

\vsp 
This completes the proof provided $\bar{\beta_0}$ satisfies (29) for some automorphism $\beta_0$ on $\ID^e$. Now we revisit D. Ornstein's construction [Or1] of $\bar{\beta}$, where he constructed a sequence of auto-morphisms $\beta_0^{(n)}$ on $\ID^e$ that commutes with $\theta$ such that $\omega_{\rho} \beta_0^{(n)}(x) \raro \omega_{\rho'}(x)$, where each $\omega_{\rho} \beta_0^{(n)}$ is an infinite tensor product state. So for each $n \ge 1$, we find an automorphism $\beta^{(n)}$ on $\IM$
such that $\omega_{\rho} \beta^{(n)}$ is also an infinite tensor product state and $\omega_{\rho} \beta^{(n)} \IE^e_{\omega_0} = \omega_{\rho} \beta^{(n)}$. So $\omega{\rho} \beta^{(n)}(x) \raro \omega_{\rho'}(x)$ for all $x \in \IM$ as same holds 
for all $x \in \ID^e$. 

\vsp 
We consider now the sequence of automorphism $\hat{\beta^{(n)}}:\IM \raro \pi_{\omega_{\rho}}(\IM)''$ associated with 
$\beta^{(n)}$ defined by $x \raro \pi_{\omega_{\rho}}(\beta^{(n)}(x))$ and extract a subsequence still denoting them by $\hat{\beta}^{(n)}$ such that $\hat{\beta}^{(n)}(x) \raro \hat{\beta}(x)$ for all $x \in \IM$ in the boubded weak topology of William Arveson on the unit ball of unital completely positive maps, where $\hat{\beta}:\IM \raro \pi_{\omega_{\rho}}(\IM)''$
is an unital completely positive map. 

\vsp 
That $\hat{\beta} \theta = \hat{\theta} \hat{\beta}$ and $\omega_{\rho} \hat{\beta} \IE^e_{\omega_0} = \omega_{\rho} \hat{\beta}$ on $\IM$ are obvious. For $\omega_{\rho} \hat{\beta} = \omega_{\rho'}$ on $\IM$, we once again use that they agree on $\ID^e$ by Ornstien's construction and both side of states are $\IE^e_{\omega_0}$ invariant on $\IM$.   

\vsp 
For the inter-twiner in the reverse direction, we consider the sequence inverse of $\beta_0^{(n)}$ instead of $\beta_0^{(n)}$. 

\end{proof}

\vsp 
Note that the faithful property of states of $\IM$ are also an invariant for translation dynamics and thus two infinite translation invariant product states with equal entropy need not give isomorphic dynamics. This observation is a contrast to classical situation where $\ID^e=C(\Omega^{\IZ})$ are isomorphic as a topological space for all values of $d \ge 2$, where $\Omega=\{1,2,..,d\}$. In contrast $C^*$ algebras $\otimes_{n \in \IZ} \!M^{(n)}_2(\IC)$ and $\otimes_{n \in \IZ}\!M^{(n)}_3(\IC)$ are not isomorphic [Gl].

\bigskip
{\centerline {\bf REFERENCES}}

\begin{itemize} 
\bigskip 
\item{[AC]} Accardi, Luigi; Cecchini, Carlo: Conditional expectations in von Neumann algebras and a theorem of Takesaki.
J. Funct. Anal. 45 (1982), no. 2, 245-273. 
 
\item{[Ar]} Araki, H.: Relative entropy of states of von-Neumann algebra,
Publ. RIMS, Kyoto Univ., 11, pp. 809-833, 1976.

\item{[ChE]} Choi, Man-Duen; Effros, Edward G.: Nuclear $C^*$-Algebras and the Approximation Property, American Journal of Mathematics, Vol. 100, No.1, pp. 61-79 (1978).

\item{[CS]} Connes, A.; St\o rmer, E.: Entropy of automorphisms of II$_1$ -von Neumann algebras, Acta Math. 134 (1975), 289-306.

\item{[CT]} Connes, A., Takesaki, M.: The flow of weights on factors of type III,
T\'ohoku Math. Journ. 29 (1977), 473-575. 

\item{[Co]} Connes, A.: Entropie de Kolmogoroff-Sinai et mécanique statistique quantique.
C. R. Acad. Sci. Paris Sér. I Math., 301, 1–6 (1985).

\item{[CNT]} Connes, A., Narnhofer, H. and Thirring, W.: Dynamical entropy of $C^*$-
algebras and von Neumann algebras, Commun. Math. Phys. 112 (1987), 691-719.

\item{[CFS]} Cornfeld, I. P. Fomin, S. V.; Sinaĭ, Ya. G. Ergodic theory. Translated from the Russian by A. B. Sosinskiĭ. Grundlehren der Mathematischen Wissenschaften [Fundamental Principles of Mathematical Sciences], 245. Springer-Verlag, New York, 1982.

\item{[Gl]} Glimm, James G.: On a certain class of operator algebras. Trans. Amer. Math. Soc. 95, 318-340 (1960).

\item{GoB]} Golodets, V. Ya., Boyko, M. S.: Dynamical entropy for a class of algebraic
origin automorphisms. Mat. Fiz. Anal. Geom., 8, 385–391 (2001).

\item {[Ka]} Kadison, Richard V.: A generalized Schwarz inequality and algebraic invariants for operator algebras,  Ann. of Math. (2)  56, 494-503 (1952). 

\item {[Kak]} Kakutani, S.: On Equivalence of Infinite Product measures, Ann. of Math (2). 49, 214-226, (1948), 
 
\item {[Mo1]} Mohari, A.: Pure inductive limit state and Kolmogorov property. II 
Journal of Operator Theory. vol 72, issue 2, 387-404 (2014).   
    
\item {[Mo2]} Mohari, A.: Translation invariant pure state on $\otimes_{k \in \IZ}\!M^{(k)}_d(\IC)$ and Haag duality, Complex Anal. Oper. Theory 8 (2014), no. 3, 745-789.

\item{[Mo3]} Mohari, A.: Translation invariant pure state on $\clb=\otimes_{k \in \IZ}\!M^{(k)}_d(\IC)$ and its split property, J. Math. Phys. 56, 061701 (2015).

\item{[Mo4]} Mohari, A.:  A mean ergodic theorem of an amenable group action  Infinite Dimensional Analysis, Quantum Probability and Related Topics Vol. 17, No. 01, 1450003 (2014).

\item{[Mo5]} Mohari, A.: Spontaneous $SU_2(C)$ symmetry breaking in the ground states of quantum spin chain. J. Math. Phys. 59 (2018), no. 11, 111701, 36 pp.

\item {[Mu]}  Murphy, G. J.: C$^*$ algebras and Operator theory. Academic press, San Diego 1990. 

\item {[NS]} Neshveyev, S.: St\o rmer, E. : Dynamical entropy in operator algebras. Springer-Verlag, Berlin, 2006. 

\item {[Or1]} Ornstein, D. S.: Two Bernoulli shifts with infinite entropy are isomorphic. Advances in Math. 5 1970 339-348 (1970).  

\item{[Or2]} Ornstein, D. S.: A K-automorphism with no square root and Pinsker's conjecture, 
Advances in Math. 10, 89-102. (1973).

\item {[OP]} Ohya, M., Petz, D.: Quantum entropy and its use, Text and monograph in physics, Springer-Verlag 1995. 

\item {[Pau]} Paulsen, V.: Completely bounded maps and operator algebras, Cambridge Studies in Advance Mathematics 78, Cambridge University Press. 2002

\item {[Pa]} Parry, W.: Topics in Ergodic Theory, Cambridge University Press, 1981. 

\item{[Par]} Park, Y. M.: Dynamical entropy of generalized quantum Markov chains. Lett.
Math. Phys., 32, 63–74 (1994).

\item {[Pow]} Powers, R. T.: Representation of uniformly hyper-finite algebras and their associated von-Neumann rings, Ann. Math. 86 (1967), 138-171. 

\item{[Pe]} Petz, D.: Entropy of Markov states. Riv. Mat. Pura Appl., 14, 33–42 (1994).

\item {[Sak]} Sakai, S.: C$^*$-algebras and W$^*$-algebras, Springer 1971.  

\item {[Si]} Sinai, Ja.: On the concept of entropy for a dynamic system. (Russian) Dokl. Akad. Nauk SSSR 124 1959 768-771. 
 
\item {[Sim]} Simon, B.: The statistical mechanics of lattice gases, vol-1, Princeton series in physics (1993). 
 
\item {[SS]} Sinclair, Allan M.; Smith, Roger R.: Finite von Neumann algebras and masas. London Mathematical Society Lecture Note Series, 351. Cambridge University Press, Cambridge, 2008.

\item{[St\o 1]} St\o rmer, E. : Symmetric states of infinite tensor products of $C^*$- algebras. J. Functional Analysis, 3, 48-68 (1969)

\item{[St\o 2]} St\o rmer, E.: A survey of non-commutative dynamical entropy. Classification of nuclear $C^*$-algebras. Entropy in operator algebras, 147-198, Encyclopaedia Math. Sci., 126, Springer, Berlin, 2002.

\item{[Ta1]} Takesaki, M.: Conditional Expectations in von Neumann Algebras, J. Funct. Anal., 9, pp. 306-321 (1972)
 
\item{[Ta2]} Takesaki, M. : Theory of Operator algebras II, Springer, 2001.

\end{itemize}

\end{document}